\documentclass[12pt,a4paper]{article}
\usepackage{latexsym,amssymb,amsfonts,amsmath,amsthm,nccmath,enumitem}
\usepackage[hidelinks]{hyperref}
\usepackage[margin=2cm]{geometry}
\usepackage[usenames,dvipsnames]{xcolor}

\usepackage{bm}

\numberwithin{equation}{section}
\newtheorem{theorem}{Theorem}%[section]

\newtheorem{lemma}[theorem]{Lemma}

\theoremstyle{definition}
\newtheorem*{definition}{Definition}

\let\oldproofname=\proofname
\renewcommand{\proofname}{\rm\bf{\oldproofname}}

\def \D {{\mathcal{D}}}

\def \P {{\mathcal {P}}}

\def \a {\alpha}
\def \b {\beta}

\def \t {\tau}

\def \mod#1{{\:({\rm mod}\ #1)}}
\def \md#1{{\:({\rm mod}\ #1)}}
\def \deg {{\rm deg}}
\def \N {{\rm Nbd}}
\def \case#1 {\noindent {\bf Case #1.}\quad}
\def \efloor#1{{\lfloor #1 \rfloor_{\rm e}}}
\def \eceil#1{{\lceil #1 \rceil_{\rm e}}}
\def \no {\nu_{\rm o}}

\title{\bf Decomposing $K_{u+w}-K_u$ into cycles of various lengths}

\author{
Daniel Horsley and Rosalind A. Hoyte \\
School of Mathematical Sciences \\
Monash University \\
Vic 3800, Australia \\[0.1cm]
\texttt{danhorsley@gmail.com}, \texttt{rahoyte@outlook.com}
}

\date{ }

\begin{document}
\maketitle\thispagestyle{empty}
\def\baselinestretch{1.1}\small\normalsize
\sloppy

\begin{abstract}
We prove that the complete graph with a hole $K_{u+w}-K_u$ can be decomposed into cycles of arbitrary specified lengths provided that the obvious necessary conditions are satisfied, each cycle has length at most $\min(u,w)$, and the longest cycle is at most three times as long as the second longest. This generalises existing results on decomposing the complete graph with a hole into cycles of uniform length, and complements work on decomposing complete graphs, complete multigraphs, and complete multipartite graphs into cycles of arbitrary specified lengths.
\end{abstract}

\section{Introduction}\label{Section:Intro}
A \emph{decomposition} of a graph $G$ is a collection of subgraphs of $G$ whose edge sets form a partition of the edge set of $G$.
There is extensive literature, dating back to the 19th century \cite{Kirkman,Lucas}, concerning the existence of decompositions of a graph $G$ into cycles of specified lengths $m_1,\ldots,m_\t$. Much attention has focussed on decompositions into cycles of uniform length (when $m_1=\cdots=m_\t$), but more recently results on decompositions into cycles of mixed lengths have also been obtained.

Most notably, the general mixed-length problem has been completely solved for decompositions of complete graphs \cite{BryHorPet14}. This result was recently generalised to complete multigraphs \cite{BHMS11}. Partial results have also been obtained for decompositions of complete bipartite graphs \cite{ChoFuHua99, Horsley12, Sotteau81} and complete multipartite graphs \cite{BahSaj16}. Here, we add to this body of work by addressing the question of when a complete graph with a hole admits a decomposition into cycles of arbitrary specified lengths. For positive integers $u$ and $w$, the complete graph of order $u+w$ with a hole of size $u$, denoted $K_{u+w}-K_u$, is the graph obtained from a complete graph of order $u+w$ by removing the edges of a complete subgraph of order $u$. Our main result is as follows.

\begin{theorem}\label{Theorem:MainTheorem}
Let $u$ and $w$ be integers with $w\geq 10$, and let $m_1,\ldots,m_\tau$ be a nondecreasing list of integers such that $m_\tau\leq \min (u,w,3m_{\tau-1})$. There exists a decomposition of $K_{u+w}-K_u$ into cycles of lengths $m_1,\ldots,m_\tau$ if and only if
\begin{itemize}
	\item[(i)]
$u$ is odd and $w$ is even;
	\item[(ii)]
$m_1\geq 3$;
	\item[(iii)]
$m_1+\cdots+m_\tau=\binom{u+w}{2}-\binom{u}{2}$; and
	\item[(iv)]
there are at most $\binom{w}{2}$ odd entries in $m_1,\ldots,m_\tau$.
\end{itemize}
\end{theorem}

Theorem~\ref{Theorem:MainTheorem} is also an extension of work on decomposing the complete graph with a hole into cycles of uniform length. Study of this problem began in 1973 when Doyen and Wilson \cite{DoyWil73} investigated decompositions into $3$-cycles, and decompositions into $m$-cycles for various $m$ have been considered in \cite{BryHofRod96,BryRodSpi97,MenRos83}. The strongest results obtained to date are found in \cite{Horsley12} and \cite{DWcurrent}, where the problem is solved whenever the number of vertices outside the hole is not too small compared with the cycle length.

It is not difficult to see that conditions (i)-(iv) in Theorem \ref{Theorem:MainTheorem} are necessary for the existence of a decomposition of $K_{u+w}-K_u$ into cycles of lengths $m_1,\ldots,m_\tau$. We establish this in Lemma~\ref{Lemma:NecessaryConditions}. Note that if $m_1=m_2=\cdots=m_\tau$, then Lemma~\ref{Lemma:NecessaryConditions} specialises to \cite[Lemma 1.1]{BryRodSpi97}, which is the analogous result for cycles of uniform length.

\begin{lemma}\label{Lemma:NecessaryConditions}
Let $m_1,\ldots,m_\tau$ be a nondecreasing list of integers and let $u$ and $w$ be positive integers. If there exists a decomposition of $K_{u+w}-K_u$ into cycles of lengths $m_1,\ldots,m_\tau$ then
\begin{itemize}
	\item[(i)]
$u$ is odd and $w$ is even;
	\item[(ii)]
$m_1\geq 3$ and $m_\tau\leq \min(u+w,2w)$;
	\item[(iii)]
$m_1+\cdots+m_\tau=\binom{u+w}{2}-\binom{u}{2}$;
	\item[(iv)]
there are at most $\binom{w}{2}$ odd entries in $m_1,\ldots,m_\tau$; and
	\item[(v)]
$\tau \geq \frac{u+w-1}{2}$.
\end{itemize}
\end{lemma}

\begin{proof}
Suppose there exists a decomposition of $K_{u+w}-K_u$ into cycles of lengths $m_1,\ldots,m_\tau$.
Since the degree of each vertex must be even, we have $w\equiv 0\mod{2}$ and $u+w-1\equiv 0\mod{2}$ so (i) follows.
Clearly $m_1\geq 3$ and $m_\tau\leq u+w$. Also, every cycle has at least half of its vertices outside the hole, so $m_\t \leq 2w$ and thus (ii) follows.
Condition (iii) clearly holds.
Any odd cycle in $K_{u+w}-K_u$ must contain at least one edge that is not incident with a  vertex in the hole, thus (iv) follows.
Finally, a fixed vertex outside the hole must be in at least $\frac{u+w-1}{2}$ cycles, so (v) follows.
\end{proof}

The remainder of the paper is concerned with proving the existence of cycle decompositions of $K_{u+w}-K_u$. Our general approach is to first construct decompositions of $K_{u+w}-K_u$ that contain collections of short cycles and then `merge' these cycles together to construct longer cycles. This method is similar to that used in \cite{DWcurrent}. Lemma~\ref{Lemma:GeneralJoining} below is the key result that allows us to merge two cycle lengths.

\section{Notation and preliminary results}

We now introduce some definitions and notation that we will require throughout the paper.

A \emph{packing} of a graph $G$ is a decomposition of some subgraph $H$ of $G$. The \emph{leave} of the packing is the graph obtained by removing the edges of $H$ from $G$. We define the \emph{reduced leave} of a packing as the graph obtained from its leave by deleting any isolated vertices. If the leave contains no edges then the reduced leave is a trivial graph with no vertices or edges. For a list of positive integers $m_1,\ldots,m_\t$, an \emph{$(m_1,\ldots,m_\t)$-packing} or \emph{$(m_1,\ldots,m_\t)$-decomposition} of a graph $G$ is a packing or decomposition of $G$ with $\t$ cycles of lengths $m_1,\dots,m_\t$. We say that a graph is even if each of its vertices has even degree. Note that a graph is even if and only if it has some decomposition into cycles.

For a set $V$, let $K_V$ denote the complete graph with vertex set $V$. For disjoint sets $U$ and $W$, let $K_{U,W}$ denote the complete bipartite graph with parts $U$ and $W$. For graphs $G$ and $H$, let $G \cup H$ denote the graph with vertex set $V(G) \cup V(H)$ and edge set $E(G) \cup E(H)$. Let $G-H$ denote the graph with vertex set $V(G)$ and edge set $E(G) \setminus E(H)$.

The neighbourhood $\N_G(x)$ of a vertex $x$ in a graph $G$ is the set of vertices in $G$ that are adjacent to $x$ (not including $x$ itself). We say vertices $x$ and $y$ in a graph $G$ are \emph{twin in $G$} if $\N_G(x)\setminus\{y\} =\N_G(y)\setminus\{x\}$. Let $U$ and $W$ be disjoint sets and consider the graph $K_{U\cup W}-K_U$. Note that two vertices are twin in $K_{U\cup W}-K_U$ if and only if they are both in $U$ or both in $W$. We say an edge $xy$ of $K_{U\cup W}-K_U$ is a \emph{pure edge} if $x,y\in W$ and we say that it is a \emph{cross edge} if $x$ or $y\in U$. Note that an even subgraph $G$ of $K_{U\cup W}-K_U$ has an even number of cross edges and hence $|E(G)|$ is congruent to the number of pure edges in $G$ modulo $2$.

The $m$-cycle with vertices $x_0,x_1,\ldots,x_{m-1}$ and edges $x_ix_{i+1}$ for $i\in\{0,1,\ldots,m-1\}$ (with subscripts modulo $m$) is denoted by $(x_0,x_1,\ldots,x_{m-1})$ and the $n$-path with vertices $y_0,y_1,\ldots,y_n$ and edges $y_jy_{j+1}$ for $j\in \{0,1,\ldots,n-1\}$ is denoted by $[y_0,y_1,\ldots,y_n]$. We will say that $y_0$ and $y_n$ are the end vertices of this path. We allow trivial $0$-paths that consist of a single vertex and no edges.

Given a permutation $\pi$ of a set $V$, a subset $S$ of $V$ and a graph $G_0$ with $V(G_0)\subseteq V$, let $\pi(S)$ be the set $\{\pi(x):x \in S\}$ and $\pi(G_0)$ be the graph with vertex set $\{\pi(x):x\in V(G_0)\}$ and edge set $\{\pi(x)\pi(y):xy\in E(G_0)\}$.

\begin{definition}
Let $G$ be a graph and let $\mathcal{P} = \{G_1,\ldots,G_t\}$ be a packing of the graph $G$. We say that another packing $\mathcal{P}'$ of $G$ is a \emph{repacking} of $\mathcal{P}$ if $\mathcal{P}' = \{G'_1,\ldots,G'_t\}$ where for each $i \in \{1,\ldots,t\}$ there is a permutation $\pi_i$ of $V(G)$ such that $\pi_i(G_i)=G'_i$ and $x$ and $\pi_i(x)$ are twin for each $x \in V(G)$.
\end{definition}

If $\P$ is a packing of $K_{u+w}-K_u$ with leave $L$ and $\P'$ is a repacking of $\P$ with leave $L'$, then $L$ and $L'$ have the same number of pure and cross edges. This fact will be used frequently through the remainder of the paper.

\begin{lemma}\label{Lemma:GeneralJoining}
Let $u \geq 5$ and $w \geq 2$ be integers such that $u$ is odd and $w$ is even, and let $M$ be a list of integers. Suppose there exists an $(M)$-packing $\P$ of $K_{u+w}-K_u$ with a reduced leave that has exactly $\mu$ pure edges, where $\mu\in\{0,1,2\}$, and has a decomposition into an $h$-cycle, an $m_1$-cycle and an $m_2$-cycle where $h$ is odd if $\mu=2$. If $m_1+m_2\leq 3h$ and $m_1+m_2+h\leq \min (2u+3, 2w+1, u+w)$, then there exists a repacking of $\P$ whose reduced leave has a decomposition into an $h$-cycle and an $(m_1+m_2)$-cycle each containing at most one pure edge.
\end{lemma}

The following lemma from \cite{DWcurrent} is a crucial technique in proving our results. This `switching' method was first applied to packings of the complete graph \cite{BryHor09, BryHor10, BryHorMae05}, and has since been generalised to other graphs \cite{BHMS11,Horsley12,DWcurrent}.

\begin{lemma}[{\cite[Lemma 2]{DWcurrent}}]\label{Lemma:CycleSwitch}
Let $u$ and $w$ be positive integers with $u$ odd and $w$ even, and let $M$ be a list of integers. Let $\mathcal{P}$ be an $(M)$-packing of $K_{u+w}-K_u$ with leave $L$, let $\a$ and $\b$ be twin
vertices in $K_{u+w}-K_u$, and let $\pi$ be the transposition $(\a\b)$. Then there exists a partition of the set $(\N_L(\a)\cup \N_L(\b)) \setminus
((\N_L(\a)\cap \N_L(\b)) \cup \{\a,\b\})$ into pairs such that, for each pair
$\{x,y\}$ of the partition, there exists an $(M)$-packing $\mathcal{P}'$ of $K_{u+w}-K_u$ whose leave
$L'$ differs from $L$ only in that $\a x$, $\a y$, $\b x$ and $\b y$ are edges in
$L'$ if and only if they are not edges in $L$. Furthermore, if $\mathcal{P} = \{C_1,C_2,\ldots,C_t\}$, then $\mathcal{P}' = \{C'_1,C'_2,\ldots,C'_t\}$, where, for each $i \in \{1,\ldots,t\}$, $C'_i$ is a cycle of the same length as $C_i$ such that
\begin{itemize}
    \item[(i)]
if neither $\alpha$ nor $\beta$ is in $V(C_i)$, then $C'_i=C_i$;
    \item[(ii)]
if exactly one of $\alpha$ and $\beta$ is in $V(C_i)$, then either
$C'_i=C_i$ or $C'_i=\pi(C_i)$; and
    \item[(iii)]
if both $\alpha$ and $\beta$ are in $V(C_i)$, then $C'_i \in
\{C_i,\pi(C_i),\pi(P_i)\cup P^{\dag}_i,P_i \cup \pi(P^{\dag}_i)\}$, where $P_i$ and
$P^{\dag}_i$ are the two paths in $C_i$ which have end vertices $\alpha$ and $\beta$.
\end{itemize}
\end{lemma}

Note that $\mathcal{P}'$ is a repacking of $\mathcal{P}$ in the above. When we apply Lemma~\ref{Lemma:CycleSwitch} we say that we are performing the \emph{$(\a,\b)$-switch with origin $x$ and terminus $y$} (equivalently, with origin $y$ and terminus $x$).

\section{Merging two cycle lengths}\label{Section:JoiningLemma}

The aim of this section is to prove Lemma~\ref{Lemma:GeneralJoining}. The work required to prove the $\mu=0$ case of this lemma has been done in \cite{Horsley12}.

\begin{lemma}\label{Lemma:BipartiteJoining}
Let $u$ and $w$ be positive integers such that $u$ is odd and $w$ is even, and let $M$ be a list of integers. Suppose there exists an $(M)$-packing $\P$ of $K_{u+w}-K_u$ with a reduced leave that has no pure edges, and has a decomposition into an $h$-cycle, an $m$-cycle and an $m'$-cycle. If $m+m'\leq 3h$ and $m+m'+h\leq 2\min (u+1, w+1)$, then there exists a repacking of $\P$ whose reduced leave has a decomposition into an $h$-cycle and an $(m+m')$-cycle.
\end{lemma}

\begin{proof}
This lemma is very similar to Lemma 3.6 from \cite{Horsley12}, but it applies to packings of $K_{u+w}-K_u$ rather than to packings of $K_{u,w}$ (the leaves of the packings are still subgraphs of $K_{u,w}$). It can be proved exactly as per the proof of \cite[Lemma 3.6]{Horsley12}, except that Lemma~\ref{Lemma:CycleSwitch} is used in place of \cite[Lemma 2.1]{Horsley12}. Note that two vertices of $K_{u+w}-K_u$ are twin if and only if they are both in the hole or both outside it. Also note that $u \neq w$ because $u$ is odd and $w$ is even.
\end{proof}

So it remains to prove Lemma~\ref{Lemma:GeneralJoining} in the cases $\mu=1$ and $\mu=2$. To help with this task we first prove two useful lemmas and then introduce some more notation.

\begin{lemma}\label{Lemma:EqualiseOneStep}
Let $G$ be an even graph, let $y$ and $z$ be twin vertices in $G$, and let $\P$ be an $(M)$-packing of $G$ with (unreduced) leave $L$. If $\deg_L(y)>\deg_L(z)$, then there is an $(M)$-packing $\P'$ of $G$ with an (unreduced) leave $L'$ such that $\deg_{L'}(y)=\deg_L(y)-2$, $\deg_{L'}(z)=\deg_L(z)+2$, and $\deg_{L'}(x)=\deg_L(x)$ for each $x \in V(G) \setminus \{y,z\}$. Furthermore the number of nontrivial components in $L'$ is at most one greater than the number of nontrivial components in $L$.
\end{lemma}

\begin{proof}
Let $\P'$ be the repacking of $\P$ obtained by applying a $(y,z)$-switch whose origin and terminus are both adjacent to $y$ in $L$. Such a switch exists because $\deg_L(y)>\deg_L(z)$. Examine the differences between $L'$ and $L$.
\end{proof}

\begin{lemma}\label{Lemma:Isolates}
Let $U$ and $W$ be disjoint sets with $|U|$ odd and $|W|$ even, and suppose that $L$ is a subgraph of $K_{U\cup W}-K_U$ such that $L$ contains exactly $\mu$ pure edges, where $\mu\in\{1,2\}$, and each vertex of $L$ has positive even degree.

\begin{itemize}
	\item[(i)]
If $|E(L)| \leq 2(|U|+1)$ and $U$ contains a vertex of degree at least $4$ in $L$, then there is a vertex $x$ in $U$ such that $x \notin V(L)$.
	\item[(ii)]
If $|E(L)|\leq \min (2(|U|+2),2|W|+1)$ and $S$ is an element of $ \{U,W\}$ such that $S$ contains either at least two vertices of degree $4$ in $L$ or at least one vertex of degree at least $6$ in $L$, then there is a vertex $x$ in $S$ such that  $x \notin V(L)$.
	\item[(iii)]
If $|E(L)|\leq \min (2(|U|+2),2|W|+1,|U|+|W|)$ and $L$ contains either at least two vertices of degree $4$ or at least one vertex of degree at least $6$, then there are twin vertices $x$ and $y$ in $K_{U\cup W}-K_U$ such that $\deg_L(x) \geq 4$ and $y \notin V(L)$.
\end{itemize}
\end{lemma}
\begin{proof} Let $\ell=|E(L)|$ and note that $\ell\equiv \mu \mod{2}$. If $\mu =2$ then the result follows by \cite[Lemma 10]{DWcurrent}. So suppose that $\mu=1$. Then we have
\begin{align*}
  \medop\sum_{x \in V(L) \cap U}\deg_L(x) &= \ell-1, \mbox{ and} \\
  \medop\sum_{x \in V(L) \cap W}\deg_L(x) &= \ell+1.
\end{align*}

\noindent {\bf Proof of (i).}\quad
Suppose that $\ell\leq 2(|U|+1)$ and $U$ contains a vertex of degree at least $4$ in $L$. Then $\ell\leq 2|U|+1$ since $\ell$ is odd. If $U\subseteq V(L)$ then $\ell-1=\sum_{x \in V(L) \cap U}\deg_L(x) \geq 2|U|+2$ which contradicts $\ell\leq 2|U|+1$.

\noindent {\bf Proof of (ii).}\quad
Suppose that $\ell\leq \min (2(|U|+2),2|W|+1, |U|+|W|)$ and $S$  is an element of $ \{U,W\}$ such that $S$ contains either at least two vertices of degree 4 in $L$ or at least one vertex of degree at least 6 in $L$. Suppose for a contradiction that $S\subseteq V(L)$. Then we have $\sum_{x\in V(L)\cap S}\deg_L(x)\geq 2|S|+4$. So, if $S=U$, then $\ell-1\geq 2|U|+4$, contradicting $\ell\leq 2|U|+3$. If $S=W$, then $\ell+1\geq 2|W|+4$, contradicting $\ell\leq 2|W|+1$.

\noindent {\bf Proof of (iii).}\quad
Because we have proved (ii), it only remains to show that if $L$ contains two vertices of degree $4$, one in $U$ and one in $W$, and every other vertex of $L$ has degree $2$, then there are twin vertices $x$ and $y$ in $K_{U\cup W}- K_U$ such that $\deg_L(x)\geq 4$ and $y \notin V(L)$. Suppose otherwise. Then it must be the case that $V(L) = U\cup W$, $\ell-1=2|U|+2$ and $\ell+1=2|W|+2$. But then $\ell=2|U|+3$ and $\ell=2|W|+1$, so $|U\cup W|=2|U|+1$ and $\ell=|U|+|W|+2$, contradicting $\ell\leq |U|+|W|$.
\end{proof}

\begin{definition}
An $(a_1,a_2,\ldots,a_s)$-chain (or $s$-chain) is the edge-disjoint union of $s\geq 2$ cycles $A_1,A_2,\ldots, A_s$ such that
\begin{itemize}
\item $A_i$ is a cycle of length $a_i$ for $1\leq i\leq s$; and
\item for $1\leq i<j\leq s$, $|V(A_i)\cap V(A_j)|=1$ if $j=i+1$ and $|V(A_i)\cap V(A_j)|=0$ otherwise.
\end{itemize}
\end{definition}
We call $A_1$ and $A_s$ the \emph{end cycles} of the chain, and for $1<i<s$ we call $A_i$ an \emph{internal cycle} of the chain. A vertex which is in two cycles of the chain is said to be the \emph{link vertex} of those cycles. We denote a $2$-chain with cycles $A_1$ and $A_2$ by $A_1 \cdot A_2$.

\begin{definition}
An $(a_1,a_2,\ldots,a_s)$-ring (or $s$-ring) is the edge-disjoint union of $s\geq 2$ cycles $A_1,A_2,\ldots, A_s$ such that
\begin{itemize}
\item $A_i$ is a cycle of length $a_i$ for $1\leq i\leq s$;
\item for $s \geq 3$ and $1\leq i<j\leq s$, $|V(A_i)\cap V(A_j)|=1$ if $j=i+1$ or if $(i,j)=(1,s)$, and $|V(A_i)\cap V(A_j)|=0$ otherwise; and
\item if $s=2$ then $|V(A_1)\cap V(A_2)|=2$.
\end{itemize}
\end{definition}
We refer to the cycles $A_1,A_2,\ldots,A_s$ as the \emph{ring cycles} of the ring in order to distinguish them from the other cycles that can be found within the graph. A vertex which is in two ring cycles of the ring is said to be a \emph{link vertex} of those cycles.

\begin{definition}
For disjoint sets of vertices $U$ and $W$, an $s$-chain that is a subgraph of $K_{U\cup W}-K_U$ is  \emph{good} if either $s=2$ or the following conditions hold:
\begin{itemize}
\item an end cycle of the chain has its link vertex in $W$ and contains at least one pure edge; and
\item each internal cycle of the chain has one link vertex in $W$ and one link vertex in $U$.
\end{itemize}
\end{definition}

\begin{definition} For disjoint sets of vertices $U$ and $W$, an $s$-ring that is a subgraph of $K_{U\cup W}-K_U$ is \emph{good} if either $s=2$ or one of the following holds:
\begin{itemize}
	\item
$s \geq 4$ is even, and each of the ring cycles has one link vertex in $U$ and one link vertex in $W$; or
	\item
$s \geq 3$ is odd, one ring cycle has both link vertices in $W$ and contains a pure edge, and each other ring cycle has one link vertex in $U$ and one link vertex in $W$.
\end{itemize}
\end{definition}

Much of the work for the $\mu=2$ case of Lemma~\ref{Lemma:GeneralJoining} has also already been done in the form of the following lemma from \cite{DWcurrent}.

\begin{lemma}[{\cite[Lemma 14]{DWcurrent}}]\label{Lemma:DW1Degree4Vertex}
Let $U$ and $W$ be disjoint sets with $|U|$ odd and $|W|$ even, and let $M$ be a list of integers. Let $m$, $m'$, $k$ and $t$ be positive integers such that $m$ and $m'$ are odd, $m, m' \geq \max(k+t-1,3)$, $m+m'\leq\min(2|U|+4, 2|W|)$, and $m+m'\leq 2(|U|+1)$ if $3\in\{m,m'\}$. Suppose there exists an $(M)$-packing $\P$ of $K_{U\cup W}-K_U$ with a reduced leave $L$ of size $m+m'$ such that $L$ contains exactly two pure edges and $L$ has exactly $k$ components, $k-1$ of which are cycles and one of which is a good $t$-chain that, if $3\in\{m,m'\}$, is not a $2$-chain with link vertex in $U$. Then there exists a repacking of $\P$ whose reduced leave is the edge-disjoint union of an $m$-cycle and an $m'$-cycle.
\end{lemma}

In order to prove Lemma~\ref{Lemma:GeneralJoining} we will require Lemma \ref{Lemma:1Degree4Vertex} which is an analogue of Lemma \ref{Lemma:DW1Degree4Vertex} for packings whose leaves have one pure edge. This will be our goal in Subsection~\ref{Subsection:Chains}.

\subsection{Proof of Lemma \ref{Lemma:1Degree4Vertex}}\label{Subsection:Chains}

The proof of Lemma~\ref{Lemma:1Degree4Vertex} proceeds as follows. Lemmas~\ref{Lemma:DWFigureOfEight3} and \ref{Lemma:PathToCycles} are used in proving Lemma~\ref{Lemma:2ChainCycles}, which gives conditions under which we can repack to transform a $2$-chain leave into a union of two cycles of specified lengths. Lemma~\ref{Lemma:2ChainCycles} then acts as a base case and is used, along with Lemmas~\ref{Lemma:ReducePathLength} and \ref{Lemma:2Paths}, in an induction proof of Lemma~\ref{Lemma:2Cycles}. Lemma~\ref{Lemma:2Cycles} gives conditions under which we can repack to transform a good $s$-chain or $s$-ring leave into a union of two cycles of specified lengths. Finally Lemma~\ref{Lemma:1Degree4Vertex} is proved from Lemma~\ref{Lemma:2Cycles}.

\begin{lemma}[{\cite[Lemma 5]{DWcurrent}}]\label{Lemma:DWFigureOfEight3}
Let $G$ be a graph and let $M$ be a list of integers. Let $m$, $p$ and $q$ be positive integers with $m$ odd, $m\geq p$ and $p+q-m\geq 3$. Suppose there exists an $(M)$-packing $\mathcal{P}$ of $G$ whose reduced leave is the $(p,q)$-chain $(x_1,x_2,\ldots,x_{p-1},c) \cdot (c,y_1,y_2,\ldots,y_{q-1})$ and such that either
\begin{itemize}
    \item[(i)]
$p$ is odd, $x_1, y_3, y_5,\ldots, y_{m-p+1}$ are pairwise twin in $G$ and $y_2, y_4,\ldots, y_{m-p+2}$ are pairwise twin in $G$; or
    \item[(ii)]
$p$ is even, $x_1,x_3,\ldots, x_{p-3}$ are pairwise twin in $G$ and $y_{m-p+2},x_2,x_4,\ldots,x_{p-2}$ are pairwise twin in $G$.
\end{itemize}
Then there exists a repacking of $\mathcal{P}$ whose reduced leave is the edge-disjoint union of an $m$-cycle and a $(p+q-m)$-cycle.
\end{lemma}

\begin{lemma}\label{Lemma:PathToCycles}
Let $U$ and $W$ be disjoint sets with $|U|$ odd and $|W|$ even, and let $M$ be a list of integers.
Let $m$, $p$ and $q$ be  positive integers with $m$ odd and $m, p+q-m\geq 3$. Suppose there exists an $(M)$-packing $\P$ of $K_{U\cup W}-K_U$ whose reduced leave $L$ is a $(p,q)$-chain $(x_1,x_2,\ldots,x_{p-1},x_0)\cdot (x_0,y_1,y_2,\ldots,y_{q-1})$ such that $L$ contains exactly one pure edge, namely $x_rx_{r+1}$ for some $r \in \{0,\ldots,p-1\}$ (subscripts modulo $p$). If $p\leq m$, or if $p\geq m+2$ and $r\leq m-3$, then there exists a repacking of $\P$ whose reduced leave is the edge-disjoint union of an $m$-cycle and a $(p+q-m)$-cycle.
\end{lemma}

\begin{proof}
The proof relies on several applications of Lemma~\ref{Lemma:CycleSwitch}. We consider the case when $p\leq m$ and the case $p\geq m+2$ and $r\leq m-3$ separately.
Note that since the $p$-cycle in $L$ contains exactly one pure edge and the $q$-cycle contains no pure edges, then $p$ is odd and $q$ is even.

\noindent \textbf{Case 1.} First suppose that $p\leq m$. If $p=m$ then we are done so assume $p<m$. Without loss of generality, assume $x_0x_1$ is not a pure edge (otherwise relabel vertices in $L$). Then the result follows by Lemma~\ref{Lemma:DWFigureOfEight3}(i) because $[x_1,x_0,y_1,\ldots,y_{m-p+2}]$ is a path with no pure edges and hence $x_1,y_3,y_5,\ldots,y_{m-p+1}$ are pairwise twin and $y_2,y_4,\ldots,y_{m-p+2}$ are pairwise twin.

\noindent \textbf{Case 2.} Now assume that $p\geq m+2$ and $r\leq m-3$. Then by a simple induction it is sufficient to obtain either the required decomposition, or a $(p-2,q+2)$-chain $(x'_1,x'_2,\ldots,x'_{p-3},x'_0)\cdot (x'_0,y'_1,y'_2,\ldots,y'_{q-1+1})$ that contains exactly one pure edge $x'_rx'_{r+1}$ for some $r \in \{0,\ldots,m-3\}$.

Let $\P'$  be the repacking of $\P$ obtained by performing the $(y_1,x_{m-1})$-switch with origin $x_0$. Note that $\{y_1,x_{m-1}\}$ and $\{x_m,x_{m-2}\}$ are twin pairs in $K_{U\cup W}-K_U$ because $r\leq m-3$ and hence $[x_{m-2},x_{m-1},\ldots,x_{p-1},x_0,y_1]$ is a path with no pure edges. If the terminus of the switch is not $x_{m-2}$ then the reduced leave of $\P'$ has a decomposition into an $m$-cycle and a $(p+q-m)$-cycle and we are done. So assume that the terminus is $x_{m-2}$. Then the reduced leave of $\P'$ is the $(q+m-2,p-m+2)$-chain $(x_1,x_2,\ldots,x_{m-2},y_1,y_2,\ldots,y_{q-1},x_0)\cdot(x_0,x_{m-1},x_m,\ldots,x_{p-1})$.

Let $\P''$ be the repacking of $\P'$ obtained by performing the $(x_m,x_{m-2})$-switch with origin $x_{m-1}$. If the terminus of the switch is not $x_{m-3}$, then the reduced leave of $\P''$ has a decomposition into an $m$-cycle and a $(p+q-m)$-cycle and we are done. Otherwise, the terminus is $x_{m-3}$ and the reduced leave of $\P''$ is the $(p-2,q+2)$-chain $(x_1,\ldots,x_{m-3},x_m,x_{m+1},\ldots,x_{p-1},x_0)\cdot(x_0,x_{m-1},x_{m-2},y_1,y_2,\ldots,y_{q-1})$ where $x_rx_{r+1}$ is a pure edge.
\end{proof}

\begin{lemma}\label{Lemma:2ChainCycles}
Let $U$ and $W$ be disjoint sets with $|U|$ odd and $|W|$ even, and let $M$ be a list of integers.
Let $m$, $p$ and $q$ be  positive integers with $m$ odd and $m, p+q-m\geq 3$.
Suppose there exists an $(M)$-packing $\P$ of $K_{U\cup W}-K_U$ whose reduced leave $L$ is a $(p,q)$-chain such that $L$ contains exactly one pure edge and the link vertex of $L$ is in $W$ if $m=3$. Then there exists a repacking of $\P$ whose reduced leave is the edge-disjoint union of an $m$-cycle and a $(p+q-m)$-cycle.
\end{lemma}

\begin{proof}
Since $L$ contains exactly one pure edge, $p+q$ must be odd. Without loss of generality suppose $p$ is odd. Then the $p$-cycle in $L$ contains the pure edge and the $q$-cycle in $L$ contains no pure edges.

\noindent \textbf{Case 1.}
Suppose either that $p\leq m$ or that $p\geq m+2$ and $L$ can be labelled $(x_1,x_2,\ldots,x_{p-1},x_0)\cdot (x_0,y_1,y_2,\ldots,y_{q-1})$ such that $x_rx_{r+1}$ is the pure edge for some $r \in \{0,\ldots,m-3\}$. Then the result follows by Lemma~\ref{Lemma:PathToCycles}.

\noindent \textbf{Case 2.}
Suppose that $p\geq m+2$ and there is no such labelling. Let $L$ be labelled $(x_1,x_2,\ldots,x_{p-1},x_0)\cdot (x_0,y_1,y_2,\ldots,y_{q-1})$ such that $x_rx_{r+1}$ is the pure edge for some $r \in \{m-2,\ldots,p-1\}$ (subscripts modulo $p$). Then $r \geq 2$, using the fact that $x_0 \in W$ if $m = 3$. It is sufficient to show that there exists a repacking $\P'$ of $\P$ whose reduced leave is either a $(p,q)$-chain that can be labelled as $(x'_1,x'_2,\ldots,x'_{p-1},x'_0)\cdot (x'_0,y'_1,y'_2,\ldots,y'_{q-1})$ where the pure edge is $x'_{r-2}x'_{r-1}$, or a $(p-2,q+2)$-chain. By repeating this process we eventually obtain a repacking of $\P$ which satisfies the criteria of Case 1.

Let $\P'$ be the repacking of $\P$ obtained by performing the $(x_0,x_2)$-switch with origin $x_3$. Note that $x_0$ and $x_2$ are twin in $K_{U\cup W}-K_U$ because $r \geq 2$ and hence $[x_0,x_1,x_2]$ is a path with no pure edges. If the terminus of the switch is $x_{p-1}$, then the reduced leave of $\P'$ is the $(p,q)$-chain $(x_3,x_4,\ldots,x_{p-1},x_2,x_1,x_0)\cdot (x_0,y_1,y_2,\ldots,y_{q-1})$ and we are done. If the terminus of the switch is not $x_{p-1}$ then the reduced leave of $\P'$ is a $(p-2,q+2)$-chain.
\end{proof}

\begin{lemma}\label{Lemma:ReducePathLength}
Let $U$ and $W$ be disjoint sets with $|U|$ odd and $|W|$ even, and let $M$ be a list of integers. Let $p$ and $s$ be integers such that $p\geq 4$ and $s\geq 2$.
Suppose there exists an $(M)$-packing $\P$ of $K_{U\cup W}-K_U$ whose reduced leave $L$ is a good $s$-chain that contains exactly one pure edge and has a decomposition $\{P,L-P\}$ into two paths such that $P$ has length $p$ and each path has both end vertices in $W$. Suppose further that $P$ has a subpath $P_0=[x_0,\ldots,x_r]$ such that $2 \leq r \leq p-1$, $x_0$ is an end vertex of $P$, $P_0$ contains no pure edge, and $\deg_L(x_{r-1})=\deg_L(x_r)=2$.
Then there is a repacking of $\P$ whose reduced leave $L'$ is a good $s$-chain that has a decomposition $\{P',L'-P'\}$ into two paths such that $P'$ has length $p-2$, each path has both end vertices in $W$, and $P'$ contains a pure edge if and only if $P$ does.
\end{lemma}

\begin{proof}
Label the vertices in $V(P) \setminus V(P_0)$ so that $P=[x_0,\ldots,x_p]$. We prove the result by induction on the length of $P_0$. If $|E(P_0)|=2$, then $\{P',L-P'\}$ where $P'=[x_2,\ldots,x_p]$ is a decomposition of $L$ with the required properties. So we can assume that $|E(P_0)| \geq 3$. By induction we can assume that $P_0$ is the shortest subpath of $P$ satisfying the required conditions. Because $r\geq 3$, this implies that $\deg_L(x_{r-2})=4$ and $x_{r-2}$ is a link vertex of $L$.

The vertices $x_r$ and $x_{r-2}$ are twin in $K_{U\cup W}-K_U$ because $[x_{r-2},x_{r-1},x_r]$ is a path with no pure edges. Let $L'$ be the reduced leave of the repacking of $\mathcal{P}$ obtained by performing the $(x_r,x_{r-2})$-switch with origin $x_{r-3}$. Note that $L'$ is a good $s$-chain irrespective of the terminus of the switch. If the terminus of the switch is not $x_{r+1}$, then $\{P',L'-P'\}$ where $P'=[x_0,x_1,\ldots,x_{r-3},x_r,x_{r+1},\ldots,x_p]$ is a decomposition of $L'$ such that $P'$ has length $p-2$, each path has both end vertices in $W$, and $P'$ contains a pure edge if and only if $P$ does.
If the terminus of the switch is $x_{r+1}$, then $\{P',L'-P'\}$ where $P'=[x_0,x_1,\ldots,x_{r-3},x_r,x_{r-1},x_{r-2},x_{r+1},x_{r+2},\ldots,x_p]$ is a decomposition of $L'$ into two paths such that $P'$ has length $p$ and contains a pure edge if and only if $P$ does, and each path has both end vertices in $W$. Further $P'$ has the subpath $P'_0=[x_0,\ldots,x_{r-3},x_{r},x_{r-1}]$ and we know that $x_0$ is an end vertex of $P'$, $P'_0$ contains no pure edge, and $\deg_{L'}(x_{r})=\deg_{L'}(x_{r-1})=2$. Thus, because $|E(P'_0)|=r-1$, we are finished by our inductive hypothesis.
\end{proof}

\begin{lemma}\label{Lemma:2Paths}
Let $U$ and $W$ be disjoint sets with $|U|$ odd and $|W|$ even, and let $M$ be a list of integers.
Let $m$, $m'$ and $s$ be integers such that $m+m'$ is odd, $m, m' \geq \max(s,3)$ and $s\geq 2$. Suppose there exists an $(M)$-packing $\P$ of $K_{U\cup W}-K_U$ whose reduced leave is a good $s$-chain of size $m+m'$ that contains exactly one pure edge. Then there exists a repacking of $\P$ whose reduced leave is a good $s$-chain that has a decomposition into an $m$-path and an $m'$-path such that the end vertices of the paths are twin in $K_{U\cup W}-K_U$.
\end{lemma}

\begin{proof}
Let $L$ be the reduced leave of $\P$ and note that $|E(L)|=m+m'$. Because $L$ is good and contains exactly one pure edge, we can find some decomposition $\{H,L-H\}$ of $L$ into two paths such that $H$ has odd length and contains the pure edge, $L-H$ has even length, and each of the paths has both end vertices in $W$. Let $m^*\in\{m,m'\}$ and $P\in\{H,L-H\}$ such that $|E(P)|\geq m^*$ and $|E(P)|\equiv m^* \md{2}$ (such an $m^*$ and $P$ exist because $|E(L)|=m+m'$).
If $|E(P)|=m^*$ then we are done, so suppose $|E(P)|>m^*$. Let $p=|E(P)|$.

\noindent {\bf Case 1.} Each cycle of $L$ contains at most two edges of $P$. Then exactly $p-s$ cycles of the chain contain two edges of $P$ and the rest contain one edge of $P$. Because $L$ is good and both end vertices of $P$ are in $W$, if $C$ is a cycle of $L$ that contains two edges of $P$ then $C$ is an end cycle of $L$, the link vertex of $C$ is in $W$, and $C\cap P$ contains no pure edges. Thus, because $p > m^*\geq s$, it must be that $p=s+2$ and $m^*=s$. Then $\{P',L-P'\}$, where $P'$ is obtained from $P$ by removing the end vertices and the incident edges, is a decomposition of $L$ into an $m$-path and an $m'$-path such that both end vertices of each path are in $U$.

\noindent {\bf Case 2.} There exists a cycle $C$ in $L$ such that $C \cap P$ is a path of length at least 3. Let $P_0=[x_0,\ldots,x_r]$ be a subpath of $P$ such that $x_0$ is an end vertex of $P$, $P_0$ contains no pure edge, and $P_0$ contains exactly two edges in $C \cap P$.
If $C \cap P$ contains no pure edge then it is easy to see such a subpath exists. If $C \cap P$ contains the pure edge then (since $L$ is good) $C$ is an end cycle of $L$ whose link vertex is in $W$ and again there exists such a subpath.
So we can apply Lemma~\ref{Lemma:ReducePathLength} to obtain a repacking of $\P$ whose reduced leave $L'$ is a good $s$-chain that has a decomposition $\{P',L'-P'\}$ into two paths such that $P'$ has length $p-2$, $P'$ has a pure edge if and only if $P$ does, and both paths have end vertices in $W$. It is clear that by repeating this procedure we will eventually obtain a repacking of $\P$ whose reduced leave either has a decomposition into two paths which satisfies the criteria for Case 1 or has a decomposition into an $m$-path and an $m'$-path such that both end vertices of each path are in $W$.
\end{proof}

\begin{lemma}\label{Lemma:2Cycles}
Let $U$ and $W$ be disjoint sets with $|U|$ odd and $|W|$ even, and let $M$ be a list of integers.
Let $m$, $m'$ and $s$ be positive integers such that $s\geq 2$, $m+m'$ is odd, $m, m'\geq \max(s,3)$, $m+m'\leq\min(2|U|+3, 2|W|+1,|U|+|W|)$, and $m+m'\leq 2|U|+1$ if $3\in\{m,m'\}$. Suppose there exists an $(M)$-packing $\P$ of $K_{U\cup W}-K_U$ whose reduced leave has size $m+m'$, contains exactly one pure edge, and is either a good $s$-chain or a good $s$-ring that, if $3\in \{m,m'\}$, is not a $2$-chain with link vertex in $U$. Then there exists a repacking of $\P$ whose reduced leave is the edge-disjoint union of an $m$-cycle and an $m'$-cycle.
\end{lemma}

\begin{proof}
Let $L$ be the reduced leave of $\P$. We first show that the result holds for $s=2$. If $L$ is a $2$-chain then the result follows by Lemma~\ref{Lemma:2ChainCycles}. If $L$ is a $2$-ring then by our hypotheses and Lemma~\ref{Lemma:Isolates} there are twin vertices $x$ and $y$ in $K_{U\cup W}-K_U$ such that $\deg_L(x) \geq 4$, $y \notin V(L)$ and $x,y\in U$ if $3\in\{m,m'\}$ (if $3\in\{m,m'\}$ then apply Lemma~\ref{Lemma:Isolates}(i), otherwise apply Lemma~\ref{Lemma:Isolates}(iii)). Performing an $(x,y)$-switch results in a repacking of $\P$ whose reduced leave is a $2$-chain with link vertex in $W$ if $3\in\{m,m'\}$, and the result follows by Lemma~\ref{Lemma:2ChainCycles}.
So the result holds for $s=2$ and it is sufficient to show, for each integer $s' \geq 3$, that if the result holds for $s=s'-1$ then it holds for $s=s'$.

\noindent {\bf Case 1.} Suppose that $L$ is a good $s'$-chain.
By Lemma~\ref{Lemma:2Paths} we can obtain a repacking of $\P$ with a reduced leave whose only component is a good $s'$-chain with a decomposition into paths of length $m$ and $m'$ whose end vertices are twin. Let $[x_0,x_1,\ldots,x_{m}]$ be the path of length $m$ and let $\P'$ be the repacking of $\P$ obtained by performing the $(x_0,x_{m})$-switch with origin $x_1$.

If the terminus of the switch is not $x_{m-1}$, then the reduced leave of $\P'$ is the edge-disjoint union of an $m$-cycle and an $m'$-cycle and we are done. If the terminus of the switch is $x_{m-1}$, then the reduced leave of $\P'$ is a good $(s'-1)$-ring that contains exactly one pure edge and the result follows by our inductive hypothesis.

\noindent {\bf Case 2.} Suppose that $L$ is a good $s'$-ring.
Let $A$ be the ring cycle of $L$ that contains the pure edge in $L$ and note that if $s'$ is odd then both link vertices of $A$ are in $W$. Let $x$ and $y$ be twin vertices in $K_{U\cup W}-K_U$ such that $x$ is a link vertex in $A$, $x\in U$ if $s'$ is even, and $y \notin V(L)$. Such a vertex $y$ exists by Lemma~\ref{Lemma:Isolates}(ii) because $m+m'\leq\min(2|U|+3, 2|W|+1,|U|+|W|)$, $W$ contains two vertices of degree 4 in $L$ if $s'$ is odd, and $U$ contains two vertices of degree 4 in $L$ if $s'$ is even (for then $s' \geq 4$).
Let $\P'$ be the repacking of $\P$ obtained by performing an $(x,y)$-switch with origin in $A$. If the terminus of this switch is also in $A$, then the reduced leave of $\P'$ is a good $s'$-chain and we can proceed as in Case 1. Otherwise, the reduced leave of $\P'$ is a good $(s'-1)$-ring and the result follows by our inductive hypothesis.
\end{proof}

\begin{lemma}\label{Lemma:1Degree4Vertex}
Let $U$ and $W$ be disjoint sets with $|U|$ odd and $|W|$ even, and let $M$ be a list of integers.
Let $m$, $m'$, $k$ and $t$ be positive integers such that $m,m' \geq \max(k+t-1,3)$, $m+m'\leq\min(2|U|+3, 2|W|+1,|U|+|W|)$, and $m+m'\leq 2|U|+1$ if $3\in\{m,m'\}$. Suppose there exists an $(M)$-packing $\P$ of $K_{U\cup W}-K_U$ with a reduced leave $L$ of size $m+m'$ such that $L$ contains exactly one pure edge and $L$ has exactly $k$ components, $k-1$ of which are cycles and one of which is a good $t$-chain that, if $3\in\{m,m'\}$, is not a $2$-chain with link vertex in $U$. Then there exists a repacking of $\P$ whose reduced leave is the edge-disjoint union of an $m$-cycle and an $m'$-cycle.
\end{lemma}

\begin{proof}
First note that, since $L$ contains exactly one pure edge and $L$ has a decomposition into cycles, $m+m'$ is odd. Without loss of generality let $m$ be odd and $m'$ be even.

By Lemma~\ref{Lemma:2Cycles} it is sufficient to show that we can construct a repacking of $\P$ whose reduced leave is a good $s$-chain for some $s \in \{2,\ldots,k+t-1\}$ and is not a $2$-chain with link vertex in $U$ if $m=3$. If $k=1$, then we are finished, so we can assume $k \geq 2$. By induction on $k$, it suffices to show that there is a repacking of $\P$ with a reduced leave with exactly $k-1$ components, one of which is a good $t'$-chain for $t' \in \{t,t+1\}$ and the remainder of which are cycles.

Let $H$ be the component of $L$ which is a good $t$-chain, and let $C$ be a component of $L$ such that $C$ is a cycle and $C$ contains the pure edge if $H$ does not. Let $H_1$ and $H_t$ be the end cycles of $H$ where $H_1$ contains the pure edge if $H$ does and the link vertex of $H_1$ is in $W$ if $t\geq 3$.

\noindent {\bf Case 1.} Suppose that either $t\geq 3$ or it is the case that $t=2$, $H_1$ contains a pure edge and the link vertex of $H$ is in $W$. Then let $x$ and $y$ be vertices such that $x\in V(H_t)$, $x$ is not a link vertex of $H$, $y\in V(C)$, $x,y\in W$ if $t$ is odd, and $x,y\in U$ if $t$ is even. Let $\P'$ be a repacking of $\P$ obtained by performing an $(x,y)$-switch with origin in $H_t$. The reduced leave of $\P'$ has exactly $k-1$ components, $k-2$ of which are cycles and one of which is a good $t'$-chain, where $t'=t+1$ if the terminus of the switch is also in $H_t$ and $t'=t$ otherwise. So we are finished by our inductive hypothesis.

\noindent {\bf Case 2.} Suppose that $t=2$ and $H$ contains no pure edge. Then $C$ contains the pure edge.
Let $x_1$ and $x_2$ be vertices such that $x_1\in V(C) \cap W$, $x_2\in V(H_1) \cap W$, and $x_2$ is not the link vertex of $H$. Let $\P'$ be a repacking of $\P$ obtained by performing an $(x_1,x_2)$-switch with origin in $H_1$ and let $L'$ be the reduced leave of $\P'$. If the terminus of this switch is in $C$, then $L'$ has exactly $k-1$ components, $k-2$ of which are cycles and one of which is a $2$-chain. In this case we are finished by our inductive hypothesis. Otherwise the terminus of this switch is in $H_1$ and $L'$ has exactly $k-1$ components, $k-2$ of which are cycles and one of which is a $3$-chain $H'$ such that one end cycle of $H'$ contains a pure edge and has its link vertex in $W$. If $H'$ is good, then we are again finished by our inductive hypothesis. Otherwise, it must be that both link vertices of $H'$ are in $W$ and we proceed as follows.

Let $H'_1$ and $H'_3$ be the end cycles of $H'$ where $H'_1$ has the pure edge.
Let $y_1,y_2\in W$ be vertices such that $y_1$ is the link vertex in $H'_3$ and $y_2 \notin V(L')$ (note that $y_2$ exists by Lemma~\ref{Lemma:Isolates}(ii)). Let $\P''$ be a repacking of $\P'$ obtained by performing a $(y_1,y_2)$-switch with origin in $H'_3$ and let $L''$ be the reduced leave of $\P''$. If the terminus of this switch is not in $H'_3$, then $L''$ has exactly $k-1$ components, $k-2$ of which are cycles and one of which is a $2$-chain. In this case we are finished by our inductive hypothesis. Otherwise, the terminus of this switch is in $H'_3$ and $L''$ has exactly $k$ components, $k-1$ of which are cycles and one of which is a $2$-chain that contains a pure edge and has its link vertex in $W$. In this case we can proceed as we did in Case 1.

\noindent {\bf Case 3.} Suppose that $t=2$, $H_1$ contains the pure edge and the link vertex of $H$ is in $U$. Let $x$ be the link vertex of $H$ and let $y$ be a vertex in $V(C) \cap U$. Let $\mathcal{P}'$ be a repacking of $\P$ obtained by performing an $(x,y)$-switch with origin in $H_2$ and let $L'$ be the reduced leave of $\mathcal{P}'$. If the terminus of this switch is in $C$ or $H_1$, then $L'$ has exactly $k-1$ components, $k-2$ of which are cycles and one of which is a $2$-chain. In this case we are finished by our inductive hypothesis.
Otherwise the terminus of this switch is in $H_2$ and $L'$ has exactly $k$ components, $k-1$ of which are cycles and one of which is a $2$-chain that contains no pure edges. In this case we can proceed as we did in Case 2.
\end{proof}

\subsection{Proof of Lemma~\ref{Lemma:GeneralJoining}}\label{Subsection:Joining}

Here we use Lemmas~\ref{Lemma:BipartiteJoining}, \ref{Lemma:DW1Degree4Vertex} and \ref{Lemma:1Degree4Vertex} to prove Lemma~\ref{Lemma:GeneralJoining}. We first require two more simple results. Lemma~\ref{Lemma:MaxComponents} is an easy bound on the maximum number of components in the reduced leave of a packing, and Lemma~\ref{Lemma:PickApart} allows us to find a repacking whose reduced leave is a vertex-disjoint union of a single $2$-chain and a collection of cycles.

\begin{lemma}\label{Lemma:MaxComponents}
Let $U$ and $W$ be disjoint sets with $|U|$ odd and $|W|$ even. If $G$ is a subgraph of $K_{U\cup W}-K_U$ such that $G$ contains $\mu$ pure edges, $G$ has one vertex of degree $4$, and each other vertex of $G$ has degree $2$, then $G$ has at most $\left\lfloor\frac{1}{4}(|E(G)|+\mu)\right\rfloor-1$ components.
\end{lemma}

\begin{proof}
Because each vertex of $G$ has even degree, $G$ has a decomposition $\mathcal{D}$ into cycles. Since there are $\mu$ pure edges in $G$, at most $\mu$ cycles in $\mathcal{D}$ have length 3 and each other cycle in $\mathcal{D}$ has length at least $4$. Thus $|E(G)| \geq 4(|\mathcal{D}|-\mu)+3\mu$ which implies  $|\mathcal{D}| \leq \left\lfloor\frac{1}{4}(|E(G)|+\mu)\right\rfloor$. At least one component of $G$ contains a vertex of degree 4 and hence contains at least two cycles, and each component of $G$ contains at least one cycle. The result follows.
\end{proof}

We will use the following notation in Lemma \ref{Lemma:PickApart} and in the proof of Lemma~\ref{Lemma:GeneralJoining}. For an $(M)$-packing $\P$ of a graph $G$ we define $$d(\P)=\mfrac{1}{2}\medop\sum_{x\in V(L)} (\deg_L(x)-2),$$ where $L$ is the reduced leave of $\P$.

\begin{lemma}\label{Lemma:PickApart}
Let $U$ and $W$ be disjoint sets with $|U|$ odd and $|W|$ even, let $M$ be a list of integers, and let $\mu\in\{1,2\}$.
Suppose there exists an $(M)$-packing $\P$ of $K_{U\cup W}-K_U$ with a reduced leave $L$ such that $|E(L)|\leq \min(2(|U|+2), 2|W|+1,|U|+|W|)$, $L$ has $k$ components, $L$ has exactly $\mu$ pure edges, and $L$ has at least one vertex of degree at least $4$. Then there exists a repacking $\P'$ of $\P$ with a reduced leave $L'$ such that exactly one vertex $x'$ of $L'$ has degree $4$, every other vertex of $L'$ has degree $2$, and $L'$ has at most $k+d(\P)-1$ components. Furthermore if $|E(L)|\leq 2|U|+2$ and there is a vertex in $W$ with degree at least $4$ in $L$, then $x'$ is in $W$.
\end{lemma}

\begin{proof} The proof is by induction on $d(\P)$. Because $L$ has at least one vertex of degree at least $4$, $d(\P) \geq 1$. If $d(\P) = 1$, then we are finished immediately because one vertex of $L$ has degree $4$ and every other vertex has degree $2$. So suppose that $d(\P) \geq 2$ and hence that $L$ contains either at least two vertices of degree $4$ or at least one vertex of degree at least $6$.

Let $\P''$ be the repacking of $\P$ obtained by applying Lemma \ref{Lemma:EqualiseOneStep} with $y$ and $z$ chosen to be vertices in $U\cup W$ such that $\deg_L(y) \geq 4$, $z \notin V(L)$, $y,z \in U$ if such vertices exist in $U$, and $y,z \in W$ otherwise. These choices for $y$ and $z$ exist by Lemma~\ref{Lemma:Isolates}(iii), and by Lemma~\ref{Lemma:Isolates}(i) they will be in $U$ unless $|E(L)| > 2|U|+2$ or $\deg_L(x)=2$ for all $x\in V(L)\cap U$. Note that $d(\P'')=d(\P)-1$ and the reduced leave of $\P''$ has at most $k+1$ components. Thus we can complete the proof by applying our inductive hypothesis.
\end{proof}

\begin{proof}[{\bf Proof of Lemma~\ref{Lemma:GeneralJoining}.}]
Note first that $m_1+m_2+h \equiv \mu \mod{2}$. If $\mu=0$, then the result follows by Lemma~\ref{Lemma:BipartiteJoining}. So suppose $\mu\in\{1,2\}$.

Let $U$ and $W$ be disjoint sets of sizes $u$ and $w$ and let $\P$ be a packing of $K_{U \cup W}-K_U$ satisfying the hypotheses of the lemma. Let $L$ be the reduced leave of $\P$ and let $k$ be the number of components of $L$ (note that $k\leq 3$). Below we will sometimes wish to apply Lemma~\ref{Lemma:DW1Degree4Vertex} or \ref{Lemma:1Degree4Vertex} with $m=h$ and $m'=m_1+m_2$. Accordingly, we note that if $h=3$ then $m_1+m_2+h \leq 2u+\mu$ because $m_1+m_2\leq 3h$, $m_1+m_2+h \equiv \mu \mod{2}$ and $u\geq 5$. We also note that if $\mu=2$ then $m_1+m_2+h \leq 2w$ because $m_1+m_2+h \equiv \mu \mod{2}$.

Let $C_1$, $C_2$ and $H$ be edge-disjoint cycles in $L$ such that $V(H)=h$ and $|V(C_1)|+|V(C_2)|=m_1+m_2$ (we do not assume that $|V(C_1)|=m_1$ and $|V(C_2)|=m_2$).

\noindent \textbf{Case 1.}
Suppose that $k=3$. Then the components of $L$ are $C_1$, $C_2$ and $H$. Let $x$ and $y$ be vertices such that $x\in V(C_1)\cap W$ and $y\in V(C_2)\cap W$. By performing an $(x,y)$-switch we obtain a repacking of $\P$ whose reduced leave is either the edge-disjoint union of an $h$-cycle and an $(m_1+m_2)$-cycle or the vertex-disjoint union of an $h$-cycle and a $2$-chain of size $m_1+m_2$ with link vertex in $W$. In the former case we are finished and in the latter case we apply Lemma~\ref{Lemma:DW1Degree4Vertex} (if $\mu=2$) or Lemma~\ref{Lemma:1Degree4Vertex} (if $\mu=1$) with $m=h$ and $m'=m_1+m_2$.

\noindent \textbf{Case 2.}
Suppose that $k\in\{1,2\}$, that $(k,d(\P))\neq (1,m_1+m_2)$
and that $W$ contains a vertex of degree at least $4$ in $L$ if $h=3$.
Note that $L$ must have a vertex of degree at least $4$. Applying Lemma~\ref{Lemma:PickApart} to $\P$, we see that there is a repacking of $\P$ whose reduced leave $L'$ is the vertex-disjoint union of a $2$-chain and $k'-1$ cycles for some $k' \leq k+d(\P)-1$. Furthermore, the link vertex of the $2$-chain is in $W$ if $h=3$.
If we can show that $h,m_1+m_2\geq k'+1$, then we can complete the proof by applying Lemma~\ref{Lemma:DW1Degree4Vertex} or Lemma~\ref{Lemma:1Degree4Vertex} with $m=h$ and $m'=m_1+m_2$.

\noindent \textbf{Case 2a.}
Suppose further that $h\leq m_1+m_2$. Then it is sufficient to show that $h\geq k' +1$.
By Lemma~\ref{Lemma:MaxComponents}, $k' +1 \leq \lfloor\frac{m_1+m_2+h+\mu}{4}\rfloor$. Because $m_1+m_2\leq 3h$ and $\mu\leq 2$, we have $\lfloor\frac{m_1+m_2+h+\mu}{4}\rfloor\leq h$. So $h\geq k' +1$ as required.

\noindent \textbf{Case 2b.}
Suppose further that $h> m_1+m_2$. Then it is sufficient to show that $m_1+m_2\geq k' +1$.
If $k=2$ then $d(\P)\leq \max (m_1,m_2)$ and if $k=1$ then $d(\P)\leq m_1+m_2$.
So, because we have assumed that $(k,d(\P)) \neq (1,m_1+m_2)$, we have $k+d(\P)\leq m_1+m_2$. Thus $m_1+m_2\geq k+d(\P) \geq k' +1$ as required.

\noindent \textbf{Case 3.}
Suppose that $(k,d(\P))= (1,m_1+m_2)$. Then $W$ contains more than one vertex of degree at least $4$ in $L$. Also $L$ has no cut vertex because $L$ has $h$ vertices and contains an $h$-cycle.

Let $x$ and $y$ be twin vertices in $K_{U\cup W}-K_U$ such that $\deg_L(x)\geq 4$ and $y\notin V(L)$ (such vertices exist by Lemma~\ref{Lemma:Isolates}(iii)). Then let $\P^*$ be a repacking of $\P$ obtained by performing an $(x,y)$-switch and let $k^*$ be the number of components in the reduced leave of $\P^*$. Note that $k^*=1$ because $L$ has no cut vertex and $d(\P^*)=d(\P)-1=m_1+m_2-1$. Now we can proceed as we did in Case 2 (note that our argument in Case 2 did not depend upon $L$ being the edge-disjoint union of an $h$-cycle, an $m_1$-cycle and an $m_2$-cycle).

\noindent \textbf{Case 4.}
Suppose that $k\in\{1,2\}$, $h=3$ and that each vertex in $V(L)\cap W$ has degree $2$ in $L$.
Then $m_1+m_2\leq 9$ and it must be that $m_1=m_2=4$ if $\mu=1$ and $\{m_1,m_2\}\in\{\{3,4\},\{3,6\},\{5,4\}\}$ if $\mu=2$.

\noindent \textbf{Case 4a.}
Suppose further that $C_1$ and $C_2$ are vertex disjoint. Let $x \in V(C_1) \cap W$ and $y \in V(C_2) \cap W$, and let $\P'$ be the repacking of $\P \cup \{H\}$ obtained by performing an $(x,y)$-switch with origin in $C_1$. If the terminus of this switch is in $C_2$, then the reduced leave of $\P'$ is an $(m_1+m_2)$-cycle and we can remove from $\P'$ a $3$-cycle that contains exactly one pure edge to complete the proof. If the terminus of this switch is in $C_1$, then the reduced leave of $\P'$ is an $(m_1,m_2)$-chain with link vertex in $W$ and we can remove from $\P'$ a $3$-cycle that contains exactly one pure edge and then proceed as in Case 2.

\noindent \textbf{Case 4b.}
Suppose further that $C_1$ and $C_2$ share at least one vertex (in $U$). By applying Lemma \ref{Lemma:EqualiseOneStep} once or twice to $\P \cup \{H\}$, we can obtain a repacking $\P'$ of $\P \cup \{H\}$ whose reduced leave $L'$ is $2$-regular. Note that $L'$ is either an $(m_1+m_2)$-cycle or
the vertex-disjoint union of two cycles whose lengths add to $m_1+m_2$. In either case we remove from $\P'$ a $3$-cycle $H'$ that contains one exactly pure edge. In the former case we are finished immediately and in the latter case we can proceed as in Case 1, 2 or 4a, depending on $V(L')\cap V(H')$. (If $V(L')\cap V(H')=\emptyset$ then proceed as in Case 1, if $W$ contains a vertex of degree $4$ in $L'\cup H'$ then proceed as in Case 2, and otherwise proceed as in Case 4a.)
\end{proof}

We conclude this section with the following result which will be used in the proof of Lemma~\ref{Lemma:BaseDecomp_ManyCrossOdds_small_m} to obtain two cycles from a leave with a vertex of degree at least $4$.

\begin{lemma}\label{Lemma:1Deg4Vertex_rearrange}
Let $U$ and $W$ be disjoint sets with $|U|$ odd and $|W|$ even, let $M$ be a list of integers, and let $\mu\in\{1,2\}$. Let $m$ and $m'$ be positive integers such that $m$ is odd,  $m, m' \geq \max(\lfloor \frac{1}{4}(m+m')+\mu  \rfloor,3)$, $m+m'\leq\min(2|U|+4, 2|W|+1,|U|+|W|)$, and $m+m'\leq 2(|U|+1)$ if $3\in\{m,m'\}$.
Suppose there exists an $(M)$-packing $\P$ of $K_{U\cup W}-K_U$ with a reduced leave $L$ of size $m+m'$ such that $L$ has exactly $\mu$ pure edges, $L$ has at least one vertex of degree at least $4$ and, if $3\in\{m,m'\}$, there is a vertex of degree at least $4$ in $V(L)\cap W$.
Then there exists a repacking of $\P$ whose reduced leave is the edge-disjoint union of an $m$-cycle and an $m'$-cycle.
\end{lemma}

\begin{proof}
Note that $m+m' \equiv \mu \mod{2}$. The proof splits into two cases.

\noindent {\bf Case 1.} Suppose that $L$ has exactly one vertex $x$ of degree $4$ and every other vertex of $L$ has degree $2$.
Note that, by the hypotheses of the lemma, $x$ is in $V(L)\cap W$ if $3\in\{m,m'\}$.
Then $L$ is the vertex-disjoint union of a $2$-chain and $k-1$ cycles, where $k$ is the number of components in $L$. So the result follows by Lemma~\ref{Lemma:1Degree4Vertex} (if $\mu=1$) or Lemma~\ref{Lemma:DW1Degree4Vertex} (if $\mu=2$). (Note that $m,m'\geq k+1$ since $k\leq \lfloor \frac{1}{4}(m+m')+\mu  \rfloor-1$ by Lemma~\ref{Lemma:MaxComponents}.)

\noindent {\bf Case 2.} Suppose that $L$ has at least two vertices of degree at least $4$, or one vertex of degree at least $6$.
Let $\P'$ be the repacking of $\P$ obtained by applying Lemma~\ref{Lemma:PickApart}, and let $L'$ be the reduced leave of $\P'$. Then $L'$ has exactly one vertex of degree $4$, every other vertex of $L'$ has degree $2$, and there is a vertex of degree $4$ in $V(L)\cap W$ if $3\in\{m,m'\}$.
We can proceed as in Case 1.
\end{proof}

\section{Base Decompositions}\label{Section:BaseDecomps}
For a nonnegative integer $i$, let $(x^i)$ denote a list of $i$ entries all equal to $x$. For a list $X=(x_1,\ldots,x_n)$, let $\sum X=\sum_{i=1}^n x_i$. For a list $X$ and a sublist $Y$ of $X$, let $X\setminus Y$ be the list obtained from $X$ by removing the entries of $Y$. For a real number $x$ we denote the greatest even integer less than or equal to $x$ by $\efloor{x}$ and the least even integer greater than or equal to $x$ by $\eceil{x}$. For technical reasons in the remainder of the paper we shall consider a $0$-cycle to be a trivial graph with no vertices or edges. Because we can add any number of $0$-cycles to a packing without altering its leave, we shall not distinguish between packings that differ only in their number of $0$-cycles nor between lists that differ only in their number of 0s.

The aim of this section is to prove Lemmas~\ref{Lemma:BaseDecomp_FewCrossOdds}, \ref{Lemma:BaseDecomp_ManyCrossOdds_large_m} and \ref{Lemma:BaseDecomp_ManyCrossOdds_small_m}. These lemmas share a common form. Under various technical conditions, they guarantee the existence of an $(N,3^a,4^b,5^c,6^d,k)$-decomposition of $K_{u+w}-K_u$ that includes cycles with lengths $(3^a,4^{b},5^c,6^{d},k)$ that each contain at most one pure edge (where $k$ is even and perhaps 0). In order to prove Theorem~\ref{Theorem:MainTheorem} we will then take a base decomposition provided by one of these lemmas and repeatedly apply Lemma~\ref{Lemma:GeneralJoining} to produce a desired $(M)$-decomposition of $K_{u+w}-K_u$. Very roughly speaking, Lemma~\ref{Lemma:BaseDecomp_FewCrossOdds} will be used when $M$ has few odd entries, Lemma~\ref{Lemma:BaseDecomp_ManyCrossOdds_large_m} will be used when $M$ has many large entries, and \ref{Lemma:BaseDecomp_ManyCrossOdds_small_m} will be used when $M$ has few large entries.

In essence, Lemmas~ \ref{Lemma:BaseDecomp_ManyCrossOdds_large_m} and \ref{Lemma:BaseDecomp_ManyCrossOdds_small_m} are proved as follows. Consider $K_{u+w}-K_u$ as $K_{U,W} \cup K_W$, where $U$ and $W$ are disjoint sets of sizes $u$ and $w$. For some entry $m$ of $N$, we use the main result of \cite{BryHorPet14} to find an $(N \setminus (m))$-packing $\P$ of $K_W$ whose leave $L$ has size $a+c+m-t$, where $t=uw-(2a+4b+4c+6d)$, $t=0$ if $m=0$ and $t \in \{2,\ldots,m-2\}$ if $m > 0$. We then use various other results to find a $(3^a,4^{b},5^c,6^{d},k,m)$-decomposition of $K_{U,W} \cup L$ such that one cycle of length $m$ contains $m-t$ edges of $L$ and each other cycle contains one edge of $L$ if it has odd length and no edges of $L$ if it has even length. Lemma~\ref{Lemma:BaseDecomp_FewCrossOdds} is proved similarly except that we consider $K_{u+w}-K_u$ as $K_{U \setminus U_1,W} \cup K_{W \cup U_1}$, where $U$ and $W$ are disjoint sets of sizes $u$ and $w$ and $U_1 \subseteq U$ with $|U_1|=1$.

We will make use of three existing results on cycle decompositions of complete graphs and complete bipartite graphs. Theorem~\ref{Theorem:Alspach} is the main result of \cite{BryHorPet14}, and Theorems~\ref{Theorem:Bipartite} and \ref{Theorem:ChouFuHuang_no8} are special cases of the main results of \cite{Horsley12} and \cite{ChoFuHua99}.

\begin{theorem}[\cite{BryHorPet14}]\label{Theorem:Alspach}
Let $v$ be an integer and let $m_1,\ldots,m_\tau$ be a list of integers. There exists an $(m_1,\ldots,m_\tau)$-decomposition of $K_v$ if and only if $v$ is odd, $3\leq m_i\leq v$ for $i\in\{1,\ldots,\tau\}$, and $m_1+\cdots+m_\tau=\binom{v}{2}$.
There exists an $(m_1,\ldots,m_\tau)$-decomposition of $K_v-I$, where $I$ is a $1$-factor with vertex set $V(K_v)$, if and only if $v$ is even, $3\leq m_i\leq v$ for $i\in\{1,\ldots,\tau\}$, and $m_1+\cdots+m_\tau=\binom{v}{2}-\frac{v}{2}$.
\end{theorem}

\begin{theorem}[{\cite{Horsley12}}]\label{Theorem:Bipartite}
Let $p$ and $q$ be positive integers such that $p$ and $q$ are even and $p\leq q$, and let $m_1,\ldots, m_{\t}$ be even integers such that $4\leq m_1 \leq m_2 \leq\dots\leq m_{\t}$. If
\begin{itemize}
    \item[$(\rm{B}1)$]
$m_1+\dots+m_{\t}= pq$;
    \item[$(\rm{B}2)$]
$m_{\t}\leq 3m_{\t-1}$; and
    \item[$(\rm{B}3)$]
$m_{\t-1}+m_{\t}\leq 2p+2$ if $p<q$ and $m_{\t-1}+m_{\t}\leq 2p$ if $p=q$;
\end{itemize}
then there exists an $(m_1,\ldots, m_{\t})$-decomposition of $K_{p,q}$.
\end{theorem}

\begin{theorem}[{\cite{ChoFuHua99}}]\label{Theorem:ChouFuHuang_no8}
Let $b$, $d$, $p$ and $q$ be positive integers such that $p$ and $q$ are even. There exists a $(4^b,6^d)$-decomposition of $K_{p,q}$ if and only if $4b+6d=pq$ and either $p,q \geq 4$ or $d=0$.
\end{theorem}

The other results in this section are tools that we will use in the proofs of Lemmas~\ref{Lemma:BaseDecomp_FewCrossOdds}, \ref{Lemma:BaseDecomp_ManyCrossOdds_large_m} and \ref{Lemma:BaseDecomp_ManyCrossOdds_small_m}.

\subsection{Preliminary results}

Lemmas \ref{Lemma:Leave_PathDecomp} and \ref{Lemma:Leave_ManyPathDecomp} provide cycle packings of the complete bipartite graph whose leaves have decompositions into two paths of specified lengths. Lemmas \ref{Lemma:PathsAndCycleToDecomp}, \ref{Lemma:BipartiteAndOneCycle} and \ref{Lemma:3s5s} provide cycle packings of the union of the complete bipartite graph with one or more cycles. Lemma \ref{Lemma:4s_to_6s} allows us to decrease the number of $4$-cycles and increase the number of $6$-cycles in a packing.

\begin{lemma}\label{Lemma:Leave_PathDecomp}
Let $U'$ and $W$ be sets such that $|U'|$ and $|W|$ are even, let $m_1,\ldots,m_\tau$ be even integers such that $4 \leq m_1 \leq \cdots \leq m_\tau \leq 3m_{\tau-1}$ and $m_1+\cdots +m_\tau=|U'||W|$.
If $m_{\tau-1}+m_\tau\leq 2|U'|$ when $|U'| = |W|$ and  $m_{\tau-1}+m_\tau\leq 2\min(|U'|,|W|)+2$ otherwise, then
\begin{itemize}
\item[(i)] for all distinct $i,j\in\{1,\ldots,\tau\}$ there exists an $((m_1,\ldots,m_\tau)\setminus (m_i,m_j))$-packing of $K_{U',W}$ whose reduced leave has a decomposition into an $m_i$-path and an $m_j$-path whose end vertices are in $W$; and

\item[(ii)] for each $i\in\{1,\ldots,\tau\}$, there exists an $((m_1,\ldots,m_\tau)\setminus (m_i))$-packing of $K_{U',W}$ whose reduced leave has a decomposition into an $(m_i-2)$-path and a $2$-path whose end vertices are in $W$.
\end{itemize}
\end{lemma}

\begin{proof}
By Theorem~\ref{Theorem:Bipartite}, there exists an $(m_1,\ldots,m_\tau)$-decomposition $\D$ of $K_{U',W}$.

We can remove an $m_i$-cycle from $\D$ to obtain the packing required by (ii), so it remains to prove (i). Let $\P$ be a packing of $K_{U',W}$ obtained from $\D$ by removing an $m_i$-cycle and an $m_j$-cycle. Assume that $m_i\leq m_j$. Let $L$ be the reduced leave of $\P$. The proof divides into cases according to whether $L$ is connected.

\noindent {\bf Case 1.} Suppose that $L$ is connected. Then $L$ has at least one and at most $m_i$ vertices of degree $4$, and every other vertex of $L$ has degree $2$. Furthermore, if $L$ has exactly $m_i$ vertices of degree $4$ then $L$ has no cut vertex, since in this case $L$ has exactly $m_j$ vertices and contains an $m_j$-cycle.  So it can be seen that, by applying \cite[Lemma 3.4]{Horsley12} and then \cite[Lemma 3.5]{Horsley12} to this packing, we can obtain an $((m_1,\ldots,m_\tau)\setminus(m_i,m_j))$-packing of $K_{U',W}$ with a reduced leave $L'$ of size $m_i+m_j$ such that exactly one vertex of $L'$ has degree 4, every other vertex of $L'$ has degree 2, and
$L'$ has at most $m_i-1$ components. Then, by applying \cite[Lemma 3.2]{Horsley12} we can obtain an $((m_1,\ldots,m_\tau)\setminus(m_i,m_j))$-packing of $K_{U',W}$ whose reduced leave has a decomposition into an $m_i$-path and an $m_j$-path whose end vertices are in $W$ (note that $4 \leq m_i \leq m_j$).

\noindent {\bf Case 2.} Suppose that $L$ is not connected. Then $L$ must consist of two vertex-disjoint cycles. Let $x,y\in U'$ such that $x$ and $y$ are in distinct cycles of $L$. By applying an $(x,y)$-switch we obtain an $((m_1,\ldots,m_\tau)\setminus(m_i,m_j))$-packing of $K_{U',W}$ whose reduced leave $L'$ is either an $(m_i+m_j)$-cycle or an $(m_i,m_j)$-chain with link vertex in $U'$. In either case it is easy to see that $L'$ has a decomposition into an $m_i$-path and an $m_j$-path whose end vertices are in $W$.
\end{proof}

\begin{lemma}\label{Lemma:Leave_ManyPathDecomp}
Let $U'$ and $W$ be sets such that $|U'|$ and $|W|\geq 8$ are even, and let $\ell$ and $t$ be integers such that $\ell \in \{2,4,\ldots,12\}$ and $t \in \{6,8,\ldots,|W|-2\}$. Let $M$ be a  list of integers such that $m\in \{4,6\}$ for all entries $m$ in $M$ and $(\sum M)+k+\ell+t=|U'||W|$, where $k=\eceil{\frac{t+2}{3}}$ if $t \geq 12$ and
$k=0$ if $t \leq 10$. If $\max(k+2,\ell,8)+t \leq 2|U'|+2$ and $(\ell,t,|U'|,|W|) \neq (12,6,8,8)$, there exists an $(M,k)$-packing of $K_{U',W}$ whose reduced leave has a decomposition into an $\ell$-path and a $t$-path whose end vertices are in $W$.
\end{lemma}

\begin{proof}
If $\ell \in \{4,6,\ldots,12\}$, then apply Lemma~\ref{Lemma:Leave_PathDecomp}(i) taking $m_1,\ldots,m_\tau$ as the list $(M,k,\ell,t)$ reordered to be nondecreasing and $(m_i,m_j)=(\ell,t)$. If $\ell=2$, then apply Lemma~\ref{Lemma:Leave_PathDecomp}(ii) taking $m_1,\ldots,m_\tau$ as the list $(M,k,t+2)$ reordered to be nondecreasing and $m_i=t+2$.
The condition that $m_\tau\leq 3m_{\tau-1}$ holds by our definition of $k$. If $|U'| < |W|$, then $m_{\tau-1}+m_\tau\leq 2|U'|+2$ holds because $\max(k+2,\ell,8)+t \leq 2|U'|+2$. If $|W| < |U'|$, it is routine to show that $m_{\tau-1}+m_\tau\leq 2|W|+2$ holds by considering the cases $\ell=2$ and $\ell \in \{4,6,\ldots,12\}$ separately. Similarly, if $|W|=|U'|$ then $m_{\tau-1}+m_\tau\leq 2|W|$ holds since $(\ell,t,|U'|,|W|) \neq (12,6,8,8)$.
\end{proof}

\begin{lemma}\label{Lemma:PathsAndCycleToDecomp}
Let $U'$ and $W$ be sets with $|U'|,|W|$ even, and let $a$, $c$, $m$ and $t$ be nonnegative integers such that either
\begin{itemize}
    \item[(i)]
$(m,t)=(0,0)$ and $a+c \in \{3,\ldots,|W|\}$; or
    \item[(ii)]
$t \in \{2,4,\ldots,\efloor{m}-2\}$ and $a+c \in \{1,\ldots,|W|-m+\frac{t}{2}+1\}$.
\end{itemize}
Suppose there is an $(M)$-packing $\P$ of $K_{U',W}$ with a reduced leave $L$ such that $\deg_L(x)=2$ for each $x \in V(L) \cap W$ and $L$ is a union of edge-disjoint paths $P_0,\ldots,P_{a+c}$ such that
\begin{itemize}
    \item
$P_0$ has length $t$, $a$ of the paths $P_1,\ldots,P_{a+c}$ have length $2$, and the remaining $c$ have length $4$;
    \item
there are vertices $x_0,\ldots,x_{a+c} \in W$ such that, for $i \in \{1,\ldots,a+c\}$, the end vertices of $P_i$ are $x_{i-1}$ and $x_i$; and
    \item
the end vertices of $P_0$ are $x_0$ and $x_{a+c}$ (if $(m,t)=(0,0)$, $x_{a+c}=x_0$ and $P_0$ is trivial).
\end{itemize}
Let $C$ be an $(a+c+m-t)$-cycle such that $V(C) \subseteq W$ if $(m,t)=(0,0)$ and $V(C) \subseteq W \cup \{\alpha\}$ for some $\alpha \notin U' \cup W$ if $t>0$ (note that $a+c+m-t \in \{0\} \cup \{3,\ldots,|W|\}$). Then there exists an $(M,3^a,5^c,m)$-decomposition $\P'$ of $K_{U',W} \cup C$ that, if $m>0$, includes an $m$-cycle containing $m-t$ edges of $C$. Furthermore, if $|V(C)|+\frac{t-2}{2} \leq |W|-1$ and $c \geq 1$, then $\P'$ includes a $5$-cycle that has exactly one edge of $K_W$ and has a vertex in $W \setminus V(C)$.
\end{lemma}

\begin{proof}
It follows from (i) and (ii) that $a+c+m-t \in \{3,\ldots,|W|\}$. Let $C$ be such an $(a+c+m-t)$-cycle. By permuting vertices in $\P$, we can assume that $x_0x_1,x_1x_2,\ldots,x_{a+c-1}x_{a+c}$ are consecutive edges in $C$ (note that $x_{a+c}=x_0$ if $(m,t)=(0,0)$ and that $|E(C)|-(a+c) \geq 2$ otherwise) and that no internal vertices of $P_0$ are in $V(C)$. To see that we can do this note that, if $t>0$, then $\frac{t-2}{2}+(a+c+m-t) \leq |W|$ by (ii). Furthermore, if $|V(C)|+\frac{t-2}{2} \leq |W|-1$ and $c \geq 1$, we can ensure that some path of length $4$ in $\{P_1,\ldots,P_{a+c}\}$ has an internal vertex in $W \setminus V(C)$. Let $P'$ be the $(m-t)$-path induced by the edges of $C$ other than $x_0x_1,x_1x_2,\ldots,x_{a+c-1}x_{a+c}$ (if $(m,t)=(0,0)$, then $x_{a+c}=x_0$ and $P'$ is trivial). Then $\{P_0\cup P'\} \cup \{P_i\cup[x_{i-1},x_i]: i=1,\ldots,a+c\}$ is a $(3^{a},5^{c},m)$-decomposition of $C \cup L$ and the result follows.
\end{proof}

\begin{lemma}\label{Lemma:BipartiteAndOneCycle}
Let $U'$ and $W$ be sets such that $|U'| \geq 2$ and $|W| \geq 4$ are even. Let $a$, $b$, $c$, $d$ and $m$ be nonnegative integers such that
\begin{itemize}
    \item[(i)]
$d=0$ if $|U'| = 2$;
    \item[(ii)]
$2a+4b+4c+6d+t=|U'||W|$ where $t \in \{0,2,4\}$;
    \item[(iii)]
$2a+4c+t \leq 2|W|$; and
    \item[(iv)]
either
\begin{itemize}
    \item
$(m,t)=(0,0)$ and $a+c \in \{0\} \cup \{3,\ldots,|W|\}$; or
    \item
$t \in \{2,4\}$, $m \in \{t+2,\ldots,|W|\}$ and $a+c \in \{1,\ldots,|W|-m+\frac{t}{2}+1\}$.
\end{itemize}
\end{itemize}
Let $C$ be an $(a+c+m-t)$-cycle such that $V(C) \subseteq W$ if $t=0$ and $V(C) \subseteq W \cup \{\alpha\}$ for some $\alpha \notin U' \cup W$ if $t>0$ (note that $a+c+m-t \in \{0\} \cup \{3,\ldots,|W|\}$). Then there exists a $(3^a,4^{b},5^c,6^{d},m)$-decomposition of $K_{U',W} \cup C$ that, if $m>0$, includes an $m$-cycle containing $m-t$ edges of $C$.
\end{lemma}

\begin{proof}
If $(a,c,m,t)=(0,0,0,0)$, then the result follows from Theorem~\ref{Theorem:ChouFuHuang_no8}. Thus, using (iv), we may assume that $a+c+m-t \in \{3,\ldots,|W|\}$ and $2a+4c+t \geq 4$.

Suppose there exists a $(4^{b},6^{d})$-packing $\P'$ of $K_{U',W}$ with a reduced leave $L'$ such that $L'$ is connected and $\deg_{L'}(x)=2$ for all $x\in V(L')\cap W$. Then $|E(L')|=2a+4c+t$ by (ii). We claim that in this case $L'$ has a suitable decomposition into paths so that we can complete the proof by applying Lemma \ref{Lemma:PathsAndCycleToDecomp} (with $M=(4^b,6^d)$) to $\P'$. If $|E(L')|=4$, then $(a,c,t)=(1,0,2)$ and $L'$ is a 4-cycle and has a suitable decomposition into two $2$-paths. If $|E(L')|\geq 6$ then, because $L'$ is a connected even graph, it has a closed Eulerian trail. Because $L'$ is bipartite and $\deg_{L'}(x)=2$ for all $x\in V(L') \cap W$, any subtrail of this trail that begins at a vertex in $V(L')\cap W$ and has length 2 or 4 is a path. Thus, a suitable decomposition of $L'$ into $a$ $2$-paths, $c$ $4$-paths and a $t$-path exists. So it suffices to find a $(4^{b},6^{d})$-packing of $K_{U',W}$ with a reduced leave $L'$ such that $L'$ is connected and $\deg_{L'}(x)=2$ for all $x\in V(L')\cap W$.

By applying Theorem~\ref{Theorem:ChouFuHuang_no8} and removing cycles, we can obtain a $(4^{b},6^{d})$-packing $\P''$ of $K_{U',W}$. We can do this because $|U'||W|-4b-6d \in \{0\} \cup \{4,6,\ldots,|U'||W|\}$ by (ii) and (iv), and because $|U'||W|-4b-6d\equiv 0\mod{4}$ when $|U'|=2$ by (i). Let $L''$ be the reduced leave of $\P''$.

\noindent {\bf Case 1.}
Suppose $\deg_{L''}(x)=2$ for all $x\in V(L'')\cap W$. If $L''$ is connected then we are done. Otherwise, let $y_1,y_2\in V(L'')\cap U'$ such that $y_1$ is in a largest component of $L''$ and $y_2$ is in another component of $L''$. By performing a $(y_1,y_2)$-switch with origin adjacent to $y_2$ we obtain a repacking of $\P''$ whose reduced leave has a component larger than any component in $L''$. We can repeat this process until we obtain a repacking of $\P''$ whose reduced leave is connected.

\noindent {\bf Case 2.}
Suppose $\deg_{L''}(x) \geq 4$ for some $x\in V(L'')\cap W$. By repeatedly applying Lemma~\ref{Lemma:EqualiseOneStep} we can obtain a repacking of $\P''$ whose reduced leave has no vertices of degree greater than 2 in $W$ (note that $|E(L'')| \leq 2|W|$ by (iii)). Thus we can proceed as in Case 1 to complete the proof.
\end{proof}

The following is a method for packing $3$- and $5$-cycles into the complete graph with a hole, where each cycle has exactly one pure edge.

\begin{lemma}\label{Lemma:3s5s}
Let $W$ be a set of even size $w \geq 6$, and let $a$ and $c$ be nonnegative integers such that $a$ is even and $(a,c) \neq (0,0)$. Let $n$ and $b$ be the integers such that $a+2c=nw-2b$ and $0 \leq b \leq \frac{w-2}{2}$, and let $U'$ be a set such that $U'\cap W=\emptyset$ and $|U'|=2n$. Let $\ell_1,\ldots,\ell_{n}$ be integers such that $\ell_i \in \{\frac{w}{2},\ldots, w\}$ for $i \in \{2,\ldots,n\}$, $\ell_1 \in \{\frac{1}{2}(w-2b),\ldots, w-2b\} \setminus \{1,2\}$ and $\ell_1+\cdots+\ell_n=a+c$. Then, for any edge-disjoint cycles $C_1,\ldots,C_{n}$ in $K_{W}$ with lengths $\ell_1,\ldots,\ell_{n}$, there exists a $(3^{a},5^{c})$-packing of $K_{U',W}\cup C_1\cup\cdots\cup C_{n}$ whose reduced leave is a subgraph of $K_{U',W}$ isomorphic to $K_{2,2b}$. Furthermore, if $a+2c \not\equiv 2 \mod{w}$, and if $c \not\equiv 2 \mod{\frac{w}{2}}$ when $a=0$, then there do exist such integers $\ell_1,\ldots,\ell_{n}$.
\end{lemma}

\begin{proof}
Suppose first that we are given a list $\ell_1,\ldots,\ell_{n}$ satisfying our hypotheses. Let $U'=\{p_1,\ldots,p_{2n}\}$. By Lemma~\ref{Lemma:BipartiteAndOneCycle}, there is a $(3^{2\ell_i-w},5^{w-\ell_i})$-decomposition $\mathcal{D}_i$ of $K_{\{p_{2i-1},p_{2i}\},W} \cup C_i$ for $i \in \{2,\ldots,n\}$. Let $W_1$ be a set of size $w-2b$ such that $V(C_1) \subseteq W_1 \subseteq W$. Also by Lemma~\ref{Lemma:BipartiteAndOneCycle}, there is a $(3^{2\ell_1-w+2b},5^{w-2b-\ell_1})$-decomposition $\mathcal{D}_1$ of $K_{\{p_{1},p_{2}\},W_1} \cup C_1$. Using the facts that $nw-2b=a+2c$ and that $\ell_1+\cdots+\ell_n=a+c$, it can be seen that $\mathcal{D}_1 \cup \cdots \cup \mathcal{D}_n$ is a $(3^{a},5^{c})$-packing of $K_{U',W}\cup C_1\cup\cdots\cup C_{n}$. The reduced leave of this packing is $K_{\{p_1,p_2\},W \setminus W_1}$, which is isomorphic to $K_{2,2b}$.

Now suppose that $a+2c \not\equiv 2 \mod{w}$ and that $c \not\equiv 2 \mod{\frac{w}{2}}$ if $a=0$. Note that, because $a+2c=nw-2b$, the former implies that $w-2b \neq 2$ and the latter implies that $w-2b \neq 4$ if $a=0$. Let $\ell_1=w-2b$ if $a \geq w-2b$ and let $\ell_1=\frac{1}{2}(w-2b+a)$ if $a < w-2b$. Then $\ell_1 \in \{\frac{1}{2}(w-2b),\ldots, w-2b\} \setminus \{1,2\}$. To show that there exist integers $\ell_2,\ldots,\ell_n$ such that $\ell_i \in \{\frac{w}{2},\ldots,w\}$ for $i \in \{2,\ldots,n\}$ and $\ell_2+\cdots+\ell_n=a+c-\ell_1$, and hence to complete the proof, it suffices to show that $\frac{w}{2}(n-1) \leq a+c-\ell_1 \leq w(n-1)$. If $a \geq w-2b$, then $a+c-\ell_1=a+c-w+2b$ and
\[\tfrac{w}{2}(n-1) = \tfrac{1}{2}(a+2c-w+2b) \leq  a+c-w+2b \leq a+2c-w+2b = w(n-1),\]
where both equalities follow from $a+2c=nw-2b$, the first inequality follows because $a \geq w-2b$ and the second inequality follows because $c \geq 0$.
If $a < w-2b$, then
\[a+c-\ell_1=\tfrac{1}{2}(a+2c-w+2b)=\tfrac{w}{2}(n-1),\]
where the first equality follows because $\ell_1=\frac{1}{2}(w-2b+a)$ and the second equality follows because $a+2c=nw-2b$.
\end{proof}

\begin{lemma}\label{Lemma:4s_to_6s}
Let $\P=\{C_1,\ldots,C_r,X_1,\ldots,X_{3j}\}$ be an $(M,4^{3j})$-decomposition of an even graph $G$ where $X_1,\ldots,X_{3j}$ are $4$-cycles. If there is a set $S$ of four vertices in $G$ that are pairwise twin and pairwise nonadjacent such that $|V(X_i) \cap S|=2$ for $i \in \{1,\ldots,3j\}$, then there is an $(M,6^{2j})$-decomposition $\P'=\{C'_1,\ldots,C'_r,Y'_1,\ldots,Y'_{2j}\}$ of $G$ such that $Y'_1,\ldots,Y'_{2j}$ are $6$-cycles and, for each $i \in \{1,\ldots,r\}$, $V(C'_i)=\pi_i(V(C_i))$ for some permutation $\pi_i$ of $V(G)$ that fixes each vertex in $V(G) \setminus S$.
\end{lemma}

\begin{proof}
The result is trivial if $j=0$, so we may assume that $j \geq 1$. It suffices to show there is an $(M,4^{3j-3},6^{2})$-decomposition $\P^{\star}$ of $G$ such that $\P^{\star}$ contains two 6-cycles $Y^{\star}_1$ and $Y^{\star}_2$ and the packing $\P^{\star} \setminus \{Y^{\star}_1,Y^{\star}_2\}$ can be obtained from the packing $\P \setminus \{X_{3j-2},X_{3j-1},X_{3j}\}$ via a sequence of switches on vertices in $S$. Then by Lemma~\ref{Lemma:CycleSwitch}, $\P^{\star}=\{C^{\star}_1,\ldots,C^{\star}_r,X^{\star}_1,\ldots,X^{\star}_{3j-3},Y^{\star}_1,Y^{\star}_2\}$ where $X^{\star}_1,\ldots,X^{\star}_{3j-3}$ are $4$-cycles that each contain two vertices in $S$ and, for each $i \in \{1,\ldots,r\}$, $C^{\star}_i=\pi_i(C_i)$ for some permutation $\pi_i$ of $V(G)$ that fixes each vertex in $V(G) \setminus S$. So the procedure can be repeated $j-1$ further times to complete the proof.

Let $L$ be the reduced leave of $\P \setminus \{X_{3j-2},X_{3j-1},X_{3j}\}$ and note that $L$ is an even subgraph of $K_{S,V(G) \setminus S}$ with exactly 12 edges (we will not make use of the fact that $L$ has a decomposition into three $4$-cycles). If $L$ contains a $6$-cycle, then $L$ has a $(6,6)$-decomposition and adding these two $6$-cycles to $\P \setminus \{X_{3j-2},X_{3j-1},X_{3j}\}$ produces a decomposition with the required properties. So we may suppose that $L$ has no $6$-cycle but has a $(4,8)$-decomposition or a $(4,4,4)$-decomposition.

\noindent {\bf Case 1.} Suppose that no vertex in $S$ has degree $6$ in $L$ and that $L$ is connected. It is routine to check that $L$ contains a path $[x_0,\ldots,x_6]$ of length $6$ and a vertex $y \notin \{x_0,\ldots,x_6\}$ such that $x_0,x_6 \in S$ and $x_0y$ and $x_2y$ are edges in $L$. By performing the $(x_0,x_6)$-switch with origin $x_1$ we obtain a repacking of $\P \setminus \{X_{3j-2},X_{3j-1},X_{3j}\}$ whose reduced leave has a $(6,6)$-decomposition and we are finished as above. (Note that if the terminus of the switch is $x_5$ then the reduced leave contains the $6$-cycle $(x_0,y,x_2,\ldots,x_5)$, and otherwise it contains the $6$-cycle $(x_1,\ldots,x_6)$.)

\noindent {\bf Case 2.} Suppose that no vertex in $S$ has degree $6$ in $L$ and that $L$ is disconnected. Then $L$ is a vertex-disjoint union of a copy of $K_{2,4}$ and a $4$-cycle. Let $x,y \in S$ be vertices such that $\deg_L(x)=4$ and $\deg_L(y)=2$. By performing an $(x,y)$-switch whose origin is adjacent to $x$ we obtain a repacking of $\P \setminus \{X_{3j-2},X_{3j-1},X_{3j}\}$ that satisfies the conditions of Case 1 and we can proceed as we did in that case.

\noindent {\bf Case 3.} Suppose that a vertex in $S$ has degree at least $6$ in $L$. By repeatedly applying Lemma~\ref{Lemma:EqualiseOneStep}, we can obtain a repacking of $\P \setminus \{X_{3j-2},X_{3j-1},X_{3j}\}$ that satisfies the conditions of either Case 1 or Case 2 and can proceed as we did in those cases (note that each application of Lemma~\ref{Lemma:EqualiseOneStep} is simply a switch on vertices in $S$).
\end{proof}

\subsection{Lists with few odd entries}

\begin{lemma}\label{Lemma:BaseDecomp_FewCrossOdds}
Let $u\geq 5$ and $w\geq 8$ be integers such that $u$ is odd and $w$ is even. Let $N$ be a list of integers and let $a$, $b$, $c$ and $d$ be nonnegative integers such that the following hold.
\begin{itemize}
	\item[(i)]
$(\sum N)-t +a+c =\binom{w+1}{2}$, where $t \in \{0,2,\ldots,w-2\}$;
	\item[(ii)]
$2a+4b+4c+6d+k+t=(u-1)w$, where $k=\eceil{\frac{t+2}{3}}$ if $t\geq 12$ and $k=0$ otherwise;
    \item[(iii)]
$3\leq \ell\leq \min (u, w)$ for each entry $\ell$ in $N$, and $d=0$ if $u=5$;
	\item[(iv)]
if $t>0$, there is some entry $m$ in $N$ such that $m\geq t+2$; and
    \item[(v)] either
\begin{itemize}
    \item
$a+c\geq 6$, $a+2c\leq w$, and $b \geq 1$; or
	\item
$a+c=3$, $(m,t) \neq (w,2)$, and $(a,c,t,u,w)\neq (0,3,6,9,8)$.
\end{itemize}
\end{itemize}
Then there exists an $(N,3^a,4^b,5^c,6^d,k)$-decomposition of $K_{u+w}-K_u$ that includes cycles with lengths $(3^a,4^{b},5^c,6^{d},k)$ that each contain at most one pure edge.
\end{lemma}

\begin{proof}
Let $U$ and $W$ be disjoint sets of sizes $u$ and $w$ and let $U_1\subseteq U$ with $|U_1|=1$. Observe that $K_{U\cup W}-K_U=K_{U \setminus U_1,W} \cup K_{W \cup U_1}$. Let $m=0$ if $t=0$. We first choose integers $a_2$, $a_3$, $c_2$ and $c_3$. Let $(a_3,c_3)$ be the first pair from the appropriate row below such that $a_3\leq a$ and $c_3\leq c$, and let $a_2=a-a_3$ and $c_2=c-c_3$. It is routine to check using (v) that some pair in the appropriate row will always satisfy these conditions.
\begin{center}
\begin{tabular}{l|l}
  case & $(a_3,c_3)$ \\ \hline
  $a+c=3$ & $(a,c)$ \\
  $a+c \geq 6$, $t>0$, $a$ even & $(0,1), (2,0)$ \\
  $a+c \geq 6$, $t>0$, $a$ odd & $(1,0)$ \\
  $a+c \geq 6$, $t=0$, $a$ even & $(0,0)$ \\
  $a+c \geq 6$, $t=0$, $a$ odd & $(3,0), (1,2)$ \\
\end{tabular}
\end{center}
This choice ensures that $a_2$, $a_3$, $c_2$ and $c_3$ are nonnegative integers such that $a_2$ is even, $a_2+a_3=a$, $c_2+c_3=c$, $2a_3+4c_3 \leq 12$, $a_3+c_3 \in \{1,2,3\}$ if $t>0$, and $a_3+c_3 \in \{0,3\}$ if $t=0$.

We now construct packings $\P_1,\P_2,\P_3$ as follows (we justify that these packings exist later).
\begin{itemize}
	\item $\P_1$ is an $(N\setminus (m))$-packing of $K_{W\cup U_1}$. The reduced leave of $\P_1$ is $C^*\cup  C^\dag$, where $C^*$ is an $(a_2+c_2)$-cycle such that $U_1 \nsubseteq V(C^*)$ and $C^\dag$ is an $(m-t+a_3+c_3)$-cycle such that $U_1 \nsubseteq V(C^{\dag})$ if $t=0$.

Let $U_2\subseteq U\setminus U_1$ with $|U_2|=0$ if $(a_2,c_2)=(0,0)$ and $|U_2|=2$ otherwise. Let $b_2=\frac{1}{4}(|U_2|w-2a_2-4c_2)$. By (v) and the choice of $a_2$ and $c_2$, we have  $b_2 \in \{0,\ldots,\frac{w-4}{2}\}$.
	\item
If $|U_2|=2$, then $\P_2$ is a $(3^{a_2},4^{b_2},5^{c_2})$-decomposition of $K_{U_2,W}\cup C^*$ and, if $b_2>0$, the union of the $4$-cycles in $\P_2$ is a copy of $K_{2,2b_2}$. If $|U_2|=0$, then $\P_2=\emptyset$.
	\item
$\P_3$ is a $(3^{a_3},4^{b-b_2+3j},5^{c_3},6^{d-2j},k,m)$-decomposition of $K_{U_3,W}\cup C^\dag$, where $U_3=U\setminus (U_1\cup U_2)$ and
$$
j=\left\{\begin{array}{ll}
0 & \text{ if } b\geq b_2,\\
\lceil\frac{1}{3}(b_2-b)\rceil & \text{ otherwise.}
\end{array}
\right.$$
Furthermore, if $m>0$, there is an $m$-cycle in $\P_3$  that contains $m-t$ edges of $C^\dag$.\end{itemize}

Note that the union $\P'=\P_1 \cup \P_2 \cup \P_3$ will be an $(N,3^a,4^{b+3j},5^c,6^{d-2j},k)$-decomposition of $K_{U\cup W}-K_U$ and will contain cycles with lengths $(3^a,4^{b+3j},5^c,6^{d-2j},k)$ that each contain at most one pure edge (by the definition of $\P_3$ and $C^\dag$, $\P_3$ has cycles with lengths $(3^{a_3},4^{b-b_2+3j},5^{c_3},6^{d-2j},k)$ that each contain at most one pure edge).
If $b \geq b_2$, then $j=0$ and this will complete the proof. Otherwise $b_2 > b$ and we will be able to apply Lemma~\ref{Lemma:4s_to_6s} to $\P'$ to obtain a decomposition with the required properties provided we can find $3j$ $4$-cycles in $\P'$ that meet the hypotheses of Lemma~\ref{Lemma:4s_to_6s}. If $3j=b_2-b$, we will be able to use $3j$ $4$-cycles from $\P_2$. If $3j \in \{b_2-b+1,b_2-b+2\}$, then $b_2 \geq 3j+b-2 \geq 3j-1$ (since $b\geq 1$) and we will be able to use $3j-1$ $4$-cycles from $\P_2$ and any one $4$-cycle from $\P_3$. (It must be the case that $b \geq 1$ by (v) because $b_2 > b$ implies $(a_2,c_2) \neq (0,0)$ and $a+c \geq 6$ by our choices of $b_2$, $a_2$ and $c_2$.)

So it remains to establish the existence of the packings
$\P_1,\P_2,\P_3$. In what follows we will often use the facts that $w \geq 8$ and that either $(m,t)=(0,0)$ or $t < m \leq w$ (the latter follows from (iv)).

\noindent {\bf Proof that $\P_1$ exists.}
First observe that $a_2+c_2\in\{0\}\cup\{3,\ldots,w\}$ and $m-t+a_3+c_3\in\{0\} \cup \{3,\ldots,w+1\}$ by (iii), (iv), (v), our choice of $a_3$ and $c_3$ and the definition of $m$. Then, by Theorem~\ref{Theorem:Alspach}, a packing with the required properties exists by (iii) and because
\begin{align*}
{}&\textstyle\sum (N \setminus (m)) + |E(C^{\dag})| + |E(C^*)|  \\
={}& (\tbinom{w+1}{2}+t-(a+c)-m) + (m-t+a_3+c_3) + (a_2+c_2) \\
={}& \tbinom{w+1}{2}
\end{align*}
where the first equality follows by (i).
If $t=0$, then $|V(C^*)|+|V(C^{\dag})| = a+c \leq w$ by (v) and we can permute the vertices of this packing so that $U_1 \nsubseteq V(C^*) \cup V(C^\dag)$. If $t>0$, then $|V(C^*)| = a_2+c_2 \leq w$ by (v) and we can permute the vertices of this packing so that $U_1 \nsubseteq V(C^*)$.

\noindent {\bf Proof that $\P_2$ exists.} This is trivial if $(a_2,c_2)=(0,0)$ and $|U_2|=0$, so assume that $|U_2|=2$ and $(a_2,c_2) \neq (0,0)$. Then the definition of $b_2$ implies that $a_2+2c_2=w-2b_2$. So the existence of $\P_2$ follows immediately by Lemma~\ref{Lemma:3s5s} because $a_2+c_2\in \{\frac{1}{2}(w-2b_2),\ldots, w-2b_2\} \setminus \{1,2\}$ by (v), and our choice of $a_2$ and $c_2$.

\noindent {\bf Proof that $\P_3$ exists.} We will show that $\P_3$ exists using either Lemma~\ref{Lemma:BipartiteAndOneCycle} (if $t\in\{0,2,4\}$) or Lemmas~\ref{Lemma:Leave_ManyPathDecomp} and \ref{Lemma:PathsAndCycleToDecomp} (if $t\geq 6$). Note that $|U_3|=u-|U_2|-1\in\{u-3,u-1\}$ and hence $|U_3|\geq 4$, except when $u=5$ and $|U_2|=2$. We first establish two useful facts.

\begin{description}[leftmargin=0mm]
    \item[(a) $\bm{b-b_2+3j \geq 0}$ and $\bm{d-2j \geq 0}$.]
Obviously $b-b_2+3j \geq 0$ by the definition of $j$, and clearly $d-2j \geq 0$ if $b \geq b_2$ and hence $j=0$. So it remains to show that $d-2j \geq 0$ when $b < b_2$. Then $j=\lceil\frac{1}{3}(b_2-b)\rceil$. Observe that
\begin{equation}\label{fewCrossOddIneqs}
2a_3+4b+4c_3+6d+k+t = w|U_3|+4b_2
\end{equation}
by (ii) and the definitions of $a_3$, $c_3$, $b_2$ and $U_3$.
So it cannot be that $u=5$ because then $b>b_2$ by \eqref{fewCrossOddIneqs} (since $d=0$ by (iii), $|U_3| \geq 2$, $2a_3+4c_3\leq 12$ and, by (iii) and (iv), $t\leq 2$ and $k=0$). So assume that $u \geq 7$ and hence $|U_3| \geq 4$. Then it follows from \eqref{fewCrossOddIneqs} that $d \geq \frac{2}{3}(b_2-b)$ using the facts that $2a_3+4c_3 \leq 12$ and $k \leq t \leq w-2$. So we have $d-2j \geq 0$.
\item[(b) $\bm{2a_3+4(b-b_2+3j)+4c_3+6(d-2j)+k+t = |U_3|w}$.] Observe that
\begin{align*}
  &2a_3+4(b-b_2+3j)+4c_3+6(d-2j)+k+t \\
  ={}& (2a+4b+4c+6d+k+t)-(2a_2+4b_2+4c_2) \\
  ={}& |U_3|w
\end{align*}
where the final equality follows by (ii) and because $2a_2+4b_2+4c_2=|U_2|w$ by the definitions of $U_2$ and $b_2$.
\end{description}

\noindent \textbf{Case 1.} Suppose that $t\in\{0,2,4\}$. Then $k=0$. We claim that $\P_3$ exists by Lemma~\ref{Lemma:BipartiteAndOneCycle}. To see that we can apply Lemma~\ref{Lemma:BipartiteAndOneCycle}, note that $|U_3| \geq 2$ and that $d=0$ if $|U_3|=2$ by (iii). Also, using $2a_3+4c_3 \leq 12$, $2a_3+4c_3+t \leq 2w$. Finally, $a_3+c_3\in\{0,3\}$ if $t=0$ and $a_3+c_3\in\{1,\ldots,w-m+\frac{t}{2}+1\}$ if $t \in \{2,4\}$ by (v) and our choice of $a_3$ and $c_3$.

\noindent \textbf{Case 2.} Suppose that $t \geq 6$. Then $u \geq 9$  because $u \geq m \geq 8$ by (iii) and (iv). Also, $|U_3|\geq u-3 \geq 6$. By Lemma~\ref{Lemma:PathsAndCycleToDecomp}, to show that $\P_3$ exists it suffices to find a $(4^{b-b_2+3j},6^{d-2j},k)$-packing of $K_{U_3,W}$ whose reduced leave has a decomposition into a $t$-path and a $(2a_3+4c_3)$-path with end vertices in $W$ (note that $1 \leq a_3+c_3\leq 3\leq w-m+\frac{t}{2}+1$). By Lemma~\ref{Lemma:Leave_ManyPathDecomp}, to find such a packing it suffices to show that
$\max(k+2,2a_3+4c_3,8)+t\leq 2|U_3|+2$ and $(2a_3+4c_3,t,|U_3|,w)\neq (12,6,8,8)$ (note that $2a_3+4c_3 \in \{2,4,\ldots,12\}$). We have $\max(k+2,8)+t\leq 2|U_3|+2$ because $t \leq u-3$ (by (iii) and (iv)), $k\leq \frac{t+7}{3}$ and $|U_3|\geq u-3$. We have $2a_3+4c_3+t\leq 2|U_3|+2$ because $t \leq u-3$, $2a_3+4c_3 \leq 12$ and either $2a_3+4c_3 \leq 6$ or $|U_3|=u-1$ (by our choice of $a_3$ and $c_3$ and the definitions of $U_2$ and $U_3$). It follows directly from (v) that $(2a_3+4c_3,t,|U_3|,|W|)\neq (12,6,8,8)$.
\end{proof}

\subsection{Lists with many large entries}

\begin{lemma}\label{Lemma:BaseDecomp_ManyCrossOdds_large_m}
Let $u \geq 7$ and $w\geq 8$ be integers such that $u$ is odd and $w$ is even. Let $N$ be a list of integers and let $a$, $b$, $c$ and $d$ be nonnegative integers such that the following conditions hold.
\begin{itemize}
	\item[(i)]
$(\sum N)-t +a+c =\binom{w}{2}$, where $t \in \{2,4,\ldots,w-2\}$;
	\item[(ii)]
$2a+4b+4c+6d+k+t=uw$, where $k=\eceil{\frac{t+2}{3}}$ if $t\geq 12$ and $k=0$ otherwise;
	\item[(iii)]
$3\leq \ell \leq \min (u, w)$ for each entry $\ell$ in $N$;
	\item[(iv)]
$a\geq \frac{w}{2}+1$, $b\geq 1$ and $c\leq 1$;
	\item[(v)]
there is some entry $m\geq \max(t+2,7)$ in $N$ such that $uw\geq (a+c)\efloor{m}$ if  $a+2c> \frac{3w}{2}+3$;
	\item[(vi)]
$(m,t)\neq(w,2)$, and if $a\geq \frac{w}{2}+4$, $(m,t) \notin \{(w-1,2),(w,4)\}$.
\end{itemize}
Then there exists an $(N,3^a,4^b,5^c,6^d,k)$-decomposition of $K_{u+w}-K_u$ that includes cycles with lengths $(3^a,4^{b},5^c,6^{d},k)$ that each contain at most one pure edge.
\end{lemma}

\begin{proof}
Let $U$ and $W$ be disjoint sets of sizes $u$ and $w$ and observe that $K_{U\cup W}-K_U=K_{U,W} \cup K_{W}$. We first choose integers $a_2$, $a_3$, $c_2$ and $c_3$. Let $(a_3,c_3)$ be the first pair from the appropriate row below such that $a_3\leq a$, $c_3\leq c$, and $a_2+2c_2 \not\equiv 2 \mod{w}$ where $a_2=a-\frac{w}{2}-a_3$ and $c_2=c-c_3$. It is routine to check using (iv) that some pair in the appropriate row will always satisfy these conditions.
\begin{center}
\begin{tabular}{l|l}
  case & $(a_3,c_3)$ \\ \hline
  $a - \frac{w}{2}$ even & $(0,1), (2,0), (2,1), (4,0)$ \\
  $a - \frac{w}{2}$ odd & $(1,0), (1,1), (3,0)$ \\
\end{tabular}
\end{center}
This choice ensures that $a_2$, $a_3$, $c_2$ and $c_3$ are nonnegative integers such that $a_2$ is even, $\frac{w}{2}+a_2+a_3=a$, $c_2+c_3=c$, $a_3+c_3 \in \{1,2,3,4\}$, and $2a_3+4c_3 \leq 8$.

We now construct packings $\P_0,\ldots,\P_3$ as follows (we justify that these packings exist later).
\begin{itemize}
	\item
$\P_0$ is an $(N\setminus(m))$-packing of $K_{W}-I$, where $I$ is a $1$-factor on vertex set $W$. The reduced leave of $\P_0$ is a union of cycles
$C^\dag \cup C_1 \cup \cdots\cup C_{n}$, where
\begin{itemize}
    \item
$n$ and $b_2$ are the nonnegative integers such that $a_2+2c_2=nw-2b_2$ and $0 \leq b_2 \leq \frac{w-4}{2}$ (note that $a_2+2c_2 \not\equiv 2 \mod{w}$);
    \item
$C^\dag$ is an $(m-t+a_3+c_3)$-cycle;
    \item if $a_2+c_2>0$,
$|V(C_i)| \in\{\frac{w}{2},\ldots, w\}$ for $2\leq i\leq n$, and $|V(C_1)| \in \{\frac{1}{2}(w-b_2),\ldots, w-b_2\} \setminus \{1,2\}$;
    \item
$|V(C_1)|+\cdots+|V(C_{n})|=a_2+c_2$.
\end{itemize}
The cycle lengths $|V(C_1)|,\ldots,|V(C_{n})|$ exist by Lemma~\ref{Lemma:3s5s}, noting that $c_2 \leq 1$ by (iv), that $a_2$ is even, and that $a_2+2c_2 \not\equiv 2 \mod{w}$.
    \item
$\P_1$ is a $(3^{w/2})$-decomposition of $K_{U_1,W}\cup I$ for some $U_1 \subseteq U$ with $|U_1|=1$.
    \item
$\P_2$ is a $(3^{a_2},4^{b_2},5^{c_2})$-decomposition of $K_{U_2,W}\cup C_1\cup \cdots\cup C_n $, where $U_2 \subseteq U \setminus U_1$ with $|U_2|=2n$ and, if $b_2 > 0$, the union of the $4$-cycles in $\P_2$ is a copy of $K_{2,2b_2}$.
    \item
$\P_3$ is a $(3^{a_3},4^{b-b_2+3j},5^{c_3},6^{d-2j},k,m)$-decomposition of $K_{U_3,W} \cup C^\dag$, where $U_3=U\setminus (U_1\cup U_2)$ and
$$
j=\left\{\begin{array}{ll}
0 & \text{ if } b\geq b_2,\\
\lceil\frac{1}{3}(b_2-b)\rceil & \text{ otherwise.}
\end{array}
\right.$$

Furthermore, if $m>0$, there is an $m$-cycle in $\P_3$  that contains $m-t$ edges of $C^\dag$.
\end{itemize}

Note that the union $\P'=\P_0 \cup \cdots \cup \P_3$ will be an $(N,3^a,4^{b+3j},5^c,6^{d-2j},k)$-decomposition of $K_{U \cup W}-K_U$ and will contain cycles with lengths $(3^a,4^{b+3j},5^c,6^{d-2j},k)$ that each contain at most one pure edge (by the definition of $\P_3$ and $C^\dag$,  $\P_3$ has cycles with lengths $(3^{a_3},4^{b-b_2+3j},5^{c_3},6^{d-2j},k)$ that each contain at most one pure edge). If $b \geq b_2$, then $j=0$ and this will complete the proof. Otherwise $b_2 > b$ and we will be able to apply Lemma~\ref{Lemma:4s_to_6s} to $\P'$ to obtain a decomposition with the required properties provided we can find $3j$ $4$-cycles in $\P'$ that meet the hypotheses of Lemma~\ref{Lemma:4s_to_6s}. If $3j=b_2-b$, then $b_2 \geq 3j$ and we will be able to use $3j$ $4$-cycles from $\P_2$. If $3j \in \{b_2-b+1,b_2-b+2\}$, then $b_2 \geq 3j-1$ because $b \geq 1$ by (iv) and we will be able to use $3j-1$ $4$-cycles from $\P_2$ and any one $4$-cycle from $\P_3$.

So it remains to establish the existence of the packings $\P_0,\ldots,\P_3$. Obviously $\mathcal{P}_1$ exists.

\noindent {\bf Proof that $\P_0$ exists.}
First observe that $m-t+a_3+c_3 \in \{3,\ldots,w\}$  by (iii), (v), (vi)  and our choices of $a_3$ and $c_3$. Then, by Theorem~\ref{Theorem:Alspach}, a packing with the required properties exists by (iii) and because
\begin{align*}
{}&\textstyle\sum (N \setminus (m)) + |E(C^{\dag})| + |E(C_1)| + \cdots + |E(C_{n})|  \\
={}& (\tbinom{w}{2}+t-(a+c)-m) + (m-t+a_3+c_3) + (a_2+c_2)\\
={}& \tbinom{w}{2}-\tfrac{w}{2}.
\end{align*}
The first equality follows by (i) and the definitions of $C^\dag$ and $C_1,\ldots,C_{n}$. The second equality follows because $a_2+a_3+\frac{w}{2}=a$ and $c_2+c_3=c$.

\noindent {\bf Proof that $\P_2$ exists.} This is trivial if $(a_2,c_2)=(0,0)$. If $(a_2,c_2)\neq 0$, then this follows immediately by Lemma~\ref{Lemma:3s5s} because $|V(C_i)| \in\{\frac{w}{2},\ldots, w\}$ for $2\leq i\leq n$, $|V(C_1)| \in \{\frac{1}{2}(w-b_2),\ldots, w-b_2\} \setminus \{1,2\}$ and $|V(C_1)|+\cdots+|V(C_{n})|=a_2+c_2$.

\noindent {\bf Proof that $\P_3$ exists.}  We will show that $\P_3$ exists using either Lemma~\ref{Lemma:BipartiteAndOneCycle} (if $t\in\{2,4\}$) or Lemmas~\ref{Lemma:Leave_ManyPathDecomp} and \ref{Lemma:PathsAndCycleToDecomp} (if $t\geq 6$). We first establish some useful facts. Recall that $|E(C^{\dag})|=m-t+a_3+c_3$ and note that $m>t$.

\begin{description}[leftmargin=0mm]
    \item[(a) $\bm{b-b_2+3j \geq 0}$ and $\bm{d-2j \geq 0}$.]
Obviously $b-b_2+3j \geq 0$ by the definition of $j$ and we can establish that $d-2j \geq 0$ by a similar argument to the one used in the proof of Lemma~\ref{Lemma:BaseDecomp_FewCrossOdds}.
    \item[(b) $\bm{2a_3+4(b-b_2+3j)+4c_3+6(d-2j)+k+t=w|U_3|}$.]
Observe that
\begin{align*}
{}& 2a_3+4(b-b_2+3j)+4c_3+6(d-2j)+k+t  \\
={}& 2(a-a_2-\tfrac{w}{2})+4(b-b_2)+4(c-c_2)+6d+k+t\\
={}& (2a+4b+4c+6d+k+t)-(2a_2+4b_2+4c_2)-w \\
={}& w|U_3|.
\end{align*}
The first equality follows because $a_2+a_3+\frac{w}{2}=a$ and $c_2+c_3=c$. The final equality follows because $2a+4b+4c+6d+k+t=uw$ by (ii), $2a_2+4b_2+4c_2=w|U_2|$ by the definition of $U_2$, and $|U_3|=u-|U_2|-1$.
    \item[(c) $\bm{|U_3| \geq 4}$.]
Recall $|U_3|=u-2n-1$. If $n\leq 1$, then $|U_3|\geq 4$ because $u\geq 7$. So suppose $n\geq 2$. By the definition of $n$, this implies that $a_2+2c_2 \geq w+4$ and hence that $a+2c \geq \frac{3w}{2}+4$ because $a \geq a_2+\frac{w}{2}$ and $c \geq c_2$. Thus, using (v),
\[uw \geq (a+c)\efloor{m} \geq (\tfrac{w}{2}+a_2+c_2)\efloor{m}.\]
Now $2a_2+4c_2+4b_2=2nw$ by the definitions of $n$ and $b_2$, and hence $a_2+c_2=nw-2b_2-c_2$. Observe that $c_2 \leq 1$ by (iv) and so $a_2+c_2 \geq nw-2b_2-1$.
Using this fact, we have $uw \geq ((n+\frac{1}{2})w-2b_2-1)\efloor{m}$. Rearranging yields
\[u-2n-1 \geq (n+\tfrac{1}{2})(\efloor{m}-2)-\tfrac{1}{w}(2b_2+1)\efloor{m}.\]
Because $b_2 \leq \frac{w-4}{2}$ and hence $\tfrac{1}{w}(2b_2+1)<1$, we see that
\begin{equation}\label{Eq:UDashBound}
u-2n-1 > (n-\tfrac{1}{2})(\efloor{m}-2)-2.
\end{equation}
So because $n \geq 2$ and $\efloor{m} \geq 6$ by (v), we have $|U_3|=u-2n-1 \geq 4$.
\end{description}

\noindent \textbf{Case 1.}
Suppose that $t\in\{2,4\}$. Then $k=0$. So $\P_3$ exists by Lemma~\ref{Lemma:BipartiteAndOneCycle} because $|U_3| \geq 4$, $a_3+c_3 \in \{1,2,3,4\}$, $2a_3+4c_3 \leq 8$, and $w-m+\frac{t}{2}+1 \geq a_3+c_3$ by (vi) (by our choice of $a_3$ and $c_3$, $a \geq \frac{w}{2}+4$ if $a_3+c_3=4$).

\noindent \textbf{Case 2.} Suppose that $t \geq 6$. Then $u \geq 9$ because $u \geq m \geq 8$ by (iii) and (v). By Lemma~\ref{Lemma:PathsAndCycleToDecomp}, to show that $\P_3$ exists it suffices to find a $(4^{b-b_2+3j},6^{d-2j},k)$-packing of $K_{U_3,W}$ whose reduced leave has a decomposition into a $t$-path and a $(2a_3+4c_3)$-path (note that $a_3+c_3 \in \{1,2,3,4\}$).

Noting that $2a_3+4c_3 \leq 8$, by Lemma~\ref{Lemma:Leave_ManyPathDecomp}, to find such a packing it suffices to show that $\max(k+2,8)+t\leq 2|U_3|+2$. Note that $2|U_3|+2=2u-4n$. If $n \in \{0,1\}$, this holds because $u \geq 9$, $k \leq \frac{t+7}{3}$, and $t \leq u-3$ by (iii) and (v). So suppose that $n \geq 2$. We have
\[2u-4n > (2n-1)(\efloor{m}-2)-2 \geq 3t-2,\]
where the first inequality follows by multiplying \eqref{Eq:UDashBound} by 2, and the second holds because $n \geq 2$ and $\efloor{m} \geq t+2$ by (v). Thus $\max(k+2,8)+t\leq 2u-4n$ holds because $t \geq 6$ and $k \leq \frac{t+7}{3}$.
\end{proof}

\subsection{Lists with few large entries}

Lemma~\ref{Lemma:BaseDecomp_ManyCrossOdds_small_m} is the most intricate of our base decomposition lemmas and requires some more preliminary results. Lemma~\ref{Lemma:eqFournier} is an edge-colouring result that is easily obtained by combining a theorem of Fournier \cite{Fournier73} with the well-known result that any graph with chromatic index at most $\ell$ has a proper edge-colouring with $\ell$ colours such that the sizes of any two colour classes differ by at most one (see \cite{McDiarmid72}). Lemma \ref{Lemma:1factor5cycles} will allow us to decompose the union of $K_{3,w}$ and a graph on the part of size $w$ whose vertices have odd degrees. This is useful when $a$ is small. Lemmas~\ref{Lemma:SmalltLeave} and \ref{Lemma:SmalltLeave2} are more results giving cycle packings of the union of a complete bipartite graph with one or more cycles. The hypotheses of Lemmas~\ref{Lemma:SmalltLeave} and \ref{Lemma:SmalltLeave2} both concern the quantity $\rho=2a+4c+t$. This is the number of edges of $K_{U,W}$ that will be used in the $m$-cycle that contains $m-t$ edges of $K_W$, and the $a$ $3$-cycles and $c$ $5$-cycles that each contain one edge of $K_W$.

\begin{lemma}[{\cite{Fournier73,McDiarmid72}}]\label{Lemma:eqFournier}
Let $G$ be a graph with maximum degree $\ell$. If the subgraph of $G$ induced by the vertices of degree $\ell$ contains no cycle, then $G$ has a proper edge-colouring with $\ell$ colours such that the sizes of any two colour classes differ by at most one.
\end{lemma}

\begin{lemma}\label{Lemma:1factor5cycles}
Let $G$ be a graph with vertex set $W$ such that $|E(G)| =\lceil\frac{3}{4}|W|\rceil$ and each vertex of $G$ has degree $1$ or $3$. Let $\{p_1,p_2,p_3\}$ be a set of three vertices not in $W$ and let $\a\b$ be an edge of $G$. Then there exists a $(5^{\lfloor3|W|/4\rfloor})$-packing $\P$ of $K_{\{p_1,p_2,p_3\},W} \cup G$ such that each cycle in $\P$ contains exactly one edge from $G$ and
\begin{itemize}
    \item
if $|W|\equiv 0\mod{4}$, then the reduced leave of $\P$ is empty;
    \item
if $|W|\equiv 2\mod{4}$, then the reduced leave of $\P$ is the $3$-cycle $(p_3,\a,\b)$.
\end{itemize}
\end{lemma}

\begin{proof}
Let $w=|W|$. Let $B=\{x \in W:\deg_G(x)=3\}$ and note that $|B|=\lceil\frac{3w}{4}\rceil-\frac{w}{2}$.

\noindent {\bf Case 1.} Suppose that $w\equiv 0\mod{4}$. Then $|E(G)| = \frac{3w}{4}$. Let $E=E(G)$ and let $W'=\{v_e:e \in E\}$ be a set of $|E|$ vertices disjoint from $W \cup \{p_1,p_2,p_3\}$. Let $H$ be the graph obtained from $G$ by adding the vertices in $W'$ and then replacing each edge $yz \in E$ with the two edges $yv_{yz}$ and $zv_{yz}$.

Note that the maximum degree of $H$ is $3$, no two vertices of degree 3 are adjacent in $H$, and $|E(H)|= \frac{3w}{2}$. So by Lemma~\ref{Lemma:eqFournier} there exists a proper $3$-edge colouring $\gamma$ of $H$ with colour set $\{1,2,3\}$ such that $|\gamma^{-1}(i)|=\frac{1}{3}|E(H)|=\frac{w}{2}$ for $i \in \{1,2,3\}$. For each vertex $x \in V(H)$ we denote by $\gamma(x)$ the set of colours assigned by $\gamma$ to the edges incident with $x$. For $i \in \{1,2,3\}$, let $W'_i=\{x \in V(H):\gamma(x)=\{1,2,3\}\setminus\{i\}\}$ and $A_i=\{x \in W: \gamma(x)=\{i\}\}$. Let $A=\{x \in W:\deg_G(x)=1\}$.

We shall show that there is a bijection $f:W' \rightarrow A$ such that $f(W'_i) = A_i$ for each $i \in \{1,2,3\}$. Then
$$\mathcal{P}=\{(p_{\gamma(yv_{yz})},y,z,p_{\gamma(zv_{yz})},f(v_{yz})):yz \in E\}$$
will form a packing of $K_{\{p_1,p_2,p_3\},W} \cup G$ with $|E|$ $5$-cycles, each of which contains four edges of $K_{\{p_1,p_2,p_3\},W}$ and one edge in $E$. Thus the reduced leave of $\mathcal{P}$ will be empty. So it suffices to show that such a bijection $f$ exists, and hence it suffices to show that $|W'_i| = |A_i|$ for $i \in \{1,2,3\}$.

Obviously $|W'_1|+|W'_2|+|W'_3|=|E|$. Because each edge of $H$ is incident with exactly one vertex in $W'$, we have $|\gamma^{-1}(k)|=|W'_i|+|W'_j|$ for $\{i,j,k\}=\{1,2,3\}$. So, because the colour classes of $\gamma$ have equal size, it follows that $|W'_1|=|W'_2|=|W'_3|=\frac{1}{3}|E|=\frac{w}{4}$. Furthermore, for any $\{i,j,k\}=\{1,2,3\}$, we have $2|\gamma^{-1}(i)|=|A_i|+|W'_j|+|W'_k|+|B|$ by considering the total degree of the graph induced by the colour class $\gamma^{-1}(i)$. Solving for $|A_i|$, it follows that $|A_1|=|A_2|=|A_3|=\frac{1}{3}|A|=\frac{w}{4}$.

\noindent {\bf Case 2.} Suppose that $w\equiv 2\mod{4}$. Then $|E(G)| = \frac{3w+2}{4}$. The proof proceeds as in Case 1 with the following exceptions. We let $E=E(G) \setminus \{\a\b\}$, so that $|E|=\frac{3w-2}{4}$ and again $|E(H)|=\frac{3w}{2}$. At most one pair of vertices of degree 3 are adjacent in $H$, so Lemma~\ref{Lemma:eqFournier} can still be applied and again each colour class of $\gamma$ has size $\frac{w}{2}$. We may assume without loss of generality that $\gamma(\a\b)=3$. The reduced leave of the packing $\mathcal{P}$ will be the $3$-cycle $\{p_3,\a,\b\}$. Because each edge of $H$ except $\a\b$ is incident with exactly one vertex in $W'$, we find $|W'_1|=|W'_2|=\frac{w-2}{4}$ and $|W'_3|=\frac{w+2}{4}$. We deduce that $|A_1|=|A_2|=\frac{w-2}{4}$ and $|A_3|=\frac{w+2}{4}$.
\end{proof}

\begin{lemma}\label{Lemma:SmalltLeave}
Let $U'$ and $W$ be sets such that $|U'| \geq 2$ and $|W| \geq 6$ are even, let $(m,t)\in\{(0,0),(4,2),(5,2),(6,2),(6,4)\}$, let $a$, $b$, $c$ and $d$ be nonnegative integers, and let $\rho=2a+4c+t$. Suppose that
\begin{itemize}
    \item[(i)]
$d=0$ if $|U'| = 2$;
    \item[(ii)]
$\rho+4b+6d=|U'||W|$;
    \item[(iii)]
$\rho \in \{0\} \cup \{4,6,\ldots,2|W|\}$;
    \item[(iv)]
$t=2$ when $\rho=4$, and $(a,c) \notin \{(0,2),(1,1)\}$ when $\rho \in \{6,8\}$ and $t=0$; and
    \item[(v)]
if $t \in \{2,4\}$ and $\rho=2|W|-2i$ for some $i \in \{0,1,2\}$, then $m \leq c+t+i+1$.
\end{itemize}
Let $C$ be an $(a+c+m-t)$-cycle such that $V(C) \subseteq W$ (note that $a+c+m-t \in \{0\} \cup \{3,\ldots,|W|\}$). Then there exists a $(3^a,4^{b},5^c,6^{d},m)$-decomposition of $K_{U',W} \cup C$ that, if $m>0$, includes an $m$-cycle containing $m-t$ edges of $C$.
\end{lemma}

\begin{proof}
The result follows immediately from Lemma~\ref{Lemma:BipartiteAndOneCycle}. To see that hypothesis (iv) of Lemma~\ref{Lemma:BipartiteAndOneCycle} is satisfied when $t \in \{2,4\}$, it may help to note the following facts. If $t \in \{2,4\}$ and $\rho=2|W|-2i$ for some $i \in \{0,1,2\}$, then $c+t+i+1=|W|+\frac{t}{2}+1-a-c$ by the definition of $i$, and hence $m \leq c+t+i+1$ implies $|W|-m+\frac{t}{2}+1 \geq a+c$. If $t \in \{2,4\}$ and $\rho \leq 2|W|-6$, then $a+c \leq |W|-3-\frac{t}{2}$ because $\rho \leq 2|W|-6$, and hence $a+c \leq |W|-m+\frac{t}{2}+1$ because $(m,t)\in\{(4,2),(5,2),(6,2),(6,4)\}$.
\end{proof}

\begin{lemma}\label{Lemma:SmalltLeave2}
Let $U'$ and $W$ be sets with $|U'|,|W|$ even, $|U'| \geq 4$ and $|W| \geq 10$, let $(m,t)\in\{(0,0),(4,2),(5,2),(6,2),(6,4)\}$, and let $a$, $b$, $c$ and $d$ be nonnegative integers, and let $\rho=2a+4c+t$. Suppose that
\begin{itemize}
    \item[(i)]
$\rho+4b+6d=|U'||W|$;
    \item[(ii)]
$\rho \in \{2|W|-4,2|W|-2,\ldots,4|W|\}$, $t=2$ if $\rho \in \{2|W|-4,2|W|-2\}$, and $t\in\{2,4\}$ if $\rho=2|W|$; and
    \item[(iii)]
if $\rho\geq 4|W|-6$ and $t \in \{2,4\}$ then $c \geq 3$.
\end{itemize}
Then there are integers $\ell_1,\ell_{2} \in\{3,\ldots,w\}$ such that $\ell_1+\ell_{2}=a+c+m-t$ and, for any edge-disjoint cycles $C_1$ and $C_{2}$ in $K_{W}$ with lengths $\ell_1$ and $\ell_{2}$,
\begin{itemize}
    \item
if $|U'|\geq 6$ or $d$ is even, there exists a $(3^{a},4^{b},5^{c},6^{d},m)$-decomposition $\P$ of $K_{U',W}\cup C_1 \cup C_{2}$;
    \item
if $|U'| = 4$, $d$ is odd and $c \geq 1$, there exists a $(3^{a},4^{b},5^{c-1},6^{d-1},m)$-packing $\P$ of $K_{U',W}\cup C_1 \cup C_{2}$ whose reduced leave contains exactly one edge of $K_W$ and has a $(3,4,4)$-decomposition;
    \item
if $|U'| = 4$, $d$ is odd and $c = 0$, there exists a $(3^{a-1},4^{b},6^{d-1},m)$-packing $\P$ of $K_{U',W}\cup C_1 \cup C_{2}$ whose reduced leave contains exactly one edge of $K_W$, has a $(4,5)$-decomposition, and has a vertex of degree $4$ in $W$.
\end{itemize}
Furthermore, if $m>0$, then in each case there is an $m$-cycle in $\P$  that contains $m-t$ edges of $K_W$ (or $C_2$).
\end{lemma}

\begin{proof}
Let $w=|W|$. Let $U'_1 \subseteq U'$ with $|U'_1|=2$ and let $U'_2= U' \setminus U'_1$. Let $\delta=1$ if $d$ is odd and $\delta=0$ if $d$ is even. Note that, by (i),
\begin{equation}\label{Eq:xCongSmalltLeave2}
\rho+2\delta \equiv 0 \mod{4}.
\end{equation}
We will select values for $a_1$, $b_1$, $c_1$, $a_2$, $b_2$, $c_2$ and $d_2$ such that $a_1+a_2+c_1+c_2=a+c$ according to the following table (the criteria for the cases are given below).

\begin{footnotesize}
\begin{center}
\begin{tabular}{|c||c|c|c|c|c|c|c|c|}
  case & $a_1$ & $b_1$ & $c_1$ & $a_2$ & $b_2$ & $c_2$ & $d_2$ \\[0.15cm] \hline
  1 & $\min(\efloor{a},w)$ & $0$ & $\frac{w-a_1}{2}$ & $a-a_1$ & $b$ & $c-c_1$ & $d$ \\[0.15cm]
  2 & $\min(\efloor{a},w-4)$ & $2$ & $\frac{w-a_1}{2}-2$ & $a-a_1$ & $b-2$ & $c-c_1$ & $d$ \\[0.15cm]
  3 & $\min(\efloor{a},w-4)$ & $2$ & $\frac{w-a_1}{2}-2$ & $a-a_1$ & $b+1$ & $c-c_1$ & $d-2$ \\[0.15cm]
  4 & $\min(\efloor{a+\delta},w)$ & $0$ & $\frac{w-a_1}{2}$ & $a+\delta-a_1$ & $b+\frac{3(d-\delta)}{2}+2\delta$ & $c-\delta-c_1$ & $0$ \\[0.15cm]
  5 & $\max(0,w-4-2c+2\delta)$ & $2$ & $\frac{w-a_1}{2}-2$ & $a+\delta-a_1$ & $b+\frac{3(d-\delta)}{2}+2\delta-2$ & $c-\delta-c_1$ & $0$ \\[0.15cm]
  6 & $w-2\delta$ & $\delta$ & $0$ & $a-\delta-a_1$ & $b+\frac{3(d-\delta)}{2}$ & $\delta$ & $0$ \\[0.15cm]
  7 & $w-4$ & $2$ & $0$ & $a-\delta-a_1$ & $b+\frac{3(d-\delta)}{2}+\delta-2$ & $\delta$ & $0$ \\[0.15cm]
\end{tabular}
\end{center}
\end{footnotesize}

We will apply Lemma~\ref{Lemma:SmalltLeave} to show that $a_1+c_1$ and $a_2+c_2+m-t$ are in $\{3,\ldots,w\}$ and that, for any edge-disjoint cycles $C_1$ and $C_{2}$ in $K_{W}$ with lengths $a_1+c_1$ and $a_2+c_2+m-t$, there is a $(3^{a_1},4^{b_1},5^{c_1})$-packing $\P_1$ of $K_{U'_1,W}\cup C_1$ and a $(3^{a_2},4^{b_2},5^{c_2},6^{d_2},m)$-packing $\P_2$ of $K_{U'_2,W} \cup C_{2}$ from which we can obtain a packing with the required properties. In each case the fact that the hypotheses of Lemma~\ref{Lemma:SmalltLeave} are satisfied when constructing $\P_1$ and $\P_2$ can be deduced from \eqref{Eq:xCongSmalltLeave2}, the hypotheses of this lemma and the criteria of the relevant case. In particular, we use $w \geq 10$ frequently. For brevity, let $\rho_1=2a_1+4c_1$ and $\rho_2=2a_2+4c_2+t$. For each case, we now detail the criteria for the case, explain how a packing with the required properties can be obtained from $\P_1 \cup \P_2$, and justify some of the less obvious deductions required to see that the hypotheses of Lemma~\ref{Lemma:SmalltLeave} are satisfied. To show that hypotheses (iii), (iv) and (v) of Lemma~\ref{Lemma:SmalltLeave} are satisfied in constructing $\P_2$, it suffices to show that
\begin{equation}\label{Eq:hyp4SmalltLeave2}
\mbox{$\rho_2 \geq 4$, $t=2$ when $\rho_2=4$, and $(a_2,c_2) \notin \{(0,2),(1,1)\}$ when $\rho_2 \in \{6,8\}$ and $t=0$},
\end{equation}
and
\begin{equation}\label{Eq:hyp5SmalltLeave2}
\mbox{$\rho_2 \leq 2w$, and $m \leq c_2+t+i+1$ when $t \in \{2,4\}$ and $i \in \{0,1,2\}$},
\end{equation}
where $i$ is the integer such that $\rho_2=2w-2i$.

\noindent {\bf Case 1: $\bm{|U'| \geq 6}$, $\bm{\rho \geq 2w+8}$, and $\bm{t \in \{2,4\}}$ if $\bm{\rho = 2w+8}$.} In this case $\P_1 \cup \P_2$ is itself a packing with the required properties. Note that $\rho_1=2w$ and $\rho_2=\rho-2w$.
We have $c_2 \geq 0$ because either $c_1=0$ or $a_2 \in \{0,1\}$ and $2a_2+4c_2 = \rho_2-t \geq 4$. We have \eqref{Eq:hyp4SmalltLeave2} by the criteria for this case. We have \eqref{Eq:hyp5SmalltLeave2} by (iii) when $c_1=0$ and because
\[c_2+t+i+1 = \tfrac{1}{4}(\rho_2-2a_2-t)+t+i+1 \geq \tfrac{2w+2i+3t+4-2a_2}{4} \geq 7 \quad \mbox{for $t \in \{2,4\}$}\]
when $a_2 \in \{0,1\}$.

\noindent {\bf Case 2: $\bm{|U'| \geq 6}$, $\bm{\rho \leq 2w+8}$, $\bm{t=0}$ if $\bm{\rho = 2w+8}$, and $\bm{b \geq 2}$.} In this case $\P_1 \cup \P_2$ is itself a packing with the required properties. Note that $\rho_1=2w-8$ and $\rho_2=\rho-2w+8$. We have $c_2 \geq 0$ because either $c_1=0$ or $a_2 \in \{0,1\}$ and $2a_2+4c_2 = \rho_2-t \geq 4$ by (ii). We have \eqref{Eq:hyp4SmalltLeave2} by (ii). We have \eqref{Eq:hyp5SmalltLeave2} by the criteria for this case.

\noindent {\bf Case 3: $\bm{|U'| \geq 6}$, $\bm{\rho \leq 2w+8}$, $\bm{t=0}$ if $\bm{\rho = 2w+8}$, and $\bm{b \in \{0,1\}}$.} Lemma~\ref{Lemma:4s_to_6s} can be applied to $\P_1 \cup \P_2$ (using the two $4$-cycles in $\P_1$ and any one $4$-cycle in $\P_2$) to obtain a packing with the required properties. We have $d_2 \geq 0$ because $4b+6d = |U'|w-\rho \geq 4w-8$ by (i) and the criteria for this case. We have $c_2 \geq 0$, \eqref{Eq:hyp4SmalltLeave2} and \eqref{Eq:hyp5SmalltLeave2} by similar arguments to those in Case 2.

\noindent {\bf Case 4: $\bm{|U'| = 4}$, $\bm{c \geq 1}$, $\bm{\rho \geq 2w+8}$, and $\bm{t \in \{2,4\}}$ if $\bm{\rho \in \{2w+8,2w+10\}}$.} Lemma~\ref{Lemma:4s_to_6s} can be applied to $\P_1 \cup \P_2$ (using any $\frac{3}{2}(d-\delta)$ $4$-cycles in $\P_1 \cup \P_2$) to obtain a $(3^{a+\delta},4^{b+2\delta},5^{c-\delta},6^{d-\delta},m)$-decomposition of $K_{U',W}\cup C_1 \cup C_{2}$. If $\delta=0$ this completes the proof and if $\delta=1$ we can remove two $4$-cycles and a $3$-cycle to obtain a packing with the required properties. Note that $\rho_1=2w$ and $\rho_2=\rho-2w-2\delta$. We have $c_2 \geq 0$ because either $c_1=0$ or $a_2 \in \{0,1\}$ and $2a_2+4c_2 = \rho_2-t \geq 4$ (note that if $\delta=1$, then $\rho \geq 2w+10$ by \eqref{Eq:xCongSmalltLeave2}). We have \eqref{Eq:hyp4SmalltLeave2} by the criteria for this case. We have \eqref{Eq:hyp5SmalltLeave2} by similar arguments to those in Case~1.

\noindent {\bf Case 5: $\bm{|U'| = 4}$, $\bm{c \geq 1}$, $\bm{\rho \leq 2w+10}$, and $\bm{t=0}$ if $\bm{\rho \in \{2w+8,2w+10\}}$.} Lemma~\ref{Lemma:4s_to_6s} can be applied to $\P_1 \cup \P_2$ (using any $\frac{3}{2}(d-\delta)$ $4$-cycles in $\P_1 \cup \P_2$) to obtain a $(3^{a+\delta},4^{b+2\delta},5^{c-\delta},6^{d-\delta},m)$-decomposition of $K_{U',W}\cup C_1 \cup C_{2}$. If $\delta=0$ this completes the proof and if $\delta=1$ we can remove two $4$-cycles and a $3$-cycle to obtain a packing with the required properties. Note that $\rho_1=2w-8$ and $\rho_2=\rho-2w+8-2\delta$. We have $a_2 \geq 0$ because $2a+4c = \rho-t \geq 2w-6$ by (ii) and hence $a+2c \geq w-3$. Further, $a_2 \geq \delta$ when $\rho \geq 2w-2$. We have $b_2 \geq 0$ because $4b+6d = 4w-\rho \geq 2w-10$ by (i) and the criteria for this case and hence $b \geq \frac{w-3d-5}{2}$. We have $c_2 \geq 0$ because either $c \geq \frac{w}{2}-2+\delta$ and $c_1=\frac{w}{2}-2$ or $c_1=c-\delta$. We have \eqref{Eq:hyp4SmalltLeave2} by (ii) (note that if $\delta=1$, then $\rho \geq 2w-2$ by \eqref{Eq:xCongSmalltLeave2} and that $a_2 \geq \delta$ when $\rho \geq 2w-2$). We have \eqref{Eq:hyp5SmalltLeave2} by the criteria for this case.

\noindent {\bf Case 6: $\bm{|U'| = 4}$, $\bm{c = 0}$, and $\bm{\rho \geq 2w+8}$.} If $\delta=1$, then we can ensure that the $4$-cycle in $\P_1$ and the $5$-cycle in $\P_2$ with one edge of $K_W$ share at least one vertex in $W$. We will justify this below. Lemma~\ref{Lemma:4s_to_6s} can be applied to $\P_1 \cup \P_2$ (using $\frac{3}{2}(d-\delta)$ $4$-cycles in $\P_2$) to obtain a $(3^{a-\delta},4^{b+\delta},5^{\delta},6^{d-\delta},m)$-decomposition of $K_{U',W}\cup C_1 \cup C_{2}$ in which, if $\delta=1$, a $4$-cycle with no edges of $K_W$ and a $5$-cycle with one edge of $K_W$  share a vertex in $W$. If $\delta=0$ this completes the proof and if $\delta=1$ we can remove the $4$-cycle and $5$-cycle from the decomposition to obtain a packing with the required properties. Note that $\rho_1=2w-4\delta$ and $\rho_2=\rho-2w+6\delta$. We have $a_2 \geq 0$ because $2a = \rho-t \geq 2w+4$ by the criteria for this case and hence $a \geq w+2$. We have \eqref{Eq:hyp4SmalltLeave2} by the criteria for this case. We have \eqref{Eq:hyp5SmalltLeave2} because $\rho \leq 4w-6\delta$ by (i) and because $\rho \leq 4w-8$ when $t \in \{2,4\}$ by (iii) (note that if $\delta=1$, then $\rho \leq 4w-10$ by \eqref{Eq:xCongSmalltLeave2}).

It remains to show that, if $\delta=1$, then we can ensure that the $4$-cycle $X$ in $\P_1$ and the $5$-cycle $Y$ in $\P_2$ with one edge of $K_W$ share at least one vertex in $W$. Note that $V(X) \cap W=W \setminus V(C_1)$. When $V(C_2) \nsubseteq V(C_1)$, we can permute the vertices of $\P_2$ so that the edge of $Y$ in $K_W$ is incident with a vertex in $V(C_2) \setminus V(C_1)$ and hence in $V(X)$. When $V(C_2) \subseteq V(C_1)$, noting that $|V(C_2)| \leq w-2$ and $t \leq 4$, we can ensure that $Y$ has a vertex in $W \setminus V(C_2)$. (This can be seen by directly applying Lemma \ref{Lemma:PathsAndCycleToDecomp} to construct $\P_2$. The hypotheses of Lemma \ref{Lemma:PathsAndCycleToDecomp} are satisfied since the reduced leave of a $(4^{b_2})$-packing of $K_{U_2',W}$ is a copy of $K_{2,w-2b_2}$, which clearly has the required path decomposition.) We can then permute the vertices of $\P_2$ so that this vertex of $Y$ is in $W \setminus V(C_1)$ and hence in $V(X)$.

\noindent {\bf Case 7: $\bm{|U'| = 4}$, $\bm{c = 0}$, $\bm{\rho \leq 2w+6}$.}
A packing with the required properties can be obtained from $\P_1 \cup \P_2$ as in Case 6. Note that $\rho_1=2w-8$ and $\rho_2=\rho-2w+8+2\delta$. We have $b_2 \geq 0$ because $4b+6d=4w-\rho \geq 2w-6$ by (i) and the criteria for this case and hence $b \geq \frac{w-3d-3}{2}$. We have $a_2 \geq 0$ by similar arguments to those in Case~5.  We have \eqref{Eq:hyp4SmalltLeave2} by (ii). We have $\rho_2 \leq 16$ by the criteria for this case, so \eqref{Eq:hyp5SmalltLeave2} holds.
\end{proof}

\begin{lemma}\label{Lemma:BaseDecomp_ManyCrossOdds_small_m}
Let $u \geq 5$ and $w \geq 10$ be integers such that $u$ is odd and $w$ is even. Let $N$ be a list of integers and let $a$, $b$, $c$ and $d$ be nonnegative integers such that the following conditions hold.
\begin{itemize}
	\item[(i)]
$(\sum N)-t +a+c =\binom{w}{2}$, where $t\in\{0,2,4\}$;
	\item[(ii)]
$2a+4b+4c+6d+t=uw$;
	\item[(iii)]
$3\leq \ell\leq \min (u, w)$ for each entry $\ell$ in $N$, and $d=0$ if $u=5$;
	\item[(iv)]
either $a\geq \frac{w}{2}$ and $a+c\geq \frac{w}{2}+3$, or $c\geq \frac{3w}{4}$ and $a+c\geq \frac{3w}{4}+4$;
	\item[(v)]
if $b+d\leq 2$ and $t\in\{2,4\}$, then $a \in \{0,1,2,3,4,\frac{w}{2},\frac{w}{2}+1,\frac{w}{2}+2,\frac{w}{2}+3\}$;
	\item[(vi)]
if $t\in\{2,4\}$, there is some entry $m$ in $N$ such that $(m,t) \in \{(4,2),(5,2),(6,2),(6,4)\}$.
\end{itemize}
Then there exists an $(N,3^a,4^b,5^c,6^d)$-decomposition of $K_{u+w}-K_u$ that includes cycles with lengths $(3^a,4^{b},5^c,6^{d})$ that each contain at most one pure edge.
\end{lemma}

\begin{proof}
Let $U$ and $W$ be disjoint sets of sizes $u$ and $w$ and observe that $K_{U\cup W}-K_U=K_{U,W} \cup K_{W}$. Let $m=0$ if $t=0$. We first choose disjoint subsets $U_1$ and $U_2$ of $U$ and nonnegative integers $a_1$, $a_2$, $a_3$, $c_1$, $c_2$ and $c_3$. Let $a_1$, $c_1$ and $|U_1|$ be given as follows.
\begin{center}
\begin{tabular}{|l|c|c|c|}
  \hline
  case & $a_1$ & $c_1$ & $|U_1|$ \\ \hline
  $a\geq \frac{w}{2}$ & $\frac{w}{2}$ & $0$ & $1$ \\
  $a<\frac{w}{2}$ and $w\equiv 0\mod{4}$ & $0$ & $\frac{3w}{4}$ & $3$ \\
  $1 \leq a < \frac{w}{2}$ and $w\equiv 2\mod{4}$ & $1$ & $\frac{3w-2}{4}$ & $3$ \\
  $a=0$ and $w\equiv 2\mod{4}$ & $0$ & $\frac{3w-2}{4}$ & $3$ \\
  \hline
\end{tabular}
\end{center}
Using (iv), we see that $a_1 \leq a$ and $c_1 \leq c$. Further, let $a'$, $c'$, $b_3$ and $d_3$ be given as follows.
\begin{center}
\begin{tabular}{|l|c|c|c|c|}
  \hline
  case & $a'$ & $c'$ & $b_3$ & $d_3$ \\ \hline
  $a=0$, $w\equiv 2\mod{4}$, $d$ is even & $1$ & $c-c_1-2$ & $b+1$ & $d$ \\
  $a=0$, $w\equiv 2\mod{4}$, $d$ is odd & $0$ & $c-c_1-1$ & $b+2$ & $d-1$ \\
  $a \geq 1$ or $w\equiv 0\mod{4}$ & $a-a_1$ & $c-c_1$ & $b$ & $d$ \\
  \hline
\end{tabular}
\end{center}
Let $\rho=2a'+4c'+t$. Using (iv) we see that $a'$, $c'$, $b_3$ and $d_3$ are nonnegative. Let $|U_2|$ and $\rho_3$ be the nonnegative even integers that satisfy the conditions given below.
\begin{center}
\begin{tabular}{|l|l|}
  \hline
  case & conditions \\ \hline
  $(u,|U_1|)=(5,3)$ or $\rho \leq 8$ & $\rho=|U_2|w+\rho_3$, $|U_2|=0$ \\
  $10 \leq \rho \leq (u-|U_1|-4)w+8$ & $\rho=|U_2|w+\rho_3$, $\rho_3 \in \{10,12,\ldots,2w+8\}$ \\
  $\rho \geq (u-|U_1|-4)w+10$, $(u,|U_1|)\neq(5,3)$ & $\rho=|U_2|w+\rho_3$, $|U_2|= u-|U_1|-4$ \\
  \hline
\end{tabular}

\end{center}
Note that $|U_2| \in \{0,2,\ldots,u-|U_1|-4\}$ unless $(u,|U_1|)=(5,3)$ and $|U_2|=0$. Now let
\[a_2=\min(\efloor{a'},\tfrac{1}{2}|U_2|w), \quad c_2=\tfrac{1}{4}(|U_2|w-2a_2), \quad a_3=a'-a_2, \quad c_3=c'-c_2.\]
By our definitions, $2a_2+4c_2=|U_2|w$ and $2a_3+4c_3+t=\rho_3$. Clearly, $a_2$, $c_2$ and $a_3$  are nonnegative. When $\rho_3 < 10$ we have $|U_2|=0$ and $(a_3,c_3)=(a',c')$. Thus it follows from (iv) and the definitions of $a'$ and $c'$ that either $\rho_3 \geq 10$ or $a_3+c_3 \geq 3$. Furthermore, $c_3$ is nonnegative because $c_2=0$ when $a_2= \tfrac{1}{2}|U_2|w$, and $a_3\in \{0,1\}$ and $2a_3+4c_3+t=\rho_3 \geq 6$ when $a_2=\efloor{a'}$. It may be that $a_1+a_2+a_3 \neq a$ or $c_1+c_2+c_3 \neq c$. However, $\mathcal{P}_3$ (defined below) will be produced by applying Lemma~\ref{Lemma:SmalltLeave} or \ref{Lemma:SmalltLeave2} with $(a,b,c,d)=(a_3,b_3,c_3,d_3)$ and then possibly removing cycles. Observe that $\rho_3$, $w$ and $t$ satisfy one of the following
\begin{equation}\label{Eq:xBoundsSmalltLeave}
\mbox{ $\rho_3 \leq 2w$, $t \in \{0,4\}$ if $\rho_3 \in \{2w-4,2w-2\}$, and $t=0$ if $\rho_3=2w$.}
\end{equation}
\begin{equation}\label{Eq:xBoundsSmalltLeave2}
\mbox{$\rho_3 \geq 2w-4$, $t=2$ if $\rho_3 \in \{2w-4,2w-2\}$, and $t\in\{2,4\}$ if $\rho_3=2w$.}
\end{equation}

We now construct packings $\P_0,\ldots,\P_3$ as follows (we justify that these packings exist later).
\begin{itemize}
	\item
$\P_0$ is an $(N\setminus(m))$-packing of $K_{W}-I$, where $I$ is a $1$-factor on vertex set $W$. The reduced leave of $\P_0$ is the edge-disjoint union of cycles $C^\star,C_1,\ldots,C_{n},C_1^\dag,C_2^\dag$, where
\begin{itemize}
    \item
$C^\star$ is trivial if $a \geq \frac{w}{2}$ and $|E(C^\star)|=\lceil\frac{w}{4}\rceil$ otherwise; and
    \item
$n=\frac{|U_2|}{2}$, $|E(C_i)| \in\{\frac{w}{2},\ldots, w\}$ for $1\leq i\leq n$, and $\sum_{i=1}^n|E(C_i)|=a_2+c_2$; and
    \item
$|E(C^\dag_1)|+|E(C^\dag_2)|=m-t+a_3+c_3$.
\end{itemize}
The cycle lengths $|E(C_1)|,\ldots,|E(C_{n})|$ will be given by Lemma~\ref{Lemma:3s5s} (note that $a_2+2c_2=nw$). The cycle lengths $|E(C^\dag_1)|$ and $|E(C^\dag_{2})|$ will be given by Lemma~\ref{Lemma:SmalltLeave} or \ref{Lemma:SmalltLeave2}.
    \item
$\P_1$ is a $(3^{a_1},5^{c_1})$-packing of $K_{U_1,W}\cup I \cup C^\star$. The reduced leave $L_1$ of $\P_1$ is a $3$-cycle if $a=0$ and $w \equiv 2 \mod{4}$ and is trivial otherwise.
    \item
$\P_2$ is a $(3^{a_2}, 5^{c_2})$-decomposition of $K_{U_2,W}\cup C_1\cup \cdots\cup C_n$.
    \item
$\P_3$ is a packing of $K_{U_3, W} \cup C_1^\dag \cup C_2^\dag$ that, if $m>0$, includes an $m$-cycle containing $m-t$ edges of $K_W$, where $U_3=U \setminus (U_1 \cup U_2)$, with a reduced leave $L_3$. The properties of $\P_3$ and $L_3$ divide according to the following cases. The cases are mutually exclusive because $d_3$ is defined so as to be even when $a=0$ and $w\equiv 2\mod{4}$.
\begin{description}
    \item[Case 1: $\bm{a=0}$, $\bm{w \equiv 2 \mod{4}}$, and $\bm{d}$ is even.]\hfill\\
 Then $\P_3$ is a $(3^{a_3-1},4^{b_3-1},5^{c_3},6^{d_3},m)$-packing, $L_3$ has exactly one pure edge, $L_3$ has a $(3,4)$-decomposition, and $L_1 \cup L_3$ has a vertex of degree $4$;
    \item[Case 2: $\bm{a=0}$, $\bm{w \equiv 2 \mod{4}}$, and $\bm{d}$ is odd.]\hfill\\
Then $\P_3$ is a $(3^{a_3},4^{b_3-2},5^{c_3},6^{d_3},m)$-packing, $L_3$ has no pure edges, $L_3$ has a $(4,4)$-decomposition, and $L_1 \cup L_3$ has a vertex of degree $4$;
    \item[Case 3: \eqref{Eq:xBoundsSmalltLeave2} holds, $\bm{|U_3| = 4}$, $\bm{d_3}$ is odd, and $\bm{c_3 \geq 1}$.]
    \hfill\\
Then $\P_3$ is a $(3^{a_3},4^{b_3},5^{c_3-1},6^{d_3-1},m)$-packing, $L_3$ has exactly one pure edge, and $L_3$ has a $(3,4,4)$-decomposition;
    \item[Case 4: \eqref{Eq:xBoundsSmalltLeave2} holds, $\bm{|U_3| = 4}$, $\bm{d_3}$ is odd, and $\bm{c_3 = 0}$.]\hfill\\
Then $\P_3$ is a $(3^{a_3-1},4^{b_3},5^{c_3},6^{d_3-1},m)$-packing, $L_3$ has exactly one pure edge, $L_3$ has a $(4,5)$-decomposition, and there is a vertex in $W$ with degree $4$ in $L_3$;
    \item[Case 5: otherwise.]\hfill\\
Then $\P_3$ is a $(3^{a_3},4^{b_3},5^{c_3},6^{d_3},m)$-decomposition and $L_3$ is trivial.
\end{description}
\end{itemize}

Let $\mathcal{P}'=\P_0\cup \P_1 \cup \P_2 \cup \P_3$. Then $\P'$ is a packing of $K_{U\cup W}-K_U$ with reduced leave $L_1 \cup L_3$. If we are in Case 5 then $\P'$ is an $(N,3^a,4^b,5^c,6^d)$-decomposition with the required properties. Otherwise we can obtain an $(N,3^a,4^b,5^c,6^d)$-decomposition of $K_{U\cup W}-K_U$ with the required properties by applying Lemma~\ref{Lemma:1Deg4Vertex_rearrange} with $m$ and $m'$ as per the following table. (That $\P'$ is an $(N,3^a,4^b,5^c,6^d)$-decomposition in Case 5, and that the entries in the second and third columns of the table are correct, can be checked using the definitions of $\P_0,\ldots,\P_3$, $a'$, $c'$, $a_3$, $b_3$, $c_3$ and $d_3$.)

\begin{center}
\begin{tabular}{|c|c|c|c|}
  \hline
  case & cycle type of $\mathcal{P}'$ & size of $L_1 \cup L_3$ & $(m,m')$ \\
  \hline
  1 & $(N,3^a,4^b,5^{c-2},6^d)$ & $10$ & $(5,5)$ \\
  2 & $(N,3^a,4^b,5^{c-1},6^{d-1})$ & $11$ & $(5,6)$\\
  3 & $(N,3^a,4^b,5^{c-1},6^{d-1})$ & $11$ & $(5,6)$ \\
  4 & $(N,3^{a-1},4^b,5^c,6^{d-1})$ & $9$ & $(3,6)$ \\
  \hline
\end{tabular}
\end{center}

So it remains to establish the existence of the packings $\P_0,\ldots,\P_3$. We first establish three useful facts.
\begin{description}[leftmargin=0mm]
    \item[(a) $\bm{\rho_3+4b_3+6d_3=|U_3|w}$.]
It follows from the definitions of $a'$, $c'$, $b_3$ and $d_3$ that $2a'+4b_3+4c'+6d_3+t=2a+4b+4c+6d+t-|U_1|w$. It follows from the definitions of $a_2$ and $c_2$ that $2a_2+4c_2=|U_2|w$. Thus, because $a_3=a'-a_2$ and $c_3=c'-c_2$, we have $2a_3+4b_3+4c_3+6d_3+t=(u-|U_1|-|U_2|)w$ by (ii).
    \item[(b) If $\bm{|U_3|=2}$, $\bm{t \in \{2,4\}}$ and $\bm{\rho_3 = 2w-2i}$, then $\bm{6 \leq c_3+t+i+1}$.]
Because $|U_3|=2$, $(u,|U_1|)=(5,3)$ by the definition of $U_2$ and thus $a \leq \frac{w}{2}-1$ by the definition of $U_1$. So from $2a_3+4c_3+t=2w-2i$ we deduce $c_3\geq \frac{w+2-t-2i}{4}$ and hence $c_3+t+i+1\geq\frac{w+3t+2i+6}{4}$. Because $w \geq 10$ and $t \geq 2$, the result follows.
    \item[(c) $\bm{\rho_3 \leq 4w}$ and, if $\bm{\rho_3 \geq 4w-6}$ and $\bm{t \in \{2,4\}}$, then $\bm{c_3 \geq 3}$.]
If $\rho_3 \leq 2w+8$, then $\rho_3 < 4w-6$ (note $w \geq 10$). If $\rho_3 > 2w+8$, then $|U_2| = u-|U_1|-4$ and $|U_3|=4$ by the definitions of $U_2$ and $U_3$. So $\rho_3+4b_3+6d_3=4w$ by (a) and hence $\rho_3 \leq 4w$. Furthermore, if we now suppose that $\rho_3 \geq 4w-6$, then $4b_3+6d_3 \leq 6$ and so $b_3+d_3 \leq 1$. Then $b+d \leq 2$, because $b \leq b_3$ and $d \leq d_3+1$. So by (v), $a'\leq 4$ and hence $a_3 \leq 4$. Because $2a_3+4c_3+t = \rho_3 \geq 4w-6$, it follows that $c_3 \geq 6$.
\end{description}

\noindent {\bf Proof that $\P_0$ exists.}
First observe that $3 \leq \lceil\frac{w}{4}\rceil \leq w$ because $w \geq 10$. We choose lengths $|E(C_1)|,\ldots,|E(C_{n})|$ with the required properties, which exist by Lemma~\ref{Lemma:3s5s} because $a_2+2c_2 \equiv 0 \mod{w}$. If $|U_3|=2$ or \eqref{Eq:xBoundsSmalltLeave} holds, then $a_3+c_3+m-t \in \{0\} \cup \{3,\ldots,w\}$ by Lemma~\ref{Lemma:SmalltLeave} with $(a,b,c,d)=(a_3,b_3,c_3,d_3)$ and $U'=U_3$ (the hypotheses are satisfied by (iii), (a), (b), and because either $\rho_3 \geq 10$ or $a_3+c_3 \geq 3$) and we let $|E(C^\dag_1)|=a_3+c_3+m-t$ and $|E(C^\dag_{2})|=0$. If $|U_3| \geq 4$ and \eqref{Eq:xBoundsSmalltLeave2} holds, then we let $|E(C^\dag_1)|$ and $|E(C^\dag_{2})|$ be the cycle lengths given by Lemma~\ref{Lemma:SmalltLeave2} with $(a,b,c,d)=(a_3,b_3,c_3,d_3)$ and $U'=U_3$ (the hypotheses are satisfied by (a) and (c)). Then, by Theorem~\ref{Theorem:Alspach}, a packing with the required properties exists by (iii) and because
\begin{align*}
{}&\textstyle\sum (N \setminus (m)) + |E(C^{\star})| + |E(C_1^{\dag})| +|E(C_2^{\dag})|+ |E(C_1)| + \cdots + |E(C_{n})|  \\
={}& (\tbinom{w}{2}+t-m-a-c) + |E(C^{\star})| + (m-t+a_3+c_3) + (a_2+c_2)\\
={}& \tbinom{w}{2}-(a+c)+(a'+c') + |E(C^{\star})| \\
={}& \tbinom{w}{2}-\tfrac{w}{2}.
\end{align*}
The first equality holds by (i) and the definitions of $C_1,\ldots,C_{n},C_1^\dag,C_2^\dag$. The second equality holds by the definitions of $a_3$ and $c_3$. The final equality holds because it follows from the definitions of $a'$, $c'$ and $C^{\star}$ that $a+c-a'-c'=\frac{w}{2}+|E(C^{\star})|$.

\noindent {\bf Proof that $\P_2$ exists.} This follows immediately by Lemma~\ref{Lemma:3s5s} because $|E(C_i)| \in\{\frac{w}{2},\ldots, w\}$ for $1\leq i\leq n$ and $\sum_{i=1}^n|E(C_{i})|=a_2+c_2$.

\noindent {\bf Proof that $\P_3$ exists.}
We established above that if $|U_3|=2$ or \eqref{Eq:xBoundsSmalltLeave} holds, then $C^\dag_{2}$ is trivial and we can apply  Lemma~\ref{Lemma:SmalltLeave} with $(a,b,c,d)=(a_3,b_3,c_3,d_3)$, $U'=U_3$ and $C=C^\dag_1$. Also, we established that if $|U_3| \geq 4$ and \eqref{Eq:xBoundsSmalltLeave2} holds, then we can apply  Lemma~\ref{Lemma:SmalltLeave2} with $(a,b,c,d)=(a_3,b_3,c_3,d_3)$, $U'=U_3$ and $(C_1,C_2)=(C^\dag_1,C^\dag_{2})$. Let $\P'_3$ be the packing produced by applying the appropriate lemma. In Cases 3, 4 and 5, $\P'_3$ is itself a packing with the required properties. In Case 1, we can obtain a packing with the required properties by removing a $3$-cycle and a $4$-cycle from $\P'_3$ (note that $a_3 \geq 1$ and $b_3 \geq 1$ in this case). In Case 2, we can obtain a packing with the required properties by removing two $4$-cycles from $\P'_3$ (note that $b_3 \geq 2$ in this case). In Cases 1 and 2 we will ensure that $L_1 \cup L_3$ has a vertex of degree $4$ when we construct $\P_1$ below.

\noindent {\bf Proof that $\P_1$ exists.}
If $a \geq \frac{w}{2}$, then $(a_1,c_1)=(\frac{w}{2},0)$, $C^\star$ is trivial and a packing with the required properties clearly exists. So we may assume that $a < \frac{w}{2}$. By Lemma~\ref{Lemma:1factor5cycles} there is a packing $\P'_1$ of $K_{U_1,W}\cup I \cup C^\star$ with $\lfloor\frac{3w}{4}\rfloor$ $5$-cycles with a reduced leave $L'_1$ such that $L'_1$ is trivial when $w \equiv 0 \mod{4}$, $L'_1$ is a $3$-cycle when $w \equiv 2 \mod{4}$, and $L'_1$ shares a vertex with $L_3$ when $w \equiv 2 \mod{4}$ and $a = 0$. If $w \equiv 0 \mod{4}$ or $a=0$, then $\P'_1$ is a packing with the required properties.  If $1 \leq a < \frac{w}{2}$ and $w \equiv 2 \mod{4}$, then $\P'_1 \cup \{L_1\}$ is a packing with the required properties.
\end{proof}

\section{Proof of Theorem~\ref{Theorem:MainTheorem}}\label{Section:MainTheorem}

This section contains the proof of Theorem~\ref{Theorem:MainTheorem}. Lemma~\ref{Lemma:FewOdd_general} dispenses with the case where the sum of odd entries in $(m_1,\ldots,m_\t)$ is small. In this case we can obtain the required decomposition using known cycle decomposition results for the complete graph and the complete bipartite graph. The remaining cases of Theorem~\ref{Theorem:MainTheorem} are proved by repeatedly applying Lemma~\ref{Lemma:GeneralJoining} to base decompositions given by Lemmas~\ref{Lemma:BaseDecomp_FewCrossOdds}, \ref{Lemma:BaseDecomp_ManyCrossOdds_large_m} and \ref{Lemma:BaseDecomp_ManyCrossOdds_small_m}.

\begin{lemma}\label{Lemma:FewOdd_general}
Let $u \geq 5$ and $w \geq 4$ be integers such that $u$ is odd and $w$ is even, and let $m_1,\ldots,m_\tau$ be a nondecreasing list such that the following hold
\begin{itemize}
	\item[(i)]
$m_1\geq 3$ and $m_\tau\leq \min (u,w)$;
	\item[(ii)]
$m_1+\cdots+m_\tau=\binom{u+w}{2}-\binom{u}{2}$; and
	\item[(iii)]
the sum of odd entries in $m_1,\ldots,m_\tau$ is at most $\frac{w(w-2)}{2}$.
\end{itemize}
Then there exists an $(m_1,\ldots,m_\tau)$-decomposition of $K_{u+w}-K_u$.
\end{lemma}

\begin{proof}
Let $U$ and $W$ be disjoint sets of size $u$ and $w$ respectively, let $U_1\subseteq U$ such that $|U_1|=1$, and let $M=m_1,\ldots,m_\tau$. We will form an $(M)$-decomposition of $K_{U\cup W}-K_U$ from a packing $\P_0$ of $K_{W\cup U_1}$ and a packing $\P_1$ of $K_{U\setminus U_1,W}$.

Let $n_1,\ldots,n_s$ where $n_1 \leq \cdots \leq n_s$ be the sublist of $M$ containing all of its even entries. Note that $n_1+\cdots+n_s \geq \binom{u+w}{2}-\binom{u}{2}-\frac{w(w-2)}{2}= uw+\frac{w}{2}$ by (ii) and (iii).
Let $s'$ be the largest element of $\{1,\ldots,s\}$ such that $n_{s'} \leq 3n_{s'-1}$. Observe that $n_1+\cdots+n_{s'} > w(u-1)$ because \[n_{s'+1}+\cdots+n_{s} < \medop\sum_{i=0}^{\infty} \mfrac{n_s}{3^i} \leq \medop\sum_{i=0}^{\infty} \mfrac{w}{3^i} < \mfrac{3w}{2}\]
where the first inequality follows because $n_{i} < \frac{1}{3}n_{i+1}$ for each $i \in \{s',\ldots,s-1\}$, and the second inequality follows because $n_s \leq w$ by (i).

We now define a sublist $M_1$ of $(n_1,\ldots,n_{s'})$ as follows. Begin with $M_1$ empty. Iteratively apply the following procedure: while there is an entry $x$ of $(n_1,\ldots,n_{s'}) \setminus M_1$ such that $\sum M_1 + x \leq w(u-1)$, add the largest such entry to $M_1$. When no such entry exists, terminate the procedure and fix $M_1$. Let $M_0=M \setminus M_1$, let $t$ be the integer such that $\sum M_1+t=w(u-1)$. Because $n_1+\cdots+n_{s'} > w(u-1)$, this procedure will terminate and $M_1$ will be a proper sublist of $(n_1,\ldots,n_{s'})$. Thus, the smallest even entry in $M_0$ is $n_{s''}$ for some $s'' \in \{1,\ldots,s'\}$. Also, $t$ is even because $\sum M_1$ and $w(u-1)$ are. We establish three more useful facts.

\begin{description}[leftmargin=0mm]
    \item[(a) $\bm{\sum M_1=w(u-1)-t}$ and $\bm{\sum M_0={w+1 \choose 2}+t$}.]
The former follows from the definition of $t$ and the latter follows from the former by the definition of $M_0$ and by (ii).
    \item[(b) $\bm{t \leq n_{s''}-2}$ and there are at least two even entries in $\bm{M_0}$.]
If $t$ were at least $n_{s''}$, then another even entry of $(n_1,\ldots,n_{s'})\setminus M_1$ would have been added to $M_1$ before the procedure terminated. So $t \leq n_{s''}-2$. Because $n_1+\cdots+n_{s} \geq  uw+\frac{w}{2}$, the even entries in $M_0$ sum to at least $\frac{3w}{2}+t$ and hence there are at least two by (i).
    \item[(c) $\bm{(n_{s'-u+1},\ldots,n_{s'})}$ is a sublist of $\bm{M_1}$ and $\bm{t \leq w-4}$.]
Because $n_{s'} \leq w$ by (i), the first $u-1$ entries added to $M_1$ are $n_{s'},n_{s'-1},\ldots,n_{s'-u+1}$. Thus, if $t=w-2$, then $n_{s''}=w$ by (b) and it would follow that $M_1=(w^{u-1})$ and $t=0$.
\end{description}

If $t=0$, then an $(M)$-decomposition of $K_{U\cup W}-K_U$ is given by $\P_0\cup \P_1$, where $\P_0$ is an $(M_0)$-decomposition of $K_{W\cup U_1}$, and $\P_1$ is an $(M_1)$-decomposition of $K_{U\setminus U_1,W}$. Noting (a), (c) and (i), we see that $\P_0$ exists by Theorem~\ref{Theorem:Alspach} and $\P_1$ exists by Theorem~\ref{Theorem:Bipartite}. Thus we can assume that $t \in \{2,4,\ldots,w-4\}$.

We now define integers $p$, $p^{\dag}$, $b$ and $b^{\dag}$ and (possibly empty) lists $M_0'$ and $M_1'$. We will then show that there exists an $(M_0\setminus M_0')$-packing $\P_0$ of $K_{W\cup U_1}$ whose reduced leave $L_0$ is the edge-disjoint union of a $p$-path and a $p^{\dag}$-path, an $(M_1 \setminus M_1')$-packing $\P_1$ of $K_{U\setminus U_1,W}$ whose reduced leave $L_1$ is the edge-disjoint union of a $b$-path and a $b^{\dag}$-path such that there exists an $(M_0',M_1')$-decomposition $\P_2$ of $L_0 \cup L_1$. This will suffice to complete the proof as $\P_0 \cup \P_1 \cup \P_2$ will be an $(M)$-decomposition of $K_{U\cup W}-K_U$.
\begin{itemize}
    \item
If there is an entry $q$ in $M_0$ that is at least $t+3$, then let $M_0'=(q)$ and let $M_1'=(r)$ where $r$ is the smallest entry of $M_1$. Let $b=t+2$, $b^{\dag}=r-2$, $p=q-t-2$ and $p^{\dag}=2$.
    \item
If $t \geq 4$ and each entry in $M_0$ is at most $t+2$, then $M_0$ contains at least two entries equal to $t+2$ by (b). Let $M_0'=(t+2,t+2)$, and let $M_1'$ be empty. Let $b=2$, $b^{\dag}=t-2$, $p=t$ and $p^{\dag}=4$.
    \item
If $t = 2$ and each entry in $M_0$ is at most $4$, then $M_0$ contains at least two entries equal to $4$ by (b). Let $M_0'=(4,4)$ and let $M_1'=(r)$ where $r$ is the smallest entry of $M_1$. Let $b=4$, $b^{\dag}=r-2$, $p=4$, and $p^{\dag}=2$.
\end{itemize}
In each case note that $p+p^{\dag}=\sum M_0'-t$ and $b+b^{\dag} = \sum M_1'+t$. Hence $\sum(M_0\setminus M_0')+p+p^{\dag}=\binom{w+1}{2}$ and $\sum(M_1\setminus M_1')+b+b^{\dag}=w(u-1)$ by (a).

\noindent{\bf Proof that $\P_1$ exists.} Using (a), (b), (c) and (i), it can be checked that by Lemma~\ref{Lemma:Leave_PathDecomp} there is an $(M_1\setminus M_1')$-packing of $K_{U\setminus U_1,W}$ whose reduced leave has a decomposition into a $b$-path $B$ and a $b^{\dag}$-path $B^{\dag}$ with end vertices $x$ and $y$ in $W$ (apply Lemma~\ref{Lemma:Leave_PathDecomp}(ii) with $m_i=b+b^{\dag}$ if $2\in\{b,b^{\dag}\}$ and Lemma~\ref{Lemma:Leave_PathDecomp}(i) with $m_i=b$ and $m_j=b^{\dag}$ otherwise).

\noindent{\bf Proof that $\P_0$ exists.}
First suppose that $M_0' \neq (4,4)$. Using (a), (c) and (i), there is an $(M_0\setminus M_0',p+p^{\dag})$-decomposition of $K_{W\cup U_1}$ by Theorem~\ref{Theorem:Alspach} (in each case $3\leq p+p^{\dag}\leq w+1$). Let $\P_0$ be the result of removing a $p+p^{\dag}$ cycle from this decomposition and permuting vertices so that the reduced leave of the resulting packing is the edge-disjoint union of paths $P$ and $P^{\dag}$ with end vertices $x$ and $y$ such that $V(P) \cap V(B) = V(P^{\dag}) \cap V(B^{\dag})=\{x,y\}$. This relabelling is possible provided that $|V(B) \cap V(B^{\dag}) \cap W|+p+p^{\dag}-2$, $|V(B) \cap W|+p-1$ and $|V(B^{\dag}) \cap W|+p^{\dag}-1$ are each at most $w+1$ (for a proof of this, see \cite[Lemma 5.2]{Horsley12}). These inequalities can be checked using (c) and the facts that $|V(B) \cap V(B^{\dag}) \cap W| \leq |V(B) \cap W| = \frac{b+2}{2}$ and $|V(B^{\dag}) \cap W| = \frac{b^{\dag}+2}{2}$.

Now suppose that $M_0' = (4,4)$. We form $\P_0$ as above, except that we permute vertices so that $V(P) \cap V(B) = \{x,y,z\}$ where $V(B) \cap W = \{x,y,z\}$ and $z$ is not adjacent to $x$ or $y$ in $P$.

In each case the properties of $B$, $B^{\dag}$, $P$ and $P^{\dag}$ ensure that there is an $(M_0',M_1')$-decomposition $\P_2$ of $L_0 \cup L_1$.
\end{proof}

We introduce some more notation. For a list $M$ and an integer $m$ we let $\nu_m(M)$ denote the number of entries of $M$ that are equal to $m$. For a list $M$, let $\nu(M)$ denote the total number of entries of $M$ and let $\no(M)$ denote the number of odd entries of $M$. For a nondecreasing list $M=(m_1,\ldots,m_s)$, we say that $M$ is \emph{well-behaved} if $m_{s} \leq 3m_{s-1}$.

We say that a list $R$ is a \emph{refinement} of an integer $m \geq 3$ if $\sum R=m$, each entry of $R$ is at least 3 and at most one entry of $R$ is odd. For any integer $m \geq 3$ the list $(m)$ is a refinement of $m$. We say that a list $R$ is a refinement of a list $M=(m_1,\ldots,m_s)$ if $R$ can be reordered as $(R_1,\ldots,R_s)$ where $R_i$ is a refinement of $m_i$ for each $i \in \{1,\ldots,s\}$. The fact that $\no(R)=\no(M)$ is crucial and we will use it frequently. The \emph{basic refinement} of an integer $m \geq 3$ is $R$ where
\[
R =
\left\{
  \begin{array}{ll}
    (4^{m/4}), & \hbox{if $m \equiv 0 \mod{4}$;} \\
    (3,4^{(m-9)/4},6), & \hbox{if $m \equiv 1 \mod{4}$ and $m \geq 9$;} \\
    (4^{(m-6)/4},6), & \hbox{if $m \equiv 2 \mod{4}$;} \\
    (3,4^{(m-3)/4}), & \hbox{if $m \equiv 3 \mod{4}$;} \\
    (5), & \hbox{if $m=5$.}
  \end{array}
\right.
\]
We say that a list $R$ is the basic refinement of a list $M=(m_1,\ldots,m_s)$ if $R$ can be reordered as $(R_1,\ldots,R_s)$ where $R_i$ is the basic refinement of $m_i$ for each $i \in \{1,\ldots,s\}$.

Lemma~\ref{Lemma:JoinEmAll} shows how Lemma~\ref{Lemma:GeneralJoining} can be repeatedly applied to our base decompositions to obtain the decompositions required by Theorem \ref{Theorem:MainTheorem}.

\begin{lemma}\label{Lemma:JoinEmAll}
Let $u \geq 5$ and $w \geq 4$ be integers, let $N$ and $Z=(z_1,\ldots,z_q)$ be nondecreasing lists of integers such that $z_q\leq \min (u,w,3z_{q-1})$, and let $R$ be a refinement of $Z$. If there exists an $(N,R)$-decomposition of $K_{u+w}-K_u$ that includes cycles with lengths $R$ that each contain at most one pure edge, then there exists an $(N,Z)$-decomposition of $K_{u+w}-K_u$.
\end{lemma}

\begin{proof}
Assume that there exists such an $(N,R)$-decomposition $\D$ of $K_{u+w}-K_u$. Let $\ell$ be the number of entries in $R$. Note that $\ell \geq q$, and that if $\ell=q$, then $R=Z$ and the result is obviously true. So suppose that $\ell > q$. By induction, it suffices to show that there is an $(N,R')$-decomposition $\D'$ of $K_{u+w}-K_u$ where $R'$ is a refinement of $Z$ with $\ell-1$ entries and $\D'$ contains cycles with lengths $R'$ that each contain at most one pure edge. Let $R_1,R_2,\ldots,R_{q}$ be a reordering of $R$ so that $R_i$ is a refinement of $z_i$ for $i\in\{1,\ldots,q\}$.

\noindent \textbf{Case 1.} Suppose that there is exactly one list $R_i$ in $R_1,\ldots,R_q$ such that $\nu(R_i)\geq 2$. Let $R_i=a_1,\dots,a_{\nu(R_i)}$ and let $j=q$ if $i \neq q$ and $j=q-1$ if $i=q$.
Let $C_1$, $C_2$ and $C_3$ be cycles in $\D$ of lengths $a_1$, $a_2$ and $z_j$ that each contain at most one pure edge. We can obtain a decomposition $\D'$ with the required properties by applying Lemma~\ref{Lemma:GeneralJoining} to $\D \setminus \{C_1,C_2,C_3\}$ with $h=z_j$, $m_1=a_1$ and $m_2=a_2$. We have $m_1+m_2+h \leq z_i+z_j \leq 2\min(u,w)$ from our hypotheses. We have $m_1+m_2\leq 3h$ because either $a_1+a_2\leq z_i\leq z_q=h$ or $(i,j)=(q,q-1)$, in which case $a_1+a_2\leq z_q \leq 3z_{q-1}$ by our hypotheses.

\noindent \textbf{Case 2.} Suppose that there are at least two lists in $R_1,\ldots,R_q$ that each have at least two entries. Let $r$ be the largest entry in $R_1,\ldots,R_q$ and let $i\in \{1,\ldots,q\}$ such that $r$ is an entry of $R_i$. Let $j$ be an element of $\{1,\ldots,q\} \setminus \{i\}$ such that $\nu(R_j)\geq 2$ and let $R_j=a_1,\ldots,a_{\nu(R_j)}$. Let $C_1$, $C_2$ and $C_3$ be cycles in $\D$ of lengths $a_1$, $a_2$ and $r$ that each contain at most one pure edge. We can obtain a decomposition $\D'$ with the required properties by applying Lemma~\ref{Lemma:GeneralJoining} to $\D \setminus \{C_1,C_2,C_3\}$ with $h=r$, $m_1=a_1$ and $m_2=a_2$. We have $m_1+m_2+h \leq z_i+z_j \leq 2\min(u,w)$ from our hypotheses. We have $m_1+m_2\leq 3h$ because $a_1,a_2\leq r$.
\end{proof}

We are now ready to prove our main result.

\begin{proof}[{\bf Proof of Theorem~\ref{Theorem:MainTheorem}}]
If there exists an $(m_1,\dots,m_\tau)$-decomposition of $K_{u+w}-K_u$ then (i)--(iv) hold by Lemma~\ref{Lemma:NecessaryConditions}. So it remains to show, for any integers $u \geq 1$ and $w \geq 10$ and list $M=(m_1,\dots,m_\tau)$, that if $m_\tau\leq \min (u,w,3m_{\tau-1})$ and (i)--(iv) hold, then there exists an $(M)$-decomposition of $K_{u+w}-K_u$.

If $m_1=m_2=\cdots=m_\tau$ then the result follows by \cite[Theorem 1.2]{Horsley12} ($m_i$ even) or \cite[Theorem 3]{DWcurrent} ($m_i$ odd). If $u=1$, there is an $(M)$-decomposition of $K_{w+1}$ by Theorem~\ref{Theorem:Alspach} and $K_{w+1}=K_{w+1}-K_1$. If $u=3$, there is an $(M,3)$-decomposition of $K_{w+3}$ by Theorem~\ref{Theorem:Alspach} and deleting the edges of a $3$-cycle produces an $(M)$-decomposition of $K_{w+3}-K_3$.
If the sum of odd entries in $M$ is at most $\frac{w(w-2)}{2}$, then the result follows by Lemma~\ref{Lemma:FewOdd_general}. Thus we can suppose that $m_1<m_\tau$, $u\geq 5$, and the sum of odd entries in $M$ is greater than $\frac{w(w-2)}{2}$.

We will proceed as follows. First we choose a sublist $Z$ of $M$ such that $Z$ is well-behaved. Then we define a refinement $R=(3^a,4^b,5^c,6^d,k)$ of $Z$ such that $a$, $b$, $c$, $d$ and $M\setminus Z$ satisfy the hypotheses of Lemma~\ref{Lemma:BaseDecomp_FewCrossOdds}, \ref{Lemma:BaseDecomp_ManyCrossOdds_large_m} or  \ref{Lemma:BaseDecomp_ManyCrossOdds_small_m} ($R$ is not always the basic refinement of $Z$ but it is always `close' to it). The appropriate lemma will then yield an $(M\setminus Z,R)$-decomposition $\D$ of $K_{u+w}-K_u$ that contains cycles with lengths $R$ that each contain at most one pure edge. Applying Lemma~\ref{Lemma:JoinEmAll} will then produce an $(M)$-decomposition of $K_{u+w}-K_u$. So it remains to define $Z$ and $R$, and to show that the hypotheses of Lemma~\ref{Lemma:BaseDecomp_FewCrossOdds}, \ref{Lemma:BaseDecomp_ManyCrossOdds_large_m} or  \ref{Lemma:BaseDecomp_ManyCrossOdds_small_m} are satisfied.

Throughout this proof we employ some notational shorthand concerning lists. For a list $X$, a set $S$ and an integer $x$, we write $x \in X$ if at least one entry of $X$ is equal to $x$, $X \subseteq S$ if each entry of $X$ is an element of $S$, and $\max_{\rm e}(X)$ for the largest even entry of $X$. For a sublist $X=x_1,\ldots,x_s$ of $M$, we define $\sum_{\rm e}X=\sum_{i=1}^s\efloor{x_i}$ and $t(X)=uw-\sum_{\rm e}X$. Note that $\sum_{\rm e}X$ can also be written as $\sum X-\nu_{\rm o}(X)$. Then $\sum_{\rm e}X=uw-t(X)$ and, by (iii), $\sum(M\setminus X)= \binom{w}{2}-\no(X)+t(X)$. Clearly $t(X)$ is always even. We abbreviate $t(Z)$ to $t$.

The proof splits into three cases, depending on $\no(M)$ and $\nu_5(M)$.

\noindent {\bf Case 1.}  Suppose that $\no(M)-\nu_5(M)\geq \frac{w}{2}+3$. We aim to satisfy the hypotheses of either Lemma \ref{Lemma:BaseDecomp_ManyCrossOdds_large_m} (in Case 1a) or Lemma \ref{Lemma:BaseDecomp_ManyCrossOdds_small_m} (in Case 1b). We choose $Z$ according to the following procedure.

\begin{itemize}
    \item[1.]
Let $Z''$ be the list consisting of the $\frac{w}{2}+3$ largest odd entries of $M$ that are not equal to $5$.
    \item[]
Each entry in $M$ is at most $\min(u,w) \leq u$, so $\sum_{\rm e} Z'' \leq (u-1)(\frac{w}{2}+3)$ and hence $t(Z'') \geq  \frac{(u+1)(w-6)}{2} +6 \geq 2u+8$. Below, this will imply that $\nu(Z'\setminus Z'')\geq 2$.
    \item[2.]
Begin with $Z'=Z''$ and repeatedly add the largest entry of $M\setminus Z'$ to $Z'$, until $M \setminus Z'$ is empty or $Z'$ satisfies $t(Z') \leq \max(M\setminus Z')-2$.
    \item[]
It follows from this definition that $t(Z') \geq 0$ (note that $t(Z')$ is even). Because $\sum(M\setminus Z')= \binom{w}{2}-\no(Z')+t(Z')$, it follows from (iv) that $t(Z') = 0$ if $M\setminus Z'$ is empty. If $t(Z') \geq 2$, then $t(Z') \leq \max(M\setminus Z')-2 \leq w-2$ and each entry in $Z' \setminus Z''$ is at least $\max(M \setminus Z')$.
    \item[3.1]
If $t(Z') \neq 2$, then let $Z=Z'$ and let $m=0$ if $t=0$ and $m=\max(M\setminus Z')$ if $t>0$. Note that $t=t(Z')$.
    \item[3.2]
If $t(Z') = 2$ and $M\setminus Z' \nsubseteq \{3,w-1,w\}$, then let $Z=Z'$ and let $m$ be the largest entry in $M\setminus Z'$ such that $4 \leq m \leq w-2$. Note that $t=2$.
    \item[3.3]
If $t(Z') = 2$, $M\setminus Z' \subseteq \{3,w-1,w\}$ and $3 \in M\setminus Z'$, then let $Z=(Z',3)$ and let $m=0$. Note that $t=0$.
    \item[3.4]
If $t(Z') = 2$, $M\setminus Z' \subseteq \{w-1,w\}$ and $w \in M\setminus Z'$, then let $Z=Z'\setminus(\min(Z''))$, and let $m=w$. Note that $t=2+\efloor{\min(Z'')}$.
    \item[]
We have $\min(Z'') \leq w-3$ and hence $t \leq w-2$. Otherwise, because $Z' \setminus Z'' \subseteq \{w\}$, $Z'=(w^i,(w-1)^{w/2+3})$ for some $i$ and, by the definition of $t(Z')$, $t(Z') \equiv 6 \mod{w}$. This contradicts $t(Z') = 2$.
    \item[3.5]
If $t(Z') = 2$ and $M\setminus Z'\subseteq\{w-1\}$, then let $Z=(Z'\setminus (w),w-1)$, and let $m=w-1$. Note that $t=4$.
    \item[]
There is a $w$ in $Z'$ for otherwise $Z'\subseteq \{w-1\}$ and hence $M\subseteq \{w-1\}$ which contradicts $m_1<m_\tau$. Further, because $t(Z') = 2$ and $Z'=(w^i,(w-1)^j)$ for some $i$ and $j$, we have $j \equiv 1 \mod{\tfrac{w}{2}}$ and hence can deduce from $\sum(M\setminus Z')= \binom{w}{2}-\no(Z')+t(Z')$ that $M\setminus Z' = ((w-1)^h)$ for some $h \equiv -1 \mod{\tfrac{w}{2}}$. Thus $h \geq 2$ and so $m \in M\setminus Z$.
\end{itemize}

We first show that $Z$ is well-behaved. For $i \in \{\tau-1,\tau\}$, because $\nu(Z' \setminus Z'') \geq 2$, if $m_i$ is not added to $Z''$ in step 1 then it is added to $Z'$ in step 2. So unless our procedure terminates at step 3.5, $(m_{\tau-1},m_\tau)$ is a sublist of $Z$. If the procedure terminates at step 3.5, then $Z\subseteq \{w-1,w\}$, and it follows that $Z$ is well-behaved.

Let $k=\eceil{\frac{t+2}{3}}$ if $t\geq 12$ and $k=0$ otherwise. We now note some important properties that hold for any refinement $(3^a,4^b,5^c,6^d,k)$ of $Z$.
\begin{itemize}
    \item[(a)]
$2a+4b+4c+6d+k+t=uw$ and $\sum(M\setminus Z)+a+c=\binom{w}{2}+t$. These properties follow because $\sum_{\rm e}Z=uw-t$ and $\sum(M\setminus Z)= \binom{w}{2}-\no(Z)+t$.
    \item[(b)]
Either $(m,t)=(0,0)$, or $t \geq 2$ and $t+2 \leq m \leq w$. This is easy to check in each case.
    \item[(c)]
$a+c \geq \frac{w}{2}+2$ and, if $m < w$, $a+c \geq \frac{w}{2}+3$. This follows because $a+c=\no(Z)$, $\no(Z'') = \frac{w}{2}+3$ and either $Z''$ is a sublist of $Z$ or $Z'' \setminus (\min(Z''))$ is a sublist of $Z$ (the latter occurs only in case 3.4 when $m=w$).
    \item[(d)] $m \in M \setminus Z$ if $m > 0$, and $\min(Z \setminus Y) \geq m$, where $Y=Z''\setminus (\min(Z''))$ if the procedure terminates at step 3.4 and $Y= Z''$ otherwise.
This is clear if $m=0$, so we may suppose that $m > 0$ and, by (b), that $t \geq 2$. It is easy to check in each case that $m\in M\setminus Z$ and also that $m \in M \setminus Z'$, and it follows that $\max(M \setminus Z') \geq m$.
We noted $\min (Z'\setminus Z'') \geq \max(M \setminus Z')$ after step 2. If the procedure terminated at a step other than 3.5, then  $\min(Z\setminus Y)\geq \min(Z'\setminus Z'')$ and the statement holds (note it did not terminate at step 3.3 because $m >0$). If the procedure terminated at step 3.5, then $Z\setminus Y\subseteq \{w-1,w\}$ and $m=w-1$, so the statement holds.
\end{itemize}

\noindent {\bf Case 1a.} Suppose that $m \geq 7$. Then $t \geq 2$ by (b). Let $x=\max(Z)$. Let $(3^a,4^b,5^c,6^d)$ be the basic refinement of $(Z\setminus(x),x-k)$ if $x-k \neq 9$ and the basic refinement of $(Z\setminus(x),4,5)$ if $x-k = 9$. Note that, by (d) and (b), $x \geq m \geq t+2$. Thus $x-k=x \geq 7$ if $t \leq 10$ and, if $t \geq 12$, $x-k \geq t+2-\eceil{\frac{t+2}{3}} \geq 8$ . It can be seen that $a$, $b$, $c$, $d$ and $M\setminus Z$ satisfy the conditions of Lemma~\ref{Lemma:BaseDecomp_ManyCrossOdds_large_m} using (a) -- (d) and the following facts.
\begin{itemize}
    \item
$c \in \{0,1\}$ and $a=\no(Z)-c \geq \frac{w}{2}+1$. We have $c=0$ if $x-k \neq 9$ and $c=1$ if $x-k = 9$ because $5 \notin Y$ by the definition of $Y$, $5 \notin Z\setminus Y$ by (d), and $x-k \neq 5$. Then $a=\no(Z)-c \geq \frac{w}{2}+1$ using (c).
    \item
$b \geq 1$. This is obvious if $x-k = 9$. If $x-k \neq 9$, the basic refinement of $x-k$ contains a $4$ (recall that $x-k \geq 7$).
    \item
Either $a \leq \frac{w}{2}+3$ or $uw \geq (a+c)\efloor{m}$. The former holds if $\min(Z'') < m$ because then each odd entry in $M \setminus Z''$ is less than $m$ by the definition of $Z''$ and so every entry of $Z\setminus Y$ is even by (d). The latter holds if $\min(Y) \geq m$ because then $\min(Z) \geq m$ by (d) and so $(a+c)\efloor{m}=\no(Z)\efloor{m} \leq \sum_{\rm e} Z \leq uw-t$.
    \item
$(m,t) \neq (w,2)$ and, if $a \geq \frac{w}{2}+4$, then $(m,t) \notin \{(w-1,2),(w,4)\}$. Clearly $(m,t) \neq (w,2)$ and $(m,t) \neq (w-1,2)$ because if $t=2$, then the procedure terminated at step 3.2 and $m \leq w-2$. If $(m,t) = (w,4)$, then $Z\setminus Y \subseteq \{w\}$ by (d) and hence $a \leq \no(Z) = \no(Y) \leq \frac{w}{2}+3$.
\end{itemize}

\noindent {\bf Case 1b.} Suppose that $m\leq 6$. Then $t \in \{0,2,4\}$ by (b) and $k=0$. Let $(3^a,4^b,5^c,6^d)$ be the basic refinement of $Z$. It can be seen that $a$, $b$, $c$, $d$ and $M\setminus Z$ satisfy the conditions of Lemma~\ref{Lemma:BaseDecomp_ManyCrossOdds_small_m} using (a) -- (d) and the following facts.
\begin{itemize}
    \item
If $u=5$ then $m_\tau\leq 5$ by our hypotheses so $d=0$.
    \item

$a \geq \frac{w}{2}+2$ because $\no(Y)\geq \frac{w}{2}+2$, $5\notin Y$, and $Y$ is a sublist of $Z$.
    \item
If $b+d\leq 2$ and $t \in\{2,4\}$, then $a \leq \frac{w}{2}+3$. Because $b+d\leq 2$, $Z\setminus (y_1,y_2)\subseteq\{3,5\}$ for some $y_1,y_2 \in Z$ (note that the basic refinement of any integer in $\{3,\ldots,w\} \setminus \{3,5\}$ contains a $4$ or a $6$). For $i \in \{1,2\}$, if $y_i \geq 7$ and $y_i$ is odd, then $y_i \in Y$ and $y_i \notin Z \setminus Y$. Because $t \geq 2$, $3 \notin Z \setminus Y$, using (d) and (b). Thus any odd entries in $Z \setminus Y$ are $5$s and $a \leq \frac{w}{2}+3$.
\end{itemize}

\noindent {\bf Case 2.} Suppose that $\no(M)-\nu_5(M) < \frac{w}{2}+3$, $\nu_5(M)\geq w$ and $\no(M)\geq w+4$. We aim to satisfy the hypotheses of Lemma \ref{Lemma:BaseDecomp_ManyCrossOdds_small_m}. We choose $Z$ according to the following procedure.
\begin{itemize}
    \item[1.]
Begin with $Z''=(5^{\lceil3w/4\rceil})$ and add the four largest odd entries of $M\setminus (5^w)$ to $Z''$. Note that $\nu_5(M\setminus Z'') \geq \lfloor\frac{w}{4}\rfloor$ because $\nu_5(M)\geq w$.
    \item[]
Because each odd entry in $M$ is at most $\min(u,w-1) \leq u$, $\sum_{\rm e}Z'' \leq 3w+2+4(u-1)$ and hence $t(Z'') \geq (u-3)(w-4)-10$. This implies that $t(Z'') > 0$ (note that $u \geq 5$ and $w \geq 10$).
    \item[2.]
Begin with $Z'=Z''$ and repeatedly add the largest entry of $M\setminus Z'$ to $Z'$, until $M\setminus Z'$ is empty or $Z'$ satisfies $t(Z') \leq \max(M\setminus Z')-2$.
    \item[]
It follows from this definition that $t(Z') \geq 0$ (note that $t(Z')$ is even). Because $\sum(M\setminus Z')= \binom{w}{2}-\no(Z')+t(Z')$, it follows from (iv) that $t(Z') = 0$ if $M\setminus Z'$ is empty. If $t(Z') \geq 2$, then $t(Z') \leq \max(M\setminus Z')-2 \leq w-2$ and each entry in $Z' \setminus Z''$ is at least $\max(M \setminus Z')$.
    \item[3.1]
If $M\setminus Z'$ is empty or $\max(M\setminus Z') \leq 6$, then let $Z=Z'$. Let $m=0$ if $t=0$ and $m=\max(M\setminus Z)$ if $t>0$. Note that $t=t(Z')\leq 4$.
    \item[3.2]
If $\max(M\setminus Z') > 6$, then let $Z=(Z',5^i)$ where $i=\lfloor\frac{t(Z')}{4}\rfloor$ and let $m=0$ if $t=0$ and $m=5$ if $t=2$. Observe that $t=0$ if $t(Z') \equiv 0 \mod{4}$ and $t=2$ if $t(Z') \equiv 2 \mod{4}$.
    \item[]
To see that $Z$ is a sublist of $M$ note that
\[\nu_5(M\setminus Z') = \nu_5(M\setminus Z'') \geq \lfloor\tfrac{w}{4}\rfloor \geq \lfloor\tfrac{t(Z')}{4}\rfloor = i\]
because $\min(Z' \setminus Z'')\geq\max(M\setminus Z')>6$ and $t(Z') \leq w-2$. Furthermore, $\lfloor\tfrac{w}{4}\rfloor > i$ when $t(Z') \equiv 2 \mod{4}$ and so $5 \in M\setminus Z$ when $t=2$.
\end{itemize}

We first show that $Z$ is well-behaved. This is obvious if $u \leq 7$ and hence $m_\tau\leq 7$ by our hypotheses. If $u \geq 9$ then $t(Z'') \geq (u-3)(w-4)-10 > 2w$ and $\nu(Z'\setminus Z'')\geq 2$. Thus, for $i \in \{\tau-1,\tau\}$ if $m_i$ is not added to $Z''$ in step 1 then it is added to $Z'$ in step 2. So $(m_{\tau-1},m_\tau)$ is a sublist of $Z$ and $Z$ is well-behaved.

Let $(3^a,4^b,5^c,6^d)$ be the basic refinement of $Z$. It can be seen using arguments similar to those of Case 1b that $a$, $b$, $c$, $d$ and $M\setminus Z$ satisfy the hypotheses of Lemma~\ref{Lemma:BaseDecomp_ManyCrossOdds_small_m}. Note that $c \geq \frac{3w}{4}$ and $a+c \geq \frac{3w}{4}+4$ because $Z''$ is a sublist of $Z$. If $b+d\leq 2$ and $t \geq 2$, then $a \leq 4$ because, by arguments similar to those used in Case 1b, $(Z\setminus (y_1,y_2))\subseteq\{3,5\}$ for some $y_1,y_2 \in Z$ and the only odd entries in $Z \setminus Z''$ are 5s.

\noindent {\bf Case 3.} Suppose that $\no(M)-\nu_5(M)<\frac{w}{2}+3$ and that either $\no(M)<w+4$ or $\nu_5(M)<w$. We aim to satisfy the hypotheses of Lemma~\ref{Lemma:BaseDecomp_FewCrossOdds}. Accordingly we redefine $t(X)$ and $t$. For a sublist $X$ of $M$, we define $t(X)=(u-1)w-\sum_{\rm e}X$. Then $\sum_{\rm e}X=(u-1)w-t(X)$ and, by (iii), $\sum(M\setminus X)= \binom{w+1}{2}-\no(X)+t(X)$. Again, $t(X)$ is always even and we abbreviate $t(Z)$ to $t$. Observe that $\no(M) > \frac{w-2}{2} \geq 3$ since the sum of odd entries in $M$ is greater than $\frac{w(w-2)}{2}$. Let $\sigma$ be the sum of the $\nu_{\rm o}(M)-3$ smallest odd entries in $M$.

\noindent {\bf Case 3a.} Suppose further that $\sigma \leq \binom{w+1}{2}$. We choose $Z$ according to the following procedure.
\begin{itemize}
    \item[1.]
Let $Z''$ be a list consisting of the largest three odd entries of $M$.
    \item[]
We have $\sum_{\rm e} Z'' \leq 3(u-1)$ and hence $t(Z'') \geq (u-1)w-3(u-1) \geq 7u-7 > 5u$. Below, this will imply that $\nu(Z \setminus Z'') \geq 5$.
    \item[2.]
Begin with $Z=Z''$ and repeatedly add the largest even entry of $M\setminus Z$ to $Z$, until $M\setminus Z$ contains no even entries or until $Z$ satisfies $t \leq \max_{\rm{e}}(M\setminus Z)-2$. Let $m=0$ if $t=0$, $m=\max(M\setminus Z)$ if $t \geq 4$, and let $m$ be an entry of $M\setminus Z$ such that $4 \leq m \leq w-1$ if $t = 2$ (we show below that such an entry exists).
    \item[]
It follows from this definition that $t \geq 0$.
To show that a suitable choice of $m$ exists when $t=2$, suppose otherwise that $M\setminus Z \subseteq \{3,w\}$. Then, the sum of the odd entries in $M$ is at most $3(w-1)+3(\frac{w}{2}-1)=\frac{9w}{2}-6$, because each of the three odd entries in $Z$ is at most $w-1$, each odd entry in $M \setminus Z$ is a 3, and $\no(M)-\nu_5(M) \leq \frac{w}{2}+2$. This contradicts our assumption that the sum of the odd entries in $M$ is greater than $\frac{w(w-2)}{2}$ (note that $w \geq 10$).
\end{itemize}

For $i \in \{\tau-1,\tau\}$, because $\nu(Z \setminus Z'') \geq 2$, if $m_i$ is not added to $Z''$ in step 1 then it is added to $Z$ in step 2. Thus $(m_{\tau-1},m_{\tau})$ is a sublist of $Z$ and $Z$ is well-behaved. Let $k=\eceil{\frac{t+2}{3}}$ if $t\geq 12$ and $k=0$ otherwise. Let $(3^a,4^b,5^c,6^d)$ be the basic refinement of $(Z \setminus (m_{\tau}),m_{\tau}-k)$. If $m_\t \leq 6$, then $m_\t-k=m_\t$. If $m_\t \geq7$ then, as in Case 1a, $m_{\tau}-k \geq 7$. Using arguments similar to those in the previous cases and the following facts we can see that $a$, $b$, $c$, $d$ and $M\setminus Z$ satisfy the conditions of Lemma~\ref{Lemma:BaseDecomp_FewCrossOdds}.

\begin{itemize}
    \item
$a+c=\nu_{\rm o}(Z)=3$ and $(m,t) \neq (w,2)$. The former follows from the definition of $Z$ (note that $m_{\tau}-k$ and $m_{\tau}$ have the same parity). The latter follows from our choice of $m$.
    \item
$m \geq t+2$ if $t>0$. If $t=2$ this is obvious by our choice of $m$, so we may suppose that $t \geq 4$. We have $\sum(M \setminus Z)=\binom{w+1}{2}-\nu_{\rm o}(Z)+t > \binom{w+1}{2}$ because $\nu_{\rm o}(Z)=3$ and $t \geq 4$. However, the sum of the odd entries in $M \setminus Z$ is $\sigma$, and $\sigma \leq \binom{w+1}{2}$. Thus $M\setminus Z$ has an even entry and we have $m \geq \max_{\rm e}(M\setminus Z) \geq t+2$ by the definition of $Z$.
\end{itemize}

\noindent {\bf Case 3b.}  Suppose further that $\sigma > \binom{w+1}{2}$. We choose $Z$ according to the following procedure.
\begin{itemize}
    \item[1.]
Let $Z''$ be a sublist of $M$ with maximum sum subject to the constraints that $\nu(Z'')=\nu_{\rm o}(Z'')=6$ and $\nu_5(Z'') \leq 1$.
    \item[]
Such a sublist exists because $\nu_{\rm o}(M) \geq 6$ and $\nu_{\rm o}(M)-\nu_5(M) \geq 5$. These facts must hold because $\sigma > \binom{w+1}{2}$ and since either $\no(M) \leq w+3$ or $\nu_5(M) \leq w-1$ by the criteria for Case 3. We have $\sum_{\rm e} Z'' \leq 6(u-1)$ and hence $t(Z'') \geq (u-1)w-6(u-1)>3u$. Below, this will imply that $\nu(Z \setminus Z'') \geq 3$.
    \item[2.]	
Begin with $Z=Z''$ and repeatedly add to $Z$ the largest entry of $M\setminus Z$ not equal to $5$ until $Z$ satisfies $t\leq \max(M\setminus Z)-2$. Let $m=\max(M\setminus Z)$ if $t>0$ and let $m=0$ if $t=0$.
    \item[]	
It follows from this definition that $t \geq 0$. This process terminates with $M\setminus Z \nsubseteq \{5\}$, because $\sum (M\setminus Z) = \binom{w+1}{2}-\nu_{\rm o}(Z) +t \geq \frac{11w}{2}-(\frac{w}{2}+2)-1 \geq 5w-3$ and $5\nu_{5}(M\setminus Z) \leq 5\nu_{5}(M) \leq 5(w-1)$, using the criteria for Case 3 and the fact that $\nu_5(Z)=\nu_5(Z'') \in\{0, 1\}$. Note that $t \leq \max(M\setminus Z)-2 \leq w-2$.
\end{itemize}

We must have $m_\t \geq 7$ for otherwise, by the criteria for Case~3, $\sigma \leq 3(\frac{w}{2}+2)+5(w-4) =\frac{13w}{2}-14$, contradicting $\sigma > \binom{w+1}{2}$ (note $w \geq 10$). So, for $i \in \{\tau-1,\tau\}$, because $\nu(Z \setminus Z'') \geq 2$, if $m_i$ is not added to $Z''$ in step 1 then it is added to $Z$ in step 2 (note that if $m_{\tau-1}=5$ then $m_{\tau-1}\in Z''$). Thus $(m_{\tau-1},m_{\tau})$ is a sublist of $Z$ and $Z$ is well-behaved.

Let $k=\eceil{\frac{t+2}{3}}$ if $t\geq 12$ and $k=0$ otherwise. Let $(3^a,4^b,5^c,6^d)$ be the basic refinement of $(Z \setminus (m_{\tau}),m_{\tau}-k)$ if $m_{\tau}-k \neq 9$ and the basic refinement of $(Z \setminus (m_{\tau}),4,5)$ if $m_{\tau}-k = 9$. As in Case 1a, $m_{\tau}-k \geq 7$. Using arguments similar to those in the previous cases and the following facts we can see that $a$, $b$, $c$, $d$ and $M\setminus Z$ satisfy the conditions of Lemma~\ref{Lemma:BaseDecomp_FewCrossOdds}.
\begin{itemize}
    \item
$a+2c \leq w$. We have $a+2c \leq \nu_{\rm o}(Z)+\nu_{\rm 5}(Z)+1$ (equality occurs when $m_{\tau}-k = 9$). Thus, $a+2c \leq \frac{w}{2}+5 \leq w$ by the definition of $Z$ and because $\no(M)-\nu_5(M) \leq \frac{w}{2}+2$ and $\nu_5(Z) = \nu_5(Z'') \leq 1$.
    \item
$a+c \geq 6$ because $a+c= \no(Z) \geq \no(Z'') = 6$.
    \item
$b \geq 1$ as in Case 1a. \qedhere
\end{itemize}
\end{proof}

\vspace{0.3cm} \noindent{\bf Acknowledgements}

The first author was supported by Australian Research Council grants DE120100040 and DP150100506. The second author was supported by an Australian Postgraduate Award.

\bibliographystyle{acm}
\bibliography{references}

\end{document}